\documentclass{article}
\usepackage{times}
\usepackage{soul}
\usepackage{url}
\usepackage[hidelinks]{hyperref}
\usepackage[utf8]{inputenc}
\usepackage[small]{caption}
\usepackage{graphicx}
\usepackage{amsmath}
\usepackage{amsthm}
\usepackage{booktabs}
\usepackage{algorithm}
\usepackage{algorithmicx}
\usepackage{relsize}
\usepackage{amssymb,amsmath,amsfonts,graphicx,algpseudocode} 
\def\X{{\mathbb X}}

\newtheorem{exam}{Example}
\newtheorem{proposition}{Proposition}
\newtheorem{definition}{Definition}
\newtheorem{lemma}{Lemma}
\newtheorem{conjecture}{Conjecture}
\newtheorem{corollary}{Corollary}	
\newtheorem{question}{Question}

\newcommand{\ocap}{\textcircled{$\scriptstyle{\cap}$}}

\newcommand{\ocup}{\textcircled{$\scriptstyle{\cup}$}}

\newcommand{\oyager}{\textcircled{$\scriptstyle{Y}$}}

\newtheorem{theorem}{Theorem}
\newtheorem{axiom}{Axiom}

\title{Reasoning with random sets: An agenda for the future}

\author{
    Fabio Cuzzolin\\
    Visual Artificial Intelligence Laboratory (VAIL)\\
    Oxford Brookes University, Oxford, UK
    \\
    fabio.cuzzolin@brookes.ac.uk
}

\begin{document}

\maketitle

\begin{abstract}
In this paper, we discuss a potential agenda for future work in the theory of random sets and belief functions, touching upon a number of focal issues: the development of a fully-fledged theory of statistical reasoning with random sets, including the generalisation of logistic regression and of the classical laws of probability; the further development of the geometric approach to uncertainty, to include general random sets, a wider range of uncertainty measures and alternative geometric representations; the application of this new theory to high-impact areas such as climate change, machine learning and statistical learning theory.
\end{abstract}

\section{Introduction}

The theory of belief functions \cite{Shafer76,dempster67multivalued} is a modelling language for representing and combining elementary items of evidence, which do not necessarily come in the form of sharp statements, with the goal of maintaining a mathematical representation of an agent's beliefs about those aspects of the world which the agent is unable to predict with reasonable certainty.
While arguably a more appropriate mathematical description of uncertainty than classical probability theory, for the reasons we have thoroughly explored in \cite{cuzzolin2021springer},
the theory of evidence is relatively simple to understand and implement, and does not require one to abandon the notion of an event, as is the case, for instance, for Walley's imprecise probability theory \cite{walley00towards}.
It is grounded in the beautiful mathematics of random sets, 
and exhibits strong relationships with many other theories of uncertainty.
As mathematical objects, belief functions have fascinating properties in terms of their geometry,
algebra \cite{Zhou2012} and combinatorics. 

Despite initial concerns about the computational complexity of a naive implementation of the theory of evidence, evidential reasoning can actually be implemented on large sample spaces \cite{Reineking2014} and in situations involving the combination of
numerous pieces of evidence \cite{denoeux2021distributed}. Elementary items of evidence often induce simple belief functions, which can be combined very efficiently with complexity $O(n+1)$. We do not need to assign mass to all subsets, but {we do need to be allowed to do so} when necessary (e.g., in the case of missing data). 
Most relevantly, the most plausible hypotheses can be found without actually computing the overall belief function \cite{pan2018new}. At any rate, Monte Carlo approximations \cite{wilson1991monte} can easily be implemented when an explicit result for the combination is required. Last but not least, local propagation schemes \cite{4306979} allow for the parallelisation of belief function reasoning, similarly to what happens with Bayesian networks.

Statistical evidence can be represented in belief function theory in several ways, for example: 
\begin{itemize}
\item using likelihood-based belief functions \cite{Shafer76,Aickin2000}, in a way that generalises both likelihood-based and Bayesian inference;
\item via Dempster's inference approach, which makes use of auxiliary variables in a fiducial setting, and its natural evolution, weak belief \cite{martin2010,Zhang11weak};
\item in the framework of the generalised Bayesian theorem, proposed by Smets \cite{smets93belief,dezert2018total}.
\end{itemize}

Decision-making strategies based on intervals of expected utilities can be formulated which allow more cautious decisions than those based on traditional approaches, and are able to explain the empirical aversion to second-order uncertainty highlighted in Ellsberg's paradox \cite{Eichberger1999}.
A straightforward extension of the theory, originally formulated for finite sample spaces, to continuous domains \cite{wang92continuous} can be achieved via the Borel interval representation initially put forward by Strat \cite{Strat84} and Smets, when the analysis is restricted to intervals of real values \cite{Vannobel2012}. In the more general case of arbitrary subsets of the real domain, the theory of random sets \cite{nguyen21belief,molchanov1997statistical} is, in our view, the natural mathematical framework to adopt. Finally, an extensive array of estimation, classification and regression tools based on the theory of belief functions has already been developed \cite{ristic04ipmu,elouedi00classification}, and further contributions can be envisaged.

\subsection{Open issues}

According to Dempster and Shafer themselves \cite{Yager:2010:CWD:1951793},
although flourishing by some measures, belief function theory had still not addressed questions such as deciding whether bodies of evidence were independent, or what to do if they were dependent. 
They lamented the ongoing confusion and disagreement about how to interpret the theory, and its limited acceptance in the field of mathematical statistics, where it first began. In response, they proposed an agenda to move the theory forward, centred on the following three elements:
\begin{itemize}
\item 
a richer understanding of the uses of probability, for they believed the theory was best regarded as a way of using probability \cite{Shafer81,Dempster01informs,Shafer2008};
\item 
a deeper understanding of statistical modelling itself \cite{DEMPSTER2008365}, which would go beyond a traditional statistical analysis which begins by specifying probabilities that are supposed to be known except for certain parameters;
\item 
in-depth examples of sensible Dempster--Shafer analyses of a variety of problems of real scientific and technological importance.
\end{itemize}

As for the last point, interesting examples in which the use of belief functions provides sound and elegant solutions to real-life problems, essentially characterised by `missing' information, had been given by Smets in \cite{Smets:1999:PUB:2073796.2073865}. 
These include classification problems in which the training set is such that the classes are only partially known,
an information retrieval system handling interdocument relationships,
the combination of data from sensors providing data about partially overlapping frames, and
the determination of the number of sources in a multisensor environment. 

The correct epistemic interpretation of belief function theory certainly needs to be clarified once and for all. In \cite{cuzzolin2021springer} we have argued that belief measures should be seen as \emph{random variables for set-valued observations}. 
In an interesting, although rather overlooked, recent publication \cite{SHAFER2011127}, instead,
Dubucs and Shafer highlighted two ways of interpreting numerical degrees of belief in terms of betting, in a manner that echoes Shafer and Vovk's game-theoretical interpretation of probability \cite{Shafer07game,shafer01book}: 
(i)
you can offer to bet at the odds defined by the degrees of belief, or (ii)
you can make a judgement that a strategy for taking advantage of such betting offers will not multiply the capital it risks by a large factor.
Both interpretations, the authors argue, can be applied to ordinary probabilities and used to justify updating by conditioning, whereas only the second can be applied to belief functions and used to justify Dempster's rule of combination.

The most appropriate mechanism for evidence combination is also still being debated \cite{wang08-reliable}. The reason is that the choice seems to depend on meta-information about the reliability and independence of the sources involved, which is hardly accessible. As argued in \cite{cuzzolin2021springer},
\emph{working with intervals of belief functions} \cite{denoeux99reasoning} may be the way forward, as this acknowledges the meta-uncertainty about the nature of the sources generating the evidence.
The same holds for conditioning \cite{xu96reasoning,yu94conditional,kohlas88b}, as opposed to combining, belief functions.


\subsection{A research programme}

Thus, in this paper we outline
a potentially interesting research agenda for the future development of random set and belief function theory. For obvious reasons, we will only be able to touch on a few of the most interesting avenues of investigation, and even those will not be discussed beyond a certain level of detail. We hope, however, that this brief survey will stimulate the reader to pursue some of these research directions and to contribute to the further maturation of the theory in the near future.

Although random set theory as a mathematical formalism is quite well developed, thanks in particular to the work of Ilya Molchanov \cite{molchanov1997statistical,Molchanov05}, a theory of {statistical inference with random sets} is not yet in sight. 
Therefore, in Section \ref{sec:future-agenda} 
We briefly consider 
the following questions:
\begin{itemize}
\item
the introduction of the notions of lower and upper likelihoods (Section \ref{sec:future-likelihoods}), in order to move beyond belief function inference approaches which take the classical likelihood function at face value;
\item
the formulation of a framework for (generalised) logistic regression \cite{wright1995logistic} with belief functions, making use of such generalised lower and upper likelihoods (Section \ref{sec:future-logistic}) and their decomposition properties;
\item
the generalisation of the law of total probability to random sets (Section \ref{sec:future-total}), starting with belief functions \cite{cuzzolin14lap,dezert2018total};
\item
the generalisation of classical limit theorems (central limit theorem, law of large numbers) to the case of random sets (Section \ref{sec:future-theorems}): this would allow, for instance, a rigorous definition of Gaussian random sets and belief functions alternative to that proposed by Molchanov \cite{Molchanov05};
\item
the introduction (somewhat despite Dempster's aversion to them) of parametric models based on random sets (Section \ref{sec:future-families}),
which would potentially allow us to perform robust hypothesis testing (Section \ref{sec:future-hypothesis}),
thus laying the foundations for a theory of
frequentist inference with random sets \cite{martin2021imprecise} (Section \ref{sec:future-frequentist});
\item
the development of a theory of random variables and processes in which the underlying probability space is replaced by a random-set space (Section \ref{sec:future-random-variables}): in particular, this requires the generalisation of the notion of the Radon--Nikodym derivative to belief measures.
\end{itemize}

The geometric approach to uncertainty, proposed by the author in \cite{cuzzolin2021springer},
is also open to a number of further developments (Section \ref{sec:future-geometric}), including:
\begin{itemize}
\item
the geometry of combination rules other than Dempster's (Section \ref{sec:future-combination}), and the associated conditioning operators  (\ref{sec:future-conditioning});
\item
an exploration of the geometry of continuous extensions of belief functions,
starting from the geometry of belief functions on Borel intervals, to later tackle the general random-set representation;
\item
a geometric analysis of uncertainty measures beyond those presented in \cite{cuzzolin2021springer},
including e.g. capacities \cite{Goubault-Larrecq2007} and gambles \cite{bradley2019imprecise} (Section \ref{sec:future-uncertainty});
\item
newer geometric representations (Section \ref{sec:future-fancier}), based on (among others) isoperimeters of convex bodies or exterior algebras \cite{yokonuma1992tensor}.
\end{itemize}
More speculatively, 
the possibility exists of conducting inference in a purely geometric fashion, by finding a common representation for both belief measures and the data that drive the inference.

More theoretical advances are necessary, in our view, including for instance: 
\begin{itemize}
\item
a set of prescriptions for reasoning with intervals of belief functions, as \cite{cuzzolin2021springer}
this seems to be a natural way to avoid the tangled issue of choosing a combination rule and handling metadata;
\item
the further development of machine learning tools based on belief and random-set theory (see \cite{cuzzolin2021springer}, Chapter 5), 
demonstrating the potential of these methods to model the epistemic uncertainty induced by forcibly limited training sets.
\end{itemize}

Last but not least, future work will need to be directed towards tackling high-impact applications by means of random set/belief function theory (Section \ref{sec:future-applications}). Here we will briefly consider, in particular: 
\begin{itemize}
\item
the possible creation of a framework for climate change predictions based on random sets (Section \ref{sec:future-climate}), able to overcome the limitations of existing (albeit neglected) Bayesian approaches; 
\item
the generalisation of max-entropy classifiers \cite{nigam1999using} and log-linear models to the case of belief measures (Section \ref{sec:future-maxentropy}), as an example of the fusion of machine learning with uncertainty theory;
\item
new robust foundations for machine learning itself (Section \ref{sec:future-slt}), obtained by generalising Vapnik's \emph{probably approximately correct} (PAC) analysis \cite{vapnik1998statistical,vapnik2013nature} to the case in which the training and test distributions, rather than coinciding, come from the same random set.
\end{itemize}

\section{Belief functions} \label{sec:measures}

We first recall the basic definitions of belief function theory.

\subsection{Belief and plausibility measures}

\begin{definition}\label{def:bpa}
A \emph{basic probability assignment} (BPA) \cite{Augustin96} over a finite domain $\Theta$ is a set function \cite{denneberg99interaction} $m : 2^\Theta\rightarrow[0,1]$ defined on the collection $2^\Theta$ of all subsets of $\Theta$ s.t.:
\[
m(\emptyset)=0, \; \sum_{A\subset\Theta} m(A)=1.
\]
\end{definition}
The quantity $m(A)$ is called the \emph{basic probability number} or `mass' \cite{kruse91tool,kruse91reasoning} assigned to $A$. 
The elements of the power set $2^\Theta$ associated with non-zero values of $m$ are called the \emph{focal elements} (FEs) of $m$.
\begin{definition} \label{def:bel2}
The \emph{belief function} (BF) associated with a basic probability assignment $m : 2^\Theta\rightarrow[0,1]$ is the set function $Bel : 2^\Theta\rightarrow[0,1]$ defined as:
\begin{equation}\label{eq:belief} Bel(A) = \sum_{B\subseteq A} m(B). \end{equation}
\end{definition}
The domain $\Theta$ on which a belief function is defined is usually interpreted as the set of possible answers to a given problem, exactly one of which is the correct one. For each subset (`event') $A\subset \Theta$ the quantity $Bel(A)$ takes on the meaning of \emph{degree of belief} that the truth lies in $A$,
and represents the {total} belief committed to a set of possible outcomes $A$ by the available evidence $m$.

Another mathematical expression of the evidence generating a belief function $Bel$ 
is the \emph{upper probability} or \emph{plausibility} of an event $A$: $Pl(A) \doteq 1 - Bel(\bar{A})$,
as opposed to its \emph{lower probability} $Bel(A)$ \cite{cuzzolin10ida}. 
The {corresponding} \emph{plausibility function} $Pl : 2^\Theta \rightarrow [0,1]$ 
can be expressed as:
\[
Pl(A) = \sum_{B\cap A\neq \emptyset} m(B) \geq Bel(A).
\]

\subsection{Evidence combination} \label{sec:combination}

The issue of combining the belief function representing our current knowledge state with a new one encoding new evidence is central in belief theory. After an initial proposal by Dempster, other aggregation operators have been proposed, based on different assumptions on the nature of the sources of evidence to combine.

\begin{definition} \label{def:dempster}
The \emph{orthogonal sum} or \emph{Dempster's combination} $Bel_1 \oplus Bel_2 : 2^\Theta \rightarrow [0,1]$ of two belief functions $Bel_1 : 2^\Theta \rightarrow [0,1]$, $Bel_2 : 2^\Theta \rightarrow [0,1]$ defined on the same domain $\Theta$ is the unique BF on $\Theta$ with as focal elements all the {non-empty} intersections of FEs of $Bel_1$ and $Bel_2$, and basic probability assignment:
\begin{equation} \label{eq:dempster}
\displaystyle m_{\oplus}(A) = \frac{m_\cap(A)} {1- m_\cap(\emptyset)},
\end{equation}
where $m_i$ denotes the BPA of the input BF $Bel_i$, and:
\[
m_\cap(A) = \sum_{B \cap C = A} m_1(B) m_2(C).
\]
\end{definition}

Rather than normalising (as in (\ref{eq:dempster})), 
Smets' \emph{conjunctive rule} leaves the conflicting mass $m(\emptyset)$ with the empty set:
\begin{equation} \label{eq:combination-smets-conjunctive}
m_{\ocap}(A) = \left \{ 
\begin{array}{ll}
m_\cap (A) 
& \emptyset \subseteq A \subseteq \Theta, \\ 
m_\cap(\emptyset) & A = \emptyset,
\end{array}
\right .
\end{equation}
and is thus applicable to `{unnormalised}' beliefs \cite{ubf}.

In Dempster's rule, consensus between two sources is expressed by the intersection of the supported events. When the \emph{union} 
is taken to express consensus, we obtain the \emph{disjunctive} rule of combination \cite{Kramosil02probabilistic-analysis,YAMADA20081689}:
\begin{equation} \label{eq:combination-disjunctive}
m_{\ocup}(A) = \sum_{B \cup C = A} m_1(B) m_2(C).
\end{equation}
This yields more cautious inferences than conjunctive rules, by producing belief functions that are less `committed', i.e., have larger focal sets.
Under disjunctive combination, $Bel_1 \ocup Bel_2 (A) = Bel_1(A) \cdot Bel_2(A)$, i.e., input belief values are simply multiplied.
Such a `dual' combination rule was mentioned by Kramosil as well in \cite{Kramosil02probabilistic-analysis}. The disjunctive rule is called `combination by union' by Yamada (2008) \cite{YAMADA20081689}.

\subsubsection{Conditioning} \label{sec:conditioning}

Belief functions can also be conditioned, rather than combined, whenever we are presented hard evidence of the form `$A$ is true' 
\cite{Chateauneuf89,fagin91new,Jaffray92,gilboa1993updating,denneberg1994conditioning,yu94conditional,itoh95new}.

In particular,
Dempster's combination naturally induces a conditioning operator.
Given a conditioning event $A \subset \Theta$, the `logical' (or `categorical', in Smets' terminology) belief function $Bel_A$ such that $m(A)=1$ is combined via Dempster's rule with the a-priori belief function $Bel$. The resulting BF $Bel \oplus Bel_A$ is the {conditional belief function given $A$} \emph{a la Dempster}, denoted by $Bel_\oplus(A|B)$.

\subsection{Multivariate analysis} \label{sec:multivariate}

In many applications, we need to express uncertain information about a number of distinct variables (e.g., $X$ and $Y$) taking values in different domains ($\Theta_X$ and $\Theta_Y$, respectively). The reasoning process needs then to take place in the Cartesian product of the domains associated with each individual variable. 

Let then $\Theta_X$ and $\Theta_Y$ be two sample spaces associated with two distinct variables, and let $m^{XY}$ be a mass function on $\Theta_{XY} = \Theta_X \times \Theta_Y$.
The latter can be expressed in the coarser domain $\Theta_X$ by transferring each mass $m^{XY}(A)$ to the {projection} $A\downarrow \Theta_X$ of $A$ on $\Theta_X$.
We obtain a \emph{marginal} mass function on $\Theta_X$, denoted by: 
\begin{equation}
\label{eq:marginal}    
m^{XY}_{\downarrow X}(B) \doteq \sum_{\{A \subseteq \Theta_{XY}, A\downarrow \Theta_X=B\}}  m^{XY}(A), \forall B \subseteq \Theta_X.
\end{equation}
Conversely, a mass function $m^X$ on $\Theta_X$ can be expressed in $\Theta_X\times \Theta_Y$  by transferring each mass $m^X(B)$ to the {cylindrical extension} $B^{\uparrow XY} \doteq B \times \Omega_Y$ of $B$.
The \emph{vacuous extension} of $m^X$ onto $\Theta_X\times \Theta_Y$ will then be:
\begin{equation} \label{eq:vacuous-extension}
m_X^{\uparrow XY}(A) \doteq
\begin{cases}
m^X(B) & \text{if } A=B \times \Omega_Y,\\
0 & \text{otherwise}.
\end{cases}
\end{equation}
The associated BF is denoted by $Bel_X^{\uparrow XY}$.

\subsection{Refinings} 

An even more general treatment of belief functions over different but related sample spaces is provided by the notion of \emph{refining}.

\begin{definition} \label{def:refining}
Given two frames of discernment $\Theta$ and $\Omega$, a map $\rho :2^\Theta \rightarrow2^\Omega$ is said to be a \emph{refining} if it
satisfies the following conditions:
\begin{enumerate}
\item $\rho(\{\theta\})\neq\emptyset\;\forall \theta\in\Theta$.
\item $\rho(\{\theta\})\cap\rho(\{\theta'\})=\emptyset$ if
$\theta\neq\theta'$.
\item $\displaystyle \cup_{\theta\in\Theta}\rho(\{\theta\})=\Omega$.
\end{enumerate}
\end{definition}
\noindent
The \emph{outer reduction} associated with a refining $\rho$ is the map $\bar{\rho} : 2^\Omega \rightarrow 2^\Theta$
given by
\begin{equation} \label{eq:outer-reduction}
\bar{\rho}(A) = \Big \{ \theta\in\Theta \Big | \rho(\{\theta\}) \cap A \neq \emptyset \Big \}.
\end{equation}

\section{A statistical random set theory} \label{sec:future-agenda}

\subsection{Lower and upper likelihoods} \label{sec:future-likelihoods}

The traditional likelihood function 
is a conditional probability of the data given a parameter $\theta \in \Theta$, i.e., a  family of PDFs over the measurement space $\X$ parameterised by $\theta$. Most of the work on belief function inference just takes the notion of a likelihood as a given, and constructs belief functions from an input likelihood function (see \cite{cuzzolin2021springer}, Chapter 4). 

However, there is no reason why we should not formally define a `belief likelihood function' mapping a sample observation $x \in \X$ to a real number, rather than
use the conventional likelihood to construct belief measures.
It is natural to define such a belief likelihood function as a family of belief functions on $\X$, $Bel_\X(.|\theta)$, parameterised by $\theta \in \Theta$. An expert in belief theory will 
note that such a parameterised family is the input to Smets's generalised Bayesian theorem \cite{smets93belief}, a collection of `conditional' belief functions.
Such a belief likelihood takes values on \emph{sets} of outcomes, $A \subset \Theta$, of which individual outcomes are just a special case.

This seems to provide a {natural setting for computing likelihoods of set-valued observations}, in accordance with the random-set philosophy. 

\subsubsection{Belief likelihood function of repeated trials}

Let $Bel_{\X_i}(A|\theta)$, for $i=1,2, \ldots, n$ be a parameterised family of belief functions on $\X_i$, the space of quantities that can be observed at time $i$, depending on a parameter $\theta \in \Theta$. A series of repeated trials then assumes values in $\X_1 \times \cdots \times \X_n$, whose elements are tuples of the form $\vec{x} = (x_1, \ldots, x_n) \in \X_1 \times \cdots \times \X_n$.
We call such tuples `sharp' samples, as opposed to arbitrary subsets $A \subset \X_1 \times \cdots \times \X_n$ of the space of trials. Note that we are not assuming the trials to be equally distributed at this stage, nor do we assume that they come from the same sample space.

\begin{definition} \label{def:belief-likelihood}
The belief likelihood function $Bel_{\X_1 \times \cdots \times \X_n} : 2^{\X_1 \times \cdots \times \X_n} \rightarrow [0,1]$ of a series of repeated trials is defined as
\begin{equation} \label{eq:belief-likelihood}
Bel_{\X_1 \times \cdots \times \X_n}(A|\theta) \doteq Bel_{\X_1}^{\uparrow \times_i \X_i} \odot \cdots \odot Bel_{\X_n}^{\uparrow \times_i \X_i} (A|\theta),
\end{equation}
where $ Bel_{\X_j}^{\uparrow \times_i \X_i}$ is the vacuous extension (\ref{eq:vacuous-extension})
of $Bel_{\X_j}$ to the Cartesian product $\X_1 \times \cdots \times \X_n$ where the observed tuples live, 
$A \subset \X_1 \times \cdots \times \X_n$ is an arbitrary subset of series of trials
and $\odot$ is an arbitrary combination rule.
\end{definition}

In particular, when the subset $A$ reduces to a sharp sample, $A = \{\vec{x}\}$, we can define the following generalisations of the notion of a likelihood. 
\begin{definition} \label{def:lower-upper-likelihoods}
We call the quantities 
\begin{equation} \label{eq:lower-upper-likelihoods}
\begin{array}{lll}
\underline{L}(\vec{x}) & \doteq & Bel_{\X_1 \times \cdots \times \X_n}(\{(x_1, \ldots, x_n)\}|\theta),
\\ \overline{L}(\vec{x}) & \doteq & Pl_{\X_1 \times \cdots \times \X_n}(\{(x_1, \ldots, x_n)\}|\theta) 
\end{array}
\end{equation}
the \emph{lower likelihood} and \emph{upper likelihood}, respectively, of the sharp sample $A = \{ \vec{x} \} = \{(x_1, \ldots, x_n)\}$.
\end{definition}

\subsubsection{Binary trials: The conjunctive case} \label{sec:belief-likelihood-conjunctive}

Belief likelihoods factorise into simple products whenever conjunctive combination is employed (as a generalisation of classical stochastic independence) in Definition \ref{def:belief-likelihood}, and the case of trials with binary outcomes is considered.

\paragraph{Focal elements of the belief likelihood}

Let us first analyse the case $n=2$. We seek the Dempster sum $Bel_{\X_1} \oplus Bel_{\X_2}$, where $\X_1 = \X_2 = \{T,F\}$.

Figure \ref{fig:casen2}(a) is a diagram of all the intersections of focal elements of the two input belief functions on  $\X_1 \times \X_2$. 
\begin{figure}[ht!]
\begin{center}
\begin{tabular}{cc}
\includegraphics[width = 0.45\textwidth]{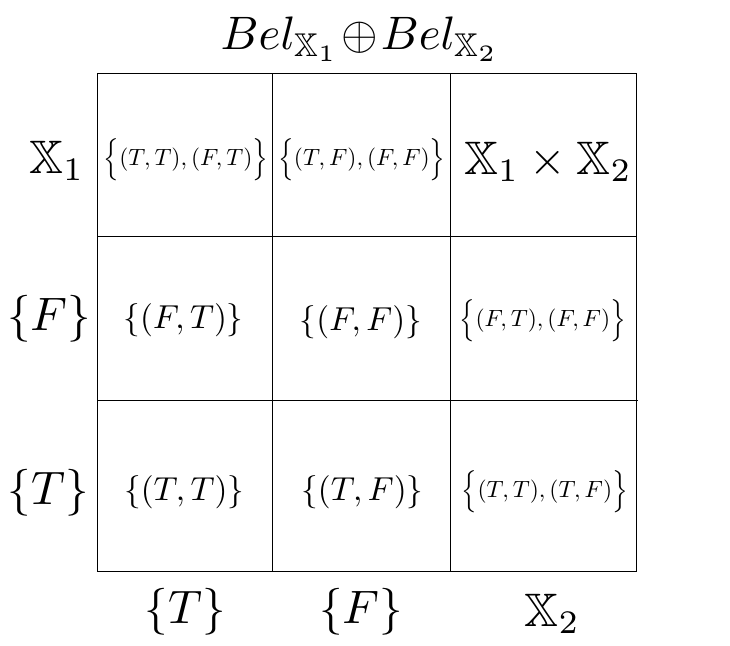}
&
\includegraphics[width = 0.45\textwidth]{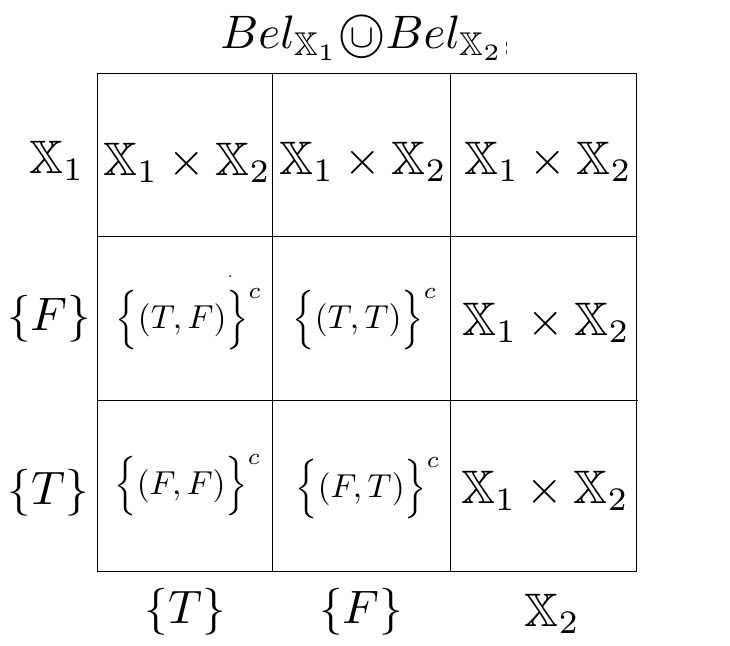}
\\
(a) & (b)
\end{tabular}
\end{center}
\caption{(a) Graphical representation of the Dempster combination $Bel_{\X_1} \oplus Bel_{\X_2}$ on $\X_1 \times \X_2$, in the binary case in which $\X_1 = \X_2 = \{T,F\}$. (b) Graphical representation of the disjunctive combination $Bel_{\X_1}$\ocup $Bel_{\X_2}$ there.} \label{fig:casen2} 
\end{figure}
There are $9 = 3^2$ distinct non-empty intersections, which correspond to the focal elements of $Bel_{\X_1} \oplus Bel_{\X_2}$. 

According to (\ref{eq:dempster}), the mass of the focal element $A_1 \times A_2$, $A_1 \subseteq \X_1$, $A_2 \subseteq \X_2$, is then
\begin{equation} \label{eq:lemma-n2}
m_{Bel_{\X_1} \oplus Bel_{\X_2}} (A_1 \times A_2) = m_{{\X_1}} (A_1) \cdot m_{{\X_2}} (A_2).
\end{equation}
Note that this result holds when the conjunctive rule (\ref{eq:combination-smets-conjunctive}) \ocap\hspace{0.5mm} is used as well, for none of the intersections are empty, and hence no normalisation is required. Nothing is assumed about the mass assignment of the two belief functions $Bel_{\X_1}$ and $Bel_{\X_2}$. 

We can now prove the following lemma.
\begin{lemma} \label{lem:factorisation-conjunctive}
For any $n \in \mathbb{Z}$, the belief function $Bel_{\X_1} \oplus \cdots \oplus Bel_{\X_n}$, where $\X_i = \X = \{T,F\}$, has $3^n$ focal elements, namely all possible Cartesian products $A_1\times \cdots \times A_n$ of $n$ non-empty subsets $A_i$ of $\X$, with BPA
\[
m_{Bel_{\X_1} \oplus \cdots \oplus Bel_{\X_n}} (A_1\times \cdots \times A_n) = \prod_{i=1}^n m_{\X_i}(A_i).
\]
\end{lemma}
\begin{proof}
The proof is by induction.
The thesis was shown to be true for $n=2$ in (\ref{eq:lemma-n2}).
In the induction step, we assume that the thesis is true for $n$, and prove it for $n+1$.

If $Bel_{\X_1} \oplus \cdots \oplus Bel_{\X_{n}}$, defined on $\X_1 \times \cdots \times \X_{n}$, has as focal elements the $n$-products $A_1\times \cdots \times A_n$ with $A_i \in \big \{ \{T\}, \{F\}, \X \big \}$ for all $i$, its vacuous extension to $\X_1 \times \cdots \times \X_n \times \X_{n+1}$ will have as focal elements the $n+1$-products of the form $A_1\times \cdots \times A_n \times \X_{n+1}$,
with $A_i \in \big \{ \{T\}, \{F\}, \X \big \}$ for all $i$.

The belief function $Bel_{\X_{n+1}}$ is defined on $\X_{n+1} = \X$, with three focal elements $\{T\}$, $\{F\}$ and $\X = \{T,F\}$. Its vacuous extension to $\X_1 \times \cdots \times \X_n \times \X_{n+1}$ thus has the following three focal elements:
$\X_1 \times \cdots \times \X_n \times \{T\}$, 
$\X_1 \times \cdots \times \X_n \times \{F\}$ and 
$\X_1 \times \cdots \times \X_n \times \X_{n+1}$.
When computing $(Bel_{\X_1} \oplus \cdots \oplus Bel_{\X_{n}}) \oplus Bel_{\X_{n+1}}$ on the common refinement $\X_1 \times \cdots \times \X_n \times \X_{n+1}$, we need to compute the intersection of their focal elements, namely
\[
\begin{array}{c}
\big ( A_1\times \cdots \times A_n \times \X_{n+1} \big ) \cap \big ( \X_1 \times \cdots \times \X_n \times A_{n+1} \big )
=
A_1\times \cdots \times A_n \times A_{n+1}
\end{array}
\]
for all $A_{i} \subseteq \X_{i}$, $i=1, \ldots, n+1$. All such intersections are distinct for distinct focal elements of the two belief functions to be combined, and there are no empty intersections. By Dempster's rule (\ref{eq:dempster}), their mass is equal to the product of the original masses:
\[
\begin{array}{lll}
& & \displaystyle
m_{Bel_{\X_1} \oplus \cdots \oplus Bel_{\X_{n+1}}}(A_1\times \cdots \times A_n \times A_{n+1}) 
\\
& = & \displaystyle
m_{Bel_{\X_1} \oplus \cdots \oplus Bel_{\X_n}}(A_1\times \cdots \times A_n) \cdot m_{Bel_{\X_{n+1}}}(A_{n+1}).
\end{array}
\]

Since we have assumed that the factorisation holds for $n$, the thesis easily follows.
\end{proof}

As no normalisation is involved in the combination $Bel_{\X_1} \oplus \cdots \oplus Bel_{\X_n}$, Dempster's rule coincides with the conjunctive rule, and Lemma \ref{lem:factorisation-conjunctive} holds for \ocap \hspace{0.3mm} as well.

\paragraph{Factorisation for `sharp' tuples} 

The following then becomes a simple corollary.
\begin{theorem} \label{the:belief-likelihood-decomposition-ocap}
When either \ocap \hspace{0.3mm} or $\oplus$ is used as a combination rule in the definition of the belief likelihood function, the following decomposition holds for tuples $(x_1, \ldots, x_n)$, $x_i \in \X_i$, which are the singleton elements of $\X_1 \times \cdots \times \X_n$, with $\X_1 = \cdots = \X_n = \{T,F\}$:
\begin{equation} \label{eq:belief-likelihood-decomposition-ocap}
\begin{array}{c}
\displaystyle
Bel_{\X_1 \times \cdots \times \X_n}(\{(x_1, \ldots, x_n)\}|\theta) = 
\prod_{i=1}^n Bel_{\X_i} (\{x_i\}|\theta).
\end{array}
\end{equation}
\end{theorem}
\begin{proof}
For the singleton elements of $\X_1 \times \cdots \times \X_n$, since $\{(x_1, \ldots, x_n)\} = \{x_1\} \times \cdots \times \{x_n\}$, (\ref{eq:belief-likelihood-decomposition-ocap}) becomes
\[
\begin{array}{lll}
& & \displaystyle
Bel_{\X_1 \times \cdots \times \X_n}(\{(x_1, \ldots, x_n)\})
\\
& = & \displaystyle
m_{Bel_{\X_1} \oplus \cdots \oplus Bel_{\X_n}}(\{ (x_1, \ldots, x_n) \}) 
=
\prod_{i=1}^n m_{\X_i}(\{x_i \}) = \prod_{i=1}^n Bel_{\X_i} (\{x_i\}),
\end{array}
\]
where the mass factorisation follows from Lemma \ref{lem:factorisation-conjunctive},
as on singletons mass and belief values coincide.
\end{proof}

There is evidence to support the following conjecture as well.
\begin{conjecture} \label{con:plausibility-likelihood-decomposition-ocap}
When either \ocap$\hspace{0.3mm}$ or $\oplus$ is used as a combination rule in the definition of the belief likelihood function, the following decomposition holds for the associated plausibility values on tuples $(x_1, \ldots, x_n)$, $x_i \in \X_i$, which are the singleton elements of $\X_1 \times \cdots \times \X_n$, with $\X_1 = \ldots = \X_n = \{ T, F \}$:
\begin{equation} \label{eq:plausibility-likelihood-decomposition-ocap}
\begin{array}{c}
\displaystyle
Pl_{\X_1 \times \cdots \times \X_n}(\{(x_1, \ldots, x_n)\}|\theta) = \prod_{i=1}^n Pl_{\X_i} (\{x_i\}|\theta).
\end{array}
\end{equation}
\end{conjecture}

In fact, we can write
\begin{equation} \label{eq:plausibilities-singletons}
\begin{array}{lll}
Pl_{\X_1 \times \cdots \times \X_n}(\{(x_1, \ldots, x_n)\}) 
& = & \displaystyle
1 - Bel_{\X_1 \times \cdots \times \X_n}(\{(x_1, \ldots, x_n)\}^c) 
\\ \\
& = & \displaystyle
1 - \sum_{B \subseteq \{(x_1,...,x_n)\}^c} m_{Bel_{\X_1} \oplus \cdots \oplus Bel_{\X_n}}(B).
\end{array}
\end{equation}

By Lemma \ref{lem:factorisation-conjunctive}, all the subsets $B$ with non-zero mass are Cartesian products of the form $A_1 \times \cdots \times A_n$, $\emptyset \neq A_i \subseteq \X_i$.
We then need to understand the nature of the focal elements of $Bel_{\X_1 \times \cdots \times \X_n}$ which are subsets of an arbitrary singleton complement $\{(x_1, \ldots, x_n)\}^c$. 

For binary spaces $\X_i = \X = \{T,F\}$, by the definition of the Cartesian product, each such $A_1 \times \cdots \times A_n \subseteq \{(x_1, \ldots, x_n)\}^c$ is obtained by replacing a number $1 \leq k \leq n$ of components of the tuple $(x_1, \ldots, x_n) = \{x_1\} \times \cdots \times \{x_n\}$ with a different subset of $\X_i$ (either $\{{x}_i\}^c = \X_i \setminus \{x_i\}$ or $\X_i$). 
There are $\binom{n}{k}$ such sets of $k$ components in a list of $n$.
Of these $k$ components, in general $1 \leq m \leq k$ will be replaced by $\{{x}_i\}^c$, while the other $1 \leq k-m < k$ will be replaced by $\X_i$. Note that not all $k$ components can be replaced by $\X_i$, since the resulting focal element would contain the tuple $\{(x_1, \ldots, x_n)\} \in \X_1 \times \cdots \times \X_n$.

The following argument can be proved for $(x_1,...,x_n) = (T,...,T)$, under the additional assumption that $Bel_{\X_1} \cdots Bel_{\X_n}$ are equally distributed with $p \doteq Bel_{\X_i}(\{T\})$, $q \doteq Bel_{\X_i}(\{F\})$ and $r \doteq Bel_{\X_i}(\X_i)$.
\\
If this is the case, for fixed values of $m$ and $k$ all the resulting focal elements have the same mass value, namely $p^{n-k} q^m r^{k-m}$. As there are exactly $\binom{k}{m}$ such focal elements, (\ref{eq:plausibilities-singletons}) can be written as
\[
1 - \sum_{k=1}^n \binom{n}{k} \sum_{m=1}^k \binom{k}{m} p^{n-k} q^{m} r^{k-m},
\]
which can be rewritten as
\[
1 - \sum_{m=1}^n q^m \sum_{k=m}^n \binom{n}{k} \binom{k}{m} p^{n-k} r^{k-m}.
\]
A change of variable $l=n-k$, where $l=0$ when $k=n$, and $l=n-m$ when $k=m$, allows us to write this as
\[
1 - \sum_{m=1}^n q^m \sum_{l=0}^{n-m} \binom{n}{n-l} \binom{n-l}{m} p^{l} 
r^{(n-m)-l},
\]
since $k-m = n-l-m$, $k=n-l$. Now, as
\[
\binom{n}{n-l} \binom{n-l}{m} = \binom{n}{m} \binom{n-m}{l},
\]
we obtain
\[
1 - \sum_{m=1}^n q^m \binom{n}{m} \sum_{l=0}^{n-m} \binom{n-m}{l} p^{l} 
r^{(n-m)-l}.
\]
By Newton's binomial theorem, the latter is equal to
\[
1 - \sum_{m=1}^n q^m \binom{n}{m} (p+r)^{n-m} = 1 - \sum_{m=1}^n q^m \binom{n}{m} (1-q)^{n-m},
\]
since $p+r = 1-q$.
As $\sum_{m=0}^n \binom{n}{m} q^m (1-q)^{n-m} = 1$, again by Newton's binomial theorem, we get
\[
Pl_{\X_1 \times \cdots \times \X_n}(\{(T, \ldots, T)\}) = 1 - [1 - (1-q)^n ] = (1-q)^n = \prod_{i=1}^n Pl_{\X_i}(\{ T\}).
\]

\begin{question}
Does the decomposition (\ref{eq:plausibilities-singletons}) hold for any sharp sample and arbitrary belief functions $Bel_{\X_1}, \ldots, Bel_{\X_n}$ defined on arbitrary frames of discernment $\X_1, \ldots, \X_n$?
\end{question}

\paragraph{Factorisation for Cartesian products}

The decomposition (\ref{eq:belief-likelihood-decomposition-ocap}) is equivalent to what Smets calls \emph{conditional conjunctive independence} in his general Bayes theorem.

In fact, for binary spaces the factorisation (\ref{eq:belief-likelihood-decomposition-ocap}) generalises to all subsets $A \subseteq \X_1 \times \cdots \times \X_n$ which are Cartesian products of subsets of $\X_1, \ldots, \X_n$, respectively: $A = A_1 \times \cdots \times A_n$, $A_i \subseteq \X_i$ for all $i$.

\begin{corollary} \label{cor:cci} 
Whenever $A_i \subseteq \X_i$, $i = 1, \ldots, n$, and $\X_i = \X = \{T,F\}$ for all $i$, under conjunctive combination we have that
\begin{equation} \label{eq:corollary-cci}
Bel_{\X_1 \times \cdots \times \X_n}(A_1 \times \cdots \times A_n | \theta) = 
\prod_{i=1}^n Bel_{\X_i} (A_i | \theta).
\end{equation}
\end{corollary}
\begin{proof}
As, by Lemma \ref{lem:factorisation-conjunctive}, all the focal elements of $Bel_{\X_1}$\ocap $\cdots$\ocap $Bel_{\X_n}$ are Cartesian products of the form $B = B_1 \times \cdots \times B_n$, $B_i \subseteq \X_i$,
\[
Bel_{\X_1 \times \cdots \times \X_n}(A_1 \times \cdots \times A_n | \theta)
=
\sum_{\stackrel{\mathlarger{B \subseteq A_1 \times \cdots \times A_n}}{B = B_1 \times \cdots \times B_n}}
m_{\X_1}(B_1) \cdot \ldots \cdot m_{\X_n}(B_n).
\] 
But 
\[
\{
B \subseteq A_1 \times \cdots \times A_n, B = B_1 \times \cdots \times B_n
\}
=
\{
B = B_1 \times \cdots \times B_n, B_i \subseteq A_i \forall i
\},
\]
since, if $B_j \not \subset A_j$ for some $j$, the resulting Cartesian product would not be contained within $A_1 \times \cdots \times A_n$. Therefore,
\[
Bel_{\X_1 \times \cdots \times \X_n}(A_1 \times \cdots \times A_n | \theta)
=
\sum_{\stackrel{\mathlarger{B = B_1 \times \cdots \times B_n}}{B_i \subseteq A_i \forall i}}
m_{\X_1}(B_1) \cdot \ldots \cdot m_{\X_n}(B_n). 
\]
For all those $A_i$'s, $i=i_1, \ldots, i_m$ that are singletons of $\X_i$, necessarily $B_i = A_i$ and we can write the above expression as
\[
m(A_{i_1}) \cdot \ldots \cdot m(A_{i_m}) 
\sum_{B_j \subseteq A_j, j \neq i_1,...,i_m} \prod_{j \neq i_1,...,i_m} m_{\X_j}(B_j).
\]
If the frames involved are binary, $\X_i = \{T,F\}$, those $A_i$'s that are not singletons coincide with $\X_i$, so that we have
\[
m(A_{i_1}) \cdot \ldots \cdot m(A_{i_m}) \sum_{B_j \subseteq \X_j, j \neq i_1,...,i_m} \prod_{j} m_{\X_j}(B_j).
\]
The quantity 
\[
\sum_{B_j \subseteq \X_j, j \neq i_1,...,i_m} \prod_{j} m_{\X_j}(B_j) 
\]
is, according to the definition of the conjunctive combination, the sum of the masses of all the possible intersections of (cylindrical extensions of) focal elements of $Bel_{\X_j}$, $j \neq i_1, \ldots, i_m$. Thus, it adds up to 1.
In conclusion (forgetting the conditioning on $\theta$ in the derivation for the sake of readability),
\[
\begin{array}{ll}
& Bel_{\X_1 \times \cdots \times \X_n}(A_1 \times \cdots \times A_n) 
\\
= & m(A_{i_1}) \cdot \ldots \cdot m(A_{i_m}) \cdot 1 \cdot \ldots \cdot 1 
\\
= & Bel_{\X_{i_1}}(A_{i_1}) \cdot \ldots \cdot Bel_{\X_{i_m}}(A_{i_m}) \cdot Bel_{\X_{j_1}}(\X_{j_1}) \cdot \ldots \cdot Bel_{\X_{j_k}}(\X_{j_k}) 
\\
= & Bel_{\X_{i_1}}(A_{i_1}) \cdot \ldots \cdot Bel_{\X_{i_m}}(A_{i_m}) \cdot Bel_{\X_{j_1}}(A_{j_1}) \cdot \ldots \cdot Bel_{\X_{j_k}}(A_{j_k}),
\end{array}
\] 
and we have (\ref{eq:corollary-cci}).
\end{proof}

Corollary \ref{cor:cci} states that conditional conjunctive independence always holds, for events that are Cartesian products, whenever the frames of discernment concerned are binary. 

\subsubsection{Binary trials: The disjunctive case} \label{sec:belief-likelihood-disjunctive}

Similar factorisation results hold when the (more cautious) disjunctive combination \ocup$\hspace{0.2mm}$ is used.
As in the conjunctive case, we first analyse the case $n=2$.

\paragraph{Structure of the focal elements}

We seek the disjunctive combination $Bel_{\X_1}$\ocup $Bel_{\X_2}$, where each $Bel_{\X_i}$ has as focal elements $\{T\}$, $\{F\}$ and $\X_i$.
Figure \ref{fig:casen2}(b) is a diagram of all the unions of focal elements of the two input belief functions on their common refinement $\X_1 \times \X_2$.
There are $5 = 2^2 +1$ distinct such unions, which correspond to the focal elements of $Bel_{\X_1}$\ocup $Bel_{\X_2}$, with mass values
\[
\begin{array}{lll}
m_{Bel_{\X_1} \oplus Bel_{\X_2}} (\{(x_i,x_j)\}^c) & = & m_{{\X_1}} (\{x_i\}^c) \cdot m_{{\X_2}} (\{x_j\}^c),
\\
m_{Bel_{\X_1} \oplus Bel_{\X_2}} (\X_1 \times \X_2) & = & \displaystyle 1 - \sum_{i,j} m_{{\X_1}} (\{x_i\}^c) \cdot m_{{\X_2}} (\{x_j\}^c).
\end{array}
\]

We can now prove the following lemma. 
\begin{lemma} \label{lem:factorisation-disjunctive}
The belief function $Bel_{\X_1}$\ocup $\cdots$\ocup $Bel_{\X_n}$, where $\X_i = \X = \{T,F\}$, has $2^n + 1$ focal elements, namely all the complements of 
the $n$-tuples $\vec{x} = (x_1, \ldots, x_n)$ of singleton elements $x_i \in \X_i$, with BPA
\begin{equation} \label{eq:mass-disjunctive}
\begin{array}{c}
m_{Bel_{\X_1} \text{\ocup} \cdots \text{\ocup} Bel_{\X_n}}(\{ (x_1, \ldots, x_n) \}^c) 
= m_{\X_1}(\{x_1\}^c) \cdot \ldots \cdot m_{\X_n}(\{x_n\}^c),
\end{array}
\end{equation}
plus the Cartesian product $\X_1 \times \cdots \times \X_n$ itself, with its mass value given by normalisation.
\end{lemma}
\begin{proof}
The proof is by induction. The case $n=2$ was proven above. In the induction step, we assume that the thesis is true for $n$, namely that the focal elements of $Bel_{\X_1}$\ocup $\cdots$\ocup $Bel_{\X_n}$ have the form 
\begin{equation} \label{eq:fe}
A = \Big \{ (x_1, \ldots, x_n) \}^c = \{ (x'_1, \ldots, x'_n) \Big | \exists i : \{x'_i\} = \{x_i\}^c \Big \}, 
\end{equation}
where $x_i \in \X_i = \X$. We need to prove it is true for $n+1$.

The vacuous extension of (\ref{eq:fe}) has, trivially, the form
\[
\begin{array}{c}
A' = \Big \{ (x'_1, \ldots, x'_n, x_{n+1}) \Big |
\exists i\in \{1, \ldots, n\} : \{x'_i\} = \{x_i\}^c,  x_{n+1} \in \X \Big \}.
\end{array}
\]
Note that only $2=|\X|$ singletons of $\X_1 \times \cdots \times \X_{n+1}$ are \emph{not} in $A'$, for any given tuple $(x_1, \ldots, x_n)$.
The vacuous extension to $\X_1 \times \cdots \times \X_{n+1}$ of a focal element $B = \{x_{n+1}\}$ of $Bel_{\X_{n+1}}$ is instead 
\[
B' = \{ (y_1, \ldots, y_n ,x_{n+1}), y_i \in \X \; \forall i=1, \ldots, n \}.
\]

Now, all the elements of $B'$, except for $(x_1, \ldots, x_n, x_{n+1})$, are also elements of $A'$.
Hence, the union $A' \cup B'$ reduces to the union of $A'$ and $(x_1, \ldots, x_n, x_{n+1})$.
The only singleton element of $\X_1 \times \cdots \times \X_{n+1}$ not in $A' \cup B'$ is therefore $(x_1, \ldots, x_n, x'_{n+1})$, $\{ x'_{n+1} \} = \{ x_{n+1} \}^c$, for it is neither in $A'$ nor in $B'$.
All such unions are distinct, and therefore, by the definition of disjunctive combination, have mass $m(\{ (x_1, \ldots, x_n) \}^c) \cdot m(\{x_{n+1}\}^c)$, which by the inductive hypothesis is equal to (\ref{eq:mass-disjunctive}).
Unions involving either $\X_{n+1}$ or $\X_1 \times \cdots \times \X_n$ are equal to $\X_1 \times \cdots \times \X_{n+1}$ by the property of the union operator.
\end{proof}

\paragraph{Factorisation}

\begin{theorem} \label{the:belief-likelihood-decomposition-ocup}
In the hypotheses of Lemma \ref{lem:factorisation-disjunctive}, 
when the disjunctive combination \ocup \hspace{0.2mm} is used in the definition of the belief likelihood function, the following decomposition holds: 
\begin{equation} \label{eq:likelihood-decomposition-ocup-belief}
Bel_{\X_1 \times \cdots \times \X_n}(\{(x_1, \ldots, x_n)\}^c | \theta) = \prod_{i=1}^n Bel_{\X_i} (\{x_i\}^c | \theta).
\end{equation}
\end{theorem}
\begin{proof}
Since $\{(x_1, \ldots, x_n)\}^c$ contains only itself as a focal element,
\[
Bel_{\X_1 \times \cdots \times \X_n}(\{(x_1, \ldots, x_n)\}^c | \theta) = m(\{(x_1, \ldots, x_n)\}^c | \theta).
\]
By Lemma \ref{lem:factorisation-disjunctive}, the latter becomes
\[
Bel_{\X_1 \times \cdots \times \X_n}(\{(x_1, \ldots, x_n)\}^c | \theta) 
=
\prod_{i=1}^n m_{\X_i} (\{x_i\}^c | \theta) = \prod_{i=1}^n Bel_{\X_i} (\{x_i\}^c | \theta),
\]
as $\{x_i\}^c$ is a singleton element of $\X_i$, and we have (\ref{eq:likelihood-decomposition-ocup-belief}).
\end{proof}

Note that 
\[
Pl_{\X_1 \times \cdots \times \X_n}(\{(x_1, \ldots, x_n)\}^c | \theta) = 1 
\]
for all tuples $(x_1, \ldots, x_n)$, as the set $\{(x_1, \ldots, x_n)\}^c$ has non-empty intersection with all the focal elements of the belief function $Bel_{\X_1} \text{\ocup} \cdots \text{\ocup} Bel_{\X_n}$.

\subsubsection{General factorisation results}

A look at the proof of Lemma \ref{lem:factorisation-conjunctive} shows that the argument is in fact valid for the conjunctive combination of belief functions defined on an arbitrary collection $\X_1, \ldots,\X_n$ of finite spaces. 

Namely, we have the following theorem.

\begin{theorem} \label{the:factorisation-conjunctive}
For any $n \in \mathbb{Z}$, the belief function $Bel_{\X_1} \oplus \cdots \oplus Bel_{\X_n}$, where $\X_1, \ldots, \X_n$ are finite spaces, has as focal elements all the Cartesian products $A_1\times \cdots \times A_n$ of $n$ focal elements $A_1 \subseteq \X_1, \ldots, A_n \subseteq \X_n$, with BPA
\[
m_{Bel_{\X_1} \oplus \cdots \oplus Bel_{\X_n}} (A_1\times \cdots \times A_n) = \prod_{i=1}^n m_{\X_i}(A_i).
\]
\end{theorem}
The proof is left to the reader. The corollary below follows.

\begin{corollary} \label{cor:belief-likelihood-decomposition-ocap}
When either \ocap \hspace{0.2mm} or $\oplus$ is used as a combination rule in the definition of the belief likelihood function, the following decomposition holds for tuples $(x_1, \ldots, x_n)$, $x_i \in \X_i$, which are the singleton elements of $\X_1 \times \cdots \times \X_n$, with $\X_1, \ldots , \X_n$ being arbitrary discrete frames of discernment:
\begin{equation} \label{eq:belief-likelihood-decomposition-ocap-general}
\begin{array}{c}
\displaystyle
Bel_{\X_1 \times \cdots \times \X_n}(\{(x_1, \ldots, x_n)\}|\theta) 
= 
\prod_{i=1}^n Bel_{\X_i} (\{x_i\}|\theta).
\end{array}
\end{equation}
\end{corollary}
\begin{proof}
For the singleton elements of $\X_1 \times \cdots \times \X_n$, since $\{(x_1, \ldots, x_n)\} = \{x_1\} \times \cdots \times \{x_n\}$, (\ref{eq:belief-likelihood-decomposition-ocap-general}) becomes
\[
\begin{array}{lll}
& & \displaystyle
Bel_{\X_1 \times \cdots \times \X_n}(\{(x_1, \ldots, x_n)\})
\\
& = & \displaystyle
m_{Bel_{\X_1} \oplus \cdots \oplus Bel_{\X_n}}(\{ (x_1,...,x_n) \}) 
=
\prod_{i=1}^n m_{\X_i}(\{x_i \}) = \prod_{i=1}^n Bel_{\X_i} (\{x_i\}),
\end{array}
\]
where the mass factorisation follows from Theorem \ref{the:factorisation-conjunctive},
as mass and belief values coincide on singletons.
\end{proof}

This new approach to inference with belief functions, based on the notion of the belief likelihood, opens up a number of interesting avenues: from a more systematic comparison with other inference approaches, to the possible generalisation of the factorisation results obtained above, and to the computation of belief and plausibility likelihoods for other major combination rules.

\subsubsection{Lower and upper likelihoods of Bernoulli trials} \label{sec:bernoulli}

In the case of Bernoulli trials, where not only is there a single binary sample space $\X_i = \X = \{T,F\}$ and conditional independence holds, but also the random variables are assumed to be equally distributed, the conventional likelihood 
reads as $p^k (1-p)^{n-k}$, where $p = P(T)$, $k$ is the number of successes ($T$) and $n$ is as usual the total number of trials.

Let us then compute the lower likelihood function for a series of Bernoulli trials, under the similar assumption that all the belief functions $Bel_{\X_i} = Bel_{\X}$, $i = 1, \ldots, n$, coincide (the analogous of equidistribution), with $Bel_\X$  parameterised by $p=m(\{T\})$, $q = m(\{F\})$ (where, this time, $p +q \leq 1$). 
\begin{corollary}
Under the above assumptions, the lower and upper likelihoods of the sample $\vec{x} = (x_1, \ldots, x_n)$ are, respectively,
\begin{equation} \label{eq:lower-upper-likelihoods-bernoulli}
\begin{array}{lll}
\underline{L}(\vec{x}) & = & \displaystyle \prod_{i=1}^n Bel_\X(\{x_i\})
= p^k q^{n-k},
\\
\overline{L}(\vec{x}) & = & 
\displaystyle \prod_{i=1}^n Pl_\X(\{x_i\}) = (1-q)^k (1-p)^{n-k}.
\end{array}
\end{equation}
\end{corollary}

The above decomposition for $\overline{L}(\vec{x})$, in particular, is valid under the assumption that Conjecture \ref{con:plausibility-likelihood-decomposition-ocap} holds, at least for Bernoulli trials, as the evidence seems to suggest.
After normalisation, these lower and upper likelihoods can be seen as PDFs over the (belief) space $\mathcal{B}$ of all belief functions definable on $\X$ \cite{cuzzolin2008geometric,cuzzolin14lap}.

Having observed a series of trials $\vec{x} = ( x_1, \ldots, x_n )$, $x_i \in \{T,F\}$, one may then seek the belief function on $\X = \{T,F\}$ which best describes the observed sample, i.e., the optimal values of the two parameters $p$ and $q$. In the next subsection we will address a similar problem in a generalised logistic regression setting.

\begin{exam}\hspace{-1.5mm}\emph{\textbf{Lower and upper likelihoods.}} \label{exa:lu-likelihoods}
Figure \ref{fig:lower-upper-likelihoods} plots the lower and upper likelihoods (\ref{eq:lower-upper-likelihoods-bernoulli}) for the case of $k=6$ successes in $n=10$ trials.
Both subsume the traditional likelihood $p^k(1-p)^{n-k}$ as their section for $p+q =1$, although this is particularly visible for the lower likelihood (a). In particular, the maximum of the lower likelihood is the traditional maximum likelihood estimate $p = k/n$, $q = 1 - p$.
This makes sense, for the lower likelihood is highest for the most committed belief functions (i.e., for probability measures).
The upper likelihood (b) has a unique maximum in $p=q=0$: this is the vacuous belief function on $\{T,F\}$, with $m(\{T,F\}) = 1$.

The interval of belief functions joining $\max \overline{L}$ with $\max \underline{L}$ is the set of belief functions such that $\frac{p}{q} = \frac{k}{n-k}$, i.e., those which \emph{preserve the ratio between the observed empirical counts}.
\end{exam}

\begin{figure*}[t!] 
\begin{tabular}{cc}
\includegraphics[width = 0.46\textwidth]{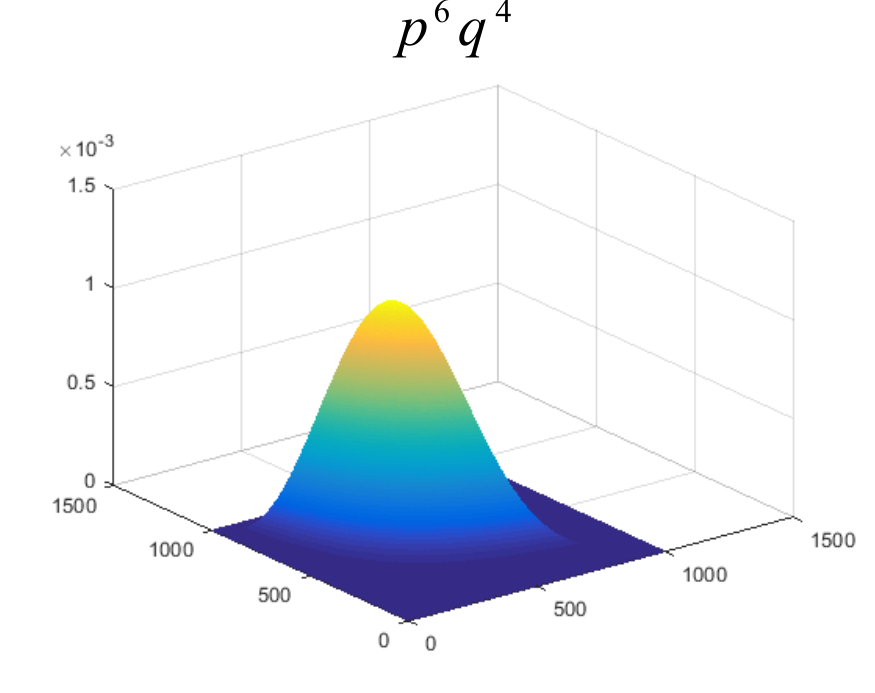}
&
\includegraphics[width = 0.46\textwidth]{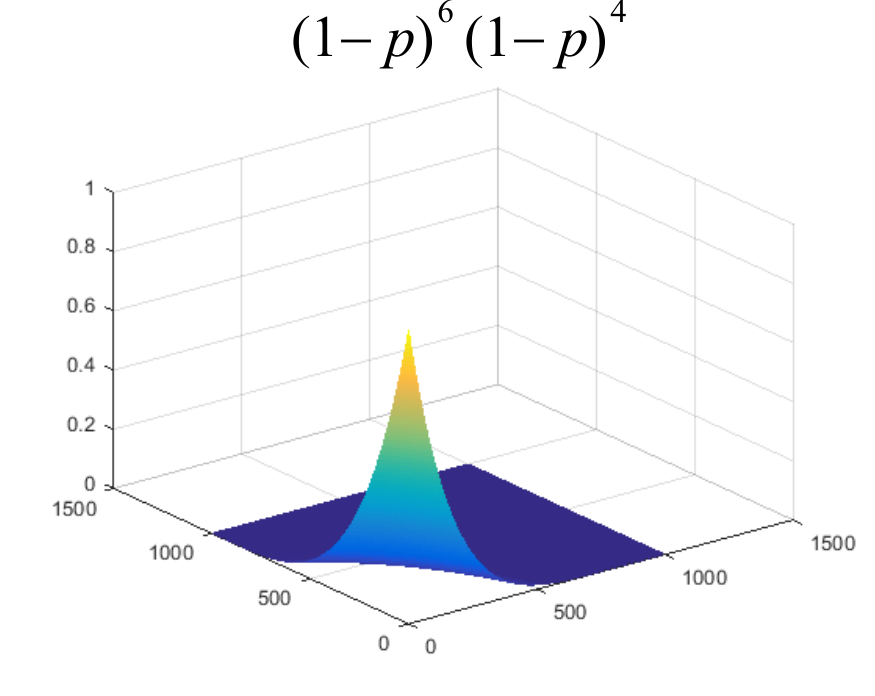}
\\
(a) & (b)
\end{tabular}
\caption{Lower (a) and upper (b) likelihood functions plotted in the space of belief functions defined on the frame $\X = \{ T,F \}$, parameterised by $p = m(T)$ ($X$ axis) and $q = m(F)$ ($Y$ axis), for the case of $k=6$ successes in $n=10$ trials.} \label{fig:lower-upper-likelihoods}
\end{figure*} 

\subsection{Generalised logistic regression} \label{sec:future-logistic} 

\subsubsection{Logistic regression} \label{sec:logistic-regression}

\emph{Logistic regression} allows us, given a sample $Y = \{Y_1,...,Y_n\}$, $X = \{x_1,...,x_n\}$ where $Y_i \in \{0,1\}$ is a binary outcome at time $i$ and $x_i$ is the corresponding observed measurement, to learn the parameters of a conditional probability relation between the two, of the form:
\begin{equation} \label{eq:logistic}
P(Y=1|x)={\frac {1}{1+e^{-(\beta _{0}+\beta _{1}x)}}},
\end{equation}
where $\beta_0$ and $\beta_1$ are two scalar parameters.
Given a new observation $x$, (\ref{eq:logistic}) delivers the probability of a positive outcome $Y=1$.
Logistic regression generalises deterministic linear regression, as it is a function of the linear combination $\beta _{0}+\beta _{1}x$.
The $n$ trials are assumed independent but not equally distributed, for $\pi_i = P(Y_i=1|x_i)$ varies with the time instant $i$ of collection.
The two scalar parameters $\beta_0,\beta_1$ in (\ref{eq:logistic}) are estimated by maximum likelihood of the sample. After denoting by
\begin{equation} \label{eq:logit}
\begin{array}{lll}
\pi_i & = & P(Y_i=1|x_i) = \displaystyle \frac{1}{1+e^{ - (\beta _{0}+\beta _{1}x_i) } }, 
\\ 
1 - \pi_i & = & P(Y_i=0|x_i) = \displaystyle \frac{e^{ - (\beta _{0}+\beta _{1} x_i) } }{1+e^{-(\beta _{0}+\beta _{1}x_i)} }
\end{array}
\end{equation}
the conditional probabilities of the two outcomes, the likelihood of the sample can be expressed as:
$
L(\beta|Y) = \prod_{i=1}^n \pi_i^{Y_i} (1-\pi_i)^{Y_i},
$
where $Y_i \in \{0,1\}$ and $\pi_i$ is a function of $\beta = [\beta_0,\beta_1]$. Maximising $L(\beta|Y)$ yields a conditional PDF $P(Y=1|x)$.

Unfortunately, logistic regression shows clear limitations when the number of samples is insufficient or when there are too few positive outcomes (1s) \cite{King2001}. 
Moreover, inference by logistic regression tends to underestimate the probability of a positive outcome 
\cite{King2001}.

\subsubsection{Generalised formulation}

Based on Theorem \ref{the:belief-likelihood-decomposition-ocap} and Conjecture \ref{con:plausibility-likelihood-decomposition-ocap}, 
we can also generalise logistic regression 
to a belief function setting by replacing the conditional probability $(\pi_i, 1-\pi_i)$ on $\X = \{T,F\}$ with a belief function ($p_i = m(\{T\})$, $q_i = m(\{F\})$, $p_i + q_i \leq 1$) on $2^\X$.
Note that, just as in a traditional logistic regression setting (Section \ref{sec:logistic-regression}), the belief functions $Bel_i$ associated with different input values $x_i$ are not equally distributed.

Writing $T=1, F=0$, the lower and upper likelihoods can be computed as
\[
\begin{array}{c}
\underline{L}(\beta|Y) = \displaystyle \prod_{i=1}^n p_i^{Y_i} q_i^{1-Y_i}, 
\quad
\overline{L}(\beta|Y) = \displaystyle \prod_{i=1}^n (1 - q_i)^{Y_i} (1 - p_i)^{1-Y_i}.
\end{array}
\]
The question becomes how to generalise the logit link between observations $x$ and outputs $y$, in order to seek an analytical mapping between observations and belief functions over a binary frame. Just assuming
\begin{equation}\label{eq:logit3}
p_i = m(Y_i = 1|x_i) = \displaystyle \frac{1}{1+e^{ - (\beta _{0}+\beta _{1}x_i) } }, 
\end{equation}
as in the classical constraint (\ref{eq:logistic}), does not yield any analytical dependency for $q_i$. To address this issue, we can just (for instance) add a parameter $\beta_2$ such that the following relationship holds:
\begin{equation}\label{eq:logit2}
q_i = m(Y_i = 0|x_i) = \beta_2 \frac{e^{-(\beta _{0}+\beta _{1}x_i)}}{1+e^{-(\beta _{0}+\beta _{1}x_i)}}.
\end{equation}
We can then seek lower and upper optimal estimates for the parameter vector $\beta = [\beta_0,\beta_1,\beta_2]'$:
\begin{equation} \label{eq:generalised-logistic}
\arg\max_{\beta} \underline{L} \mapsto \underline{\beta}_0, \underline{\beta}_1, \underline{\beta}_2, 
\quad 
\arg\max_{\beta} \overline{L} \mapsto \overline{\beta}_0, \overline{\beta}_1, \overline{\beta}_2.
\end{equation}
Plugging these optimal parameters into (\ref{eq:logit3}), (\ref{eq:logit2}) will then yield an upper and a lower family of conditional belief functions given $x$ (i.e., an interval of belief functions),
\[
Bel_\X (.|\underline{\beta},x), 
\quad 
Bel_\X (.|\overline{\beta},x).
\]
Given any new test observation $x'$, our generalised logistic regression method will then output a pair of lower and upper belief functions on $\X = \{T,F\} = \{1,0\}$, as opposed to a sharp probability value as in the classical framework. As each belief function itself provides a lower and an upper probability value for each event, both the lower and the upper regressed BFs will provide an interval for the probability $P(T)$ of success, whose width reflects the uncertainty encoded by the training set of sample series.

\subsubsection{Optimisation problem}

Both of the problems (\ref{eq:generalised-logistic}) are constrained optimisation problems (unlike the classical case, where $q_i=1-p_i$ and the optimisation problem is unconstrained). In fact, the parameter vector $\beta$ must be such that
\[
0 \leq p_i + q_i \leq 1 \quad \forall i=1, \ldots, n.
\]
Fortunately, the number of constraints can be reduced by noticing that, under the analytical relations (\ref{eq:logit3}) and (\ref{eq:logit2}), $p_i + q_i \leq 1$ for all $i$ whenever $\beta_2 \leq 1$.

The objective function can be simplified by taking the logarithm. 
In the lower-likelihood case, we get\footnote{Derivations are omitted.}
\[
\log \underline{L}(\beta|Y) = 
\sum_{i=1}^n \Big \{ -\log (1+e^{-(\beta_0+\beta_1 x_i)})
+ (1 - Y_i) \big [ \log \beta_2 - (\beta_0 + \beta_1 x_i) \big ] \Big \}.
\]

We then need to analyse the \emph{Karush--Kuhn--Tucker} (KKT) necessary conditions for the optimality of the solution of a nonlinear optimisation problem
subject to differentiable constraints.
\begin{definition} \label{def:kkt}
Suppose that the objective function ${\displaystyle f:\mathbb {R} ^{n}\rightarrow \mathbb {R} }$  and the constraint functions ${\displaystyle g_{i}:\,\!\mathbb {R} ^{n}\rightarrow \mathbb {R} }$  and ${\displaystyle h_{j}:\,\!\mathbb {R} ^{n}\rightarrow \mathbb {R} }$ of a nonlinear optimisation problem 
\[
\arg\max_x f(x),
\]
subject to
\[
g_i(x)\leq 0 \;\; i=1, \ldots, m, \quad h_j(x)=0 \;\; j = 1, \ldots, l,
\]
are continuously differentiable at a point $x^{*}$.
If $x^{*}$ is a local optimum, under some regularity conditions there then exist constants $\mu _{i}$ $(i=1,\ldots ,m)$ and $\lambda _{j}\ (j=1,\ldots ,l)$, called \emph{KKT multipliers}, such that the following conditions hold:
\begin{enumerate}
\item
\emph{Stationarity}: 
$
\nabla f(x^{*})=\sum _{i=1}^{m}\mu _{i}\nabla g_{i}(x^{*})+\sum _{j=1}^{l}\lambda _{j}\nabla h_{j}(x^{*})$.
\item
\emph{Primal feasibility}: 
$g_{i}(x^{*})\leq 0$ for all $i=1,\ldots ,m$, and $h_{j}(x^{*})=0$, for all $j=1,\ldots ,l$.
\item
\emph{Dual feasibility}:
$\mu _{i}\geq 0$ for all $i=1,\ldots ,m$.
\item
\emph{Complementary slackness}: 
$\mu _{i}g_{i}(x^{*})=0$ for all $i=1,\ldots ,m$.
\end{enumerate}
\end{definition}

For the optimisation problem considered here, we have only inequality constraints, namely
\[
\begin{array}{lll}
\beta_2 \leq 1  & \equiv & g_0 = \beta_2 - 1 \leq 0,
\\
p_i + q_i \geq 0 & \equiv & g_i = - \beta_2 - e^{\beta_0 + \beta_1 x_i} \leq 0,
\quad
i = 1, \ldots, n.
\end{array}
\]
The Lagrangian becomes
\[
\Lambda(\beta) = \log \underline{L}(\beta) + \mu_0 (\beta_2-1) - \sum_{i=1}^n \mu_i (\beta_2 + e^{\beta_0 + \beta_1 x_i}).
\]
The stationarity conditions thus read as $\nabla\Lambda(\beta) = 0$, namely
\begin{equation} \label{eq:kkt-stationarity-logistic}
\left \{
\begin{array}{l}
\displaystyle
\sum_{i=1}^n \Big [ (1-p_i) - (1-Y_i) - \mu_i e^{\beta_0+\beta_1 x_i} \Big ] = 0,
\\
\displaystyle
\sum_{i=1}^n \Big [ (1-p_i) - (1-Y_i) - \mu_i e^{\beta_0+\beta_1 x_i} \Big ] x_i = 0,
\\
\displaystyle
\sum_{i=1}^n \left (\frac{1 - Y_i}{\beta_2} - \mu_i \right ) + \mu_0 = 0,
\end{array}
\right .
\end{equation}
where, as usual, $p_i$ is as in (\ref{eq:logit3}).

The complementary slackness condition reads as follows:
\begin{equation} \label{eq:kkt-slackness-logistic}
\left \{
\begin{array}{l}
\mu_0 (\beta_2 - 1) = 0,
\\
\mu_i (\beta_2 + e^{\beta_0+\beta_1 x_i}) = 0, \quad i = 1, \ldots, n.
\end{array}
\right .
\end{equation}
Gradient descent methods can then be applied to the above systems of equations to find the optimal parameters of our generalised logistic regressor.

\begin{question}
Is there an analytical solution to the generalised logistic regression inference problem, based on the above KKT necessity conditions?
\end{question}

Similar calculations hold for the upper-likelihood problem. A multi-objective optimisation setting in which the lower likelihood is minimised as the upper likelihood is maximised can also be envisaged, in order to generate the most cautious interval of estimates.

\subsubsection{Research questions}

As far as generalised logistic regression is concerned, an analysis of the validity of such a direct extension of logit mapping and the exploration of alternative ways of generalising it are research questions that are potentially very interesting to pursue. 
\begin{question}
What other parameterisations can be envisaged for the generalised logistic regression problem, in terms of the logistic links between the data and the belief functions $Bel_i$ on $\{T,F\}$?
\end{question}

Comprehensive testing of the robustness of the generalised framework on standard estimation problems, including rare-event analysis \cite{Falk04}, needs to be conducted to further validate this new procedure. 

Note that the method, unlike traditional ones, can naturally cope with missing data (represented by vacuous observations $x_i = \X$), therefore providing a robust framework for logistic regression which can deal with incomplete series of observations.

\subsection{The total probability theorem for random sets} \label{sec:future-total}

Spies 
and others have posed the problem of generalising the law of total probability,
\[
P(A) = \sum_{i=1}^N P(A|B_i) P(B_i),
\]
where $\{B_1, \ldots, B_N\}$ is a disjoint partition of the sample space,
to the case of belief functions. They mostly did so with the aim of generalising Jeffrey's combination rule \cite{shafer81jeffrey} -- nevertheless, the question goes rather beyond their original intentions, as it involves understanding the space of solutions to the generalised total probability problem.

\subsubsection{The law of total belief}

The problem of generalising the total probability theorem to belief functions can be posed as follows (Fig. \ref{fig:total-belief}) \cite{zhou2017total}.

\begin{theorem} \emph{(\textbf{Total belief theorem})}. \label{the:total-belief}
Suppose $\Theta$ and $\Omega$ are two frames of discernment, and $\rho:2^{\Omega}\rightarrow2^{\Theta}$ the unique refining between them. Let $Bel_0$ be a belief function defined over $\Omega = \{ \omega_1, \ldots, \omega_{|\Omega|} \}$. Suppose there exists a collection of belief functions ${Bel_i} : 2^{\Pi_i} \rightarrow [0,1]$, where $\Pi = \{ \Pi_1, \ldots, \Pi_{|\Omega|} \}$, $\Pi_i = \rho(\{\omega_i\})$, is the partition of $\Theta$ induced by its coarsening $\Omega$. 

Then, there exists a belief function $Bel : 2^{\Theta} \rightarrow [0,1]$ such that:
\begin{itemize}
\item 
(P1)
$Bel_0$ is the marginal of $Bel$ to $\Omega$, $Bel_0 = Bel \upharpoonright_{\Omega}$; 
\item 
(P2)
$Bel \oplus Bel_{\Pi_i} = Bel_i$ for all $i = 1, \ldots, |\Omega|$, where $Bel_{\Pi_i}$ is the categorical belief function with BPA $m_{\Pi_i} (\Pi_i) =1$, $m_{\Pi_i}(B) = 0$ for all $B\neq \Pi_i$.
\end{itemize}
\end{theorem}

Note that a disjoint partition $\Pi_1, \ldots, \Pi_{|\Omega|}$ of $\Theta$ defines a subalgebra $\mathbb{A}^{\rho}$ of $2^{\Theta}$ as a Boolean algebra with set operations, which is isomorphic to the set algebra $\mathbb{A} = 2^{\Omega}$.
We use the notation $Bel \upharpoonright_{\Omega}$ to denote the marginal of $Bel$ to $\Omega$.

\begin{figure}[ht!]
\begin{center}
\includegraphics[width = \textwidth]{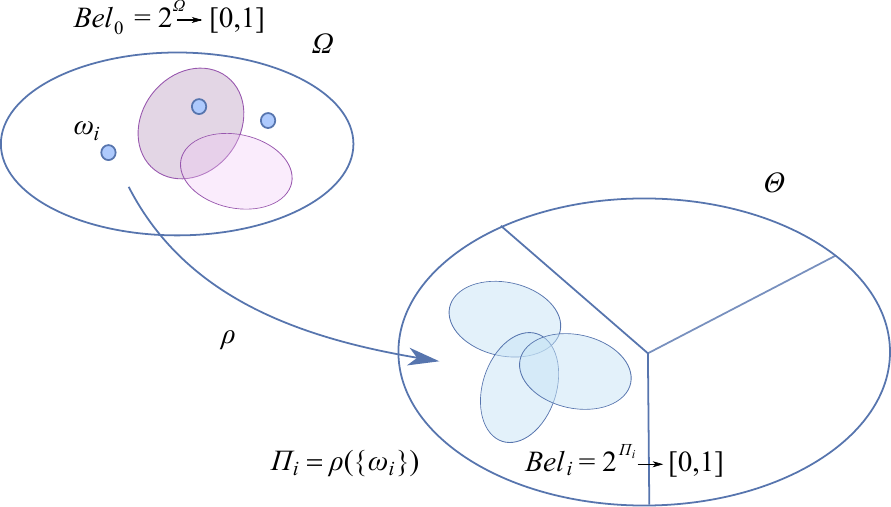} 
\caption{Pictorial representation of the hypotheses of the total belief theorem (Theorem \ref{the:total-belief}).} \label{fig:total-belief}
\end{center}
\end{figure}

\subsubsection{Proof of existence} \label{sec:proof-of-existence}

Theorem \ref{the:total-belief}'s proof makes use of the two lemmas that follow. 
Here, basically, we adapt Smets's original proof of the generalised Bayesian theorem \cite{smets93belief} to the refinement framework presented here, which is more general than his multivariate setting, in which only Cartesian products of frames are considered. As usual, we denote by $m_0$ and $m_i$ the mass functions of $Bel_0$ and $Bel_i$, respectively. Their sets of focal elements are denoted by $\mathcal{E}_{\Omega}$ and $\mathcal{E}_i$, respectively.

Suppose that $\Theta' \supseteq\Theta$ and $m$ is a mass function over $\Theta$.  The mass function $m$ can be identified with a mass function $\overrightarrow{m}_{\Theta'}$ over the larger frame $\Theta'$: for any $E'\subseteq \Theta'$:
\[
\overrightarrow{m}_{\Theta'}(E') = \left \{ 
\begin{array}{ll}
   m(E)  & E' = E \cup (\Theta' \setminus \Theta ) \\
   0 & \text{otherwise}.
\end{array}
\right .
\]
Such an $\overrightarrow{m}_{\Theta'}$ is called the \emph{conditional embedding} of $m$ into $\Theta'$.  When the context is clear, we can drop the subscript $\Theta'$. It is easy to see that conditional embedding is the inverse of Dempster conditioning.

For the collection of belief functions  $Bel_i: 2^{\Pi_i}\rightarrow [0,1]$, let  $\overrightarrow{Bel_i}$ be the conditional embedding of $Bel_i$ into $\Theta$ and let $\overrightarrow{Bel}$ denote the Dempster combination of all $\overrightarrow{Bel_i}$, i.e., $\overrightarrow{Bel} = \overrightarrow{Bel_1}\oplus \cdots \oplus \overrightarrow{Bel_{|\Omega|}}$, with mass function $\overrightarrow{m}$. 

\begin{lemma}\label{the:ConditionalBeliefs}
The belief function $\overrightarrow{Bel}$ over $\Theta$ satisfies the following two properties: (1) each focal element $\overrightarrow{e}$ of  $\overrightarrow{Bel}$ is the union of exactly one focal element $e_i$ of each conditional belief function $Bel_i$; (2) the marginal $\overrightarrow{Bel}\upharpoonright_{\Omega}$ on $\Omega$ is the vacuous belief function over $\Omega$.
\end{lemma}

\begin{proof}  Each focal element $\overrightarrow{e_i}$ of $\overrightarrow{Bel_i}$ is of the form $(\bigcup_{j\neq i} \Pi_j) \cup e_i$, where $e_i$
is some focal element of $Bel_i$.  
In other words, $\overrightarrow{e_i}= (\Theta \setminus \Pi_i) \cup e_i $.  Since $\overrightarrow{Bel}$ is the Dempster combination of the $\overrightarrow{Bel_i}$s, it is easy to see that each focal element $\overrightarrow{e}$ of
$\overrightarrow{Bel}$ is the union of exactly one focal element $e_i$ from each conditional belief function $Bel_i$.  In other words, $\overrightarrow{e} = \bigcup_{i=1}^{|\Omega|}e_i $, where $e_i \in \mathcal{E}_i$, and condition (1) is proven.

Let $\overrightarrow{\mathcal{E}}$ denote the set of all the focal elements of $\overrightarrow{Bel}$, namely
	\begin{center} $\overrightarrow{\mathcal{E}} = \Big \{\overrightarrow{e}\subseteq \Theta: \overrightarrow{e} = \bigcup_{i=1}^{|\Omega|} e_i $ where $e_i$ is a focal element of $Bel_i \Big \}$.
	\end{center}

\noindent Note that $e_i$'s coming from different conditional belief functions $Bel_i$'s are disjoint. For each $\overrightarrow{e}\in \overrightarrow{\mathcal{E}}$, $\bar{\rho}(\overrightarrow{e}) =\Omega$ (where $\bar{\rho}$ denotes the outer reduction (\ref{eq:outer-reduction})). 
It follows from (\ref{eq:marginal}) that $\overrightarrow{m}\upharpoonright_{\Omega}(\Omega) =1$ -- hence the marginal of $\overrightarrow{Bel}$ on $\Omega$ is the vacuous belief function there.
\end{proof}

Let $Bel_0^{\uparrow\Theta}$ be the vacuous extension of $Bel_0$ from $\Omega$ to $\Theta$. We define the desired total function $Bel$ to be the Dempster combination of $Bel_0^{\uparrow\Theta}$ and $\overrightarrow{Bel}$, namely
\begin{equation} \label{eq:total-belief-function}
Bel \doteq Bel_0^{\uparrow\Theta} \oplus \overrightarrow{Bel}.
\end{equation}

\begin{lemma} \label{lem:uai-2}
The belief function $Bel$ defined in (\ref{eq:total-belief-function}) on the frame $\Theta$ satisfies the following two properties:
\begin{enumerate}
	\item $Bel \oplus Bel_{\Pi_i} = Bel_i$ for all $i=1, \ldots, |\Omega|$, where $Bel_{\Pi_i}$ is again the categorical belief function focused on $\Pi_i$.
	\item $Bel_0$ is the marginal of $Bel$ on $\Omega$, i.e., $Bel_0 = Bel \upharpoonright_{\Omega}$.
\end{enumerate}
That is, $Bel$ is a valid total belief function.
\end{lemma} 
	
\begin{proof}  Let $\overrightarrow{m}$ and $m_i$ be the mass functions corresponding to $\overrightarrow{Bel}$ and $Bel_i$, respectively. For each $\overrightarrow{e} = \bigcup_{i=1}^{|\Omega|} e_i \in \overrightarrow{\mathcal{E}}$, where $e_i \in \mathcal{E}_i$, we have that 
\[
\overrightarrow{m}(\overrightarrow{e}) = \prod_{i=1}^{|\Omega|} m_i(e_i).
\]
Let $\mathcal{E}^{\uparrow{\Theta}}_{\Omega}$ denote the set of focal elements of $Bel_0^{\uparrow\Theta}$. 
Since $Bel_0^{\uparrow\Theta}$ is the vacuous extension of $Bel_0$, $\mathcal{E}_{\Omega}^{\uparrow\Theta} = \{\rho(e_{\Omega}): e_{\Omega}\in \mathcal{E}_{\Omega} \}$.   
Each element of $\mathcal{E}^{\uparrow{\Theta}}_{\Omega}$ is actually the union of some equivalence classes $\Pi_i$ of the partition $\Pi$. 
Since each focal element of $Bel_0^{\uparrow\Theta}$ has non-empty intersection with all the focal elements $\overrightarrow{e} \in \overrightarrow{\mathcal{E}}$, it follows that
\begin{align} \label{normalization}
\sum_{e_{\Omega}\in \mathcal{E}_{\Omega}, \overrightarrow{e} \in \overrightarrow{\mathcal{E}},\rho(e_{\Omega}) \cap \overrightarrow{e} \neq \emptyset}  m_0^{\uparrow\Theta}(\rho(e_{\Omega}))\overrightarrow{m}(\overrightarrow{e}) =1. 
\end{align}
Thus, the normalisation factor in the combination $Bel_0^{\uparrow\Theta}\oplus \overrightarrow{Bel}$ is equal to 1.  

Now, let $\mathcal{E}$ denote the set of focal elements of the belief function $Bel = Bel_0^{\uparrow\Theta}\oplus \overrightarrow{Bel}$.  
By the Dempster sum (\ref{eq:dempster}), each element $e$ of $\mathcal{E}$ is the union of focal elements of some conditional belief functions $Bel_i$, i.e., $e = e_{j_1} \cup e_{j_2}\cup \cdots \cup e_{j_K}$ for some $K$ such that $\{j_1, \ldots, j_K\} \subseteq \{1, \ldots, |\Omega|\}$ and $e_{j_l}$ is a focal element of $Bel_{j_l}$, $1\leq l \leq K$.  Let $m$ denote the mass function for $Bel$. 
For each such $e \in \mathcal{E}$, $e = \rho(e_{\Omega}) \cap \overrightarrow{e}$ for some $e_{\Omega}\in \mathcal{E}_{\Omega}$ and $\overrightarrow{e} \in \overrightarrow{\mathcal{E}}$, so that $e_{\Omega} = \bar{\rho}(e)$.
Thus we have
\begin{equation} \label{mass_total}
\begin{array}{lll}
& & (m_0^{\uparrow\Theta} \oplus  \overrightarrow{m} )(e) 
\\ \\
& = &  \displaystyle
\sum_{e_{\Omega}\in \mathcal{E}_{\Omega}, \overrightarrow{e} \in \overrightarrow{\mathcal{E}},\rho(e_{\Omega}) \cap \overrightarrow{e} = e} m_0^{\uparrow\Theta}(\rho(e_{\Omega})) \overrightarrow{m}(\overrightarrow{e})
=
\sum_{
\stackrel{\mathlarger{
e_{\Omega}\in \mathcal{E}_{\Omega}, \overrightarrow{e} \in \overrightarrow{\mathcal{E}}}
}{ 
\rho(e_{\Omega}) \cap \overrightarrow{e} = e
}
} m_0(e_{\Omega}) \overrightarrow{m}(\overrightarrow{e})  
\\
& = & \displaystyle
m_0(\bar{\rho}(e))  \sum_{ \overrightarrow{e} \in \overrightarrow{\mathcal{E}},\rho(\bar\rho(e)) \cap \overrightarrow{e} = e}\overrightarrow{m}(\overrightarrow{e})  
\\
& = & \displaystyle
m_0(\bar{\rho}(e)) \cdot m_{j_1} (e_{j_1})\cdots m_{j_K}(e_{j_K}) \prod_{j \not \in \{j_1, \ldots, j_K\}}\sum_{e\in \mathcal{E}_j}m_j(e) 
\\
& = & \displaystyle
m_0(\bar{\rho}(e)) \prod_{k=1}^K m_{j_k} (e_{j_k}),
\end{array}
\end{equation}
as $\overrightarrow{m}(\overrightarrow{e}) = \prod_{i=1}^n m_i(e_i)$ whenever $\overrightarrow{e} = \cup_{i=1}^n e_i$. 

Without loss of generality, we will consider the conditional mass function $m(e_1 | \Pi_1)$, where $e_1$ is a focal element of $Bel_1$ and $\Pi_1$ is the first partition class associated with the partition $\Pi$, and show that $m(e_1 |\Pi_1) = m_1(e_1)$.  In order to obtain $m(e_1 | \Pi_1)$, which is equal to 
\[
\sum_{e\in \mathcal{E}, e\cap \Pi_1=e_1} m(e) / Pl(\Pi_1), 
\]
in the following we separately compute $\sum_{e\in \mathcal{E}, e\cap \Pi_1=e_1}m(e)$ and $Pl(\Pi_1)$.

For any $e \in \mathcal{E}$, if $ e \cap \Pi_1 \neq \emptyset$, $\bar{\rho}(e)$ is a subset of $\Omega$ including $\omega_1$. Therefore,
\begin{equation} \label{eq:denominator}
\begin{array}{lll}
& & Pl(\Pi_1) 
\\
& = & \displaystyle
\sum_{e\in \mathcal{E}, e \cap \Pi_1 \neq \emptyset} m(e) 
= 
\sum_{ \mathcal{C} \subseteq \{\Pi_2, \ldots, \Pi_{|\Omega|}\}} \left ( \sum_{\rho(\bar{\rho}(e)) = 
\Pi_1 \cup (\bigcup_{E\in  \mathcal{C}} E)} m(e) \right ) 
\\ \\
& = & \displaystyle
\sum_{ \mathcal{C} \subseteq \{\Pi_2, \ldots, \Pi_{|\Omega|}\}} m_0^{\uparrow\Theta} \left ( \Pi_1\cup \bigcup_{E\in \mathcal{C}} E \right ) 
\left ( \sum_{e_1\in \mathcal{E}_1}m_1(e_1) \prod_{\Pi_l\in \mathcal{C}} 
\sum_{e_l\in\mathcal{E}_l}m_l(e_l) \right ) 
\\ \\
& = & \displaystyle
\sum_{ \mathcal{C} \subseteq \{\Pi_2, \ldots, \Pi_{|\Omega|}\}} m_0^{\uparrow\Theta} \left (\Pi_1\cup \bigcup_{E\in \mathcal{C}} E \right ) 
= 
\sum_{e_{\Omega}\in \mathcal{E}_{\Omega}, \omega_1\in e_{\Omega}} m_0(e_{\Omega}) 
= 
Pl_0 (\{\omega_1\}).  	
\end{array}
\end{equation} 
Similarly,
\begin{align}
\sum_{\stackrel{\mathlarger{e \in \mathcal{E}}}{e\cap \Pi_1 = e_1}} m(e)
& 
=  
\sum_{ \mathcal{C} \subseteq \{\Pi_2, \ldots, \Pi_{|\Omega|}\}} \sum_{\rho(\bar{\rho}(e)) 
= 
\bigcup_{E\in  \mathcal{C}} E} m(e_1 \cup e)
\nonumber 
\\
& 
=  
m_1(e_1) \sum_{ \mathcal{C} \subseteq \{\Pi_2, \ldots, \Pi_{|\Omega|}\}} m_0^{\uparrow\Theta} \left ( \Pi_1\cup \bigcup_{E\in \mathcal{C}} E \right ) \prod_{\Pi_l\in \mathcal{C}}  \sum_{e_l\in\mathcal{E}_l}m_l(e_l)
\nonumber 
\\
& 
=  
m_1(e_1) \sum_{ \mathcal{C} \subseteq \{\Pi_2, \ldots, \Pi_{|\Omega|}\}} m_0^{\uparrow\Theta} \left ( \Pi_1\cup \bigcup_{E\in \mathcal{C}} E \right ) 
\nonumber 
\\
& = m_1(e_1)\sum_{e_{\Omega}\in \mathcal{E}_{\Omega}, \omega_1\in e_{\Omega}} m_0(e_{\Omega}) = m_1(e_1) Pl_0 (\{\omega_1\}). \label{eq:numerator} 
\end{align} 

From (\ref{eq:denominator}) and (\ref{eq:numerator}), it follows that 
\[
m(e_1|\Pi_1) = \frac{\sum_{e\in \mathcal{E}, e\cap \Pi_1=e_1}m(e)}{Pl(\Pi_1)} = m_1(e_1).  
\]
This proves property 1. Proving 2 is much easier.

For any $e_{\Omega} \doteq \{\omega_{j_1}, \ldots, \omega_{j_K}\}\in \mathcal{E}_{\Omega}$,
\begin{equation}
\begin{array}{lll}
m\upharpoonright_{\Omega} (e_{\Omega}) & = & \displaystyle
\sum_{\overline{\rho}(e) = e_{\Omega}} m(e) 
=  m_0^{\uparrow \Theta}(\rho(e_{\Omega}))  \prod_{l=1}^K\sum_{e\in \mathcal{E}_{j_l}} m_{j_l}(e) =  m_0^{\uparrow \Theta}(\rho(e_{\Omega}))
\\
& = & m_0(e_{\Omega}).
\end{array}
\end{equation}
It follows that $Bel \upharpoonright_{\Omega} = Bel_0$, and hence the thesis.
\end{proof} 

The proof of Theorem \ref{the:total-belief} immediately follows from Lemmas \ref{the:ConditionalBeliefs} and \ref{lem:uai-2}.

\begin{exam}\label{exam:total_belief}

Suppose that the coarsening $\Omega:=\{\omega_1, \omega_2, \omega_3\}$ considered induces a partition $\Pi$ of $\Theta$: $\{\Pi_1, \Pi_2, \Pi_3\}$. Suppose also that the conditional belief function $Bel_1$ considered, defined on $\Pi_1$, has two focal elements $e_1^1$ and $e_1^2$, that the conditional belief function $Bel_2$ defined on $\Pi_2$ has a single focal element $e_2^1$ and that $Bel_3$, defined on $\Pi_3$, has two focal elements $e_3^1$ and $e_3^2$ (see Fig. \ref{fig:case-study}). 
\begin{figure}[ht!] 
\begin{center}
\includegraphics[width = 0.7 \textwidth]{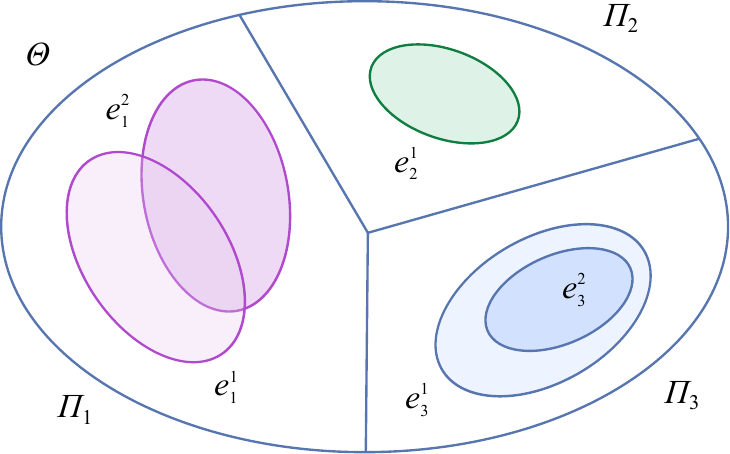}
\caption{The conditional belief functions considered in our case study. The set-theoretical relations between their focal elements are immaterial to the solution. \label{fig:case-study}}
\end{center}
\end{figure} 
According to Lemma \ref{the:ConditionalBeliefs}, the Dempster combination $\overrightarrow{Bel}$ of the conditional embeddings of the $Bel_i$'s has four focal elements, listed below:
\begin{equation} \label{eq:elastic-bands}
\begin{array}{lll}
e_1 & = & e_1^1 \cup e_2^1 \cup e_3^1, \quad  e_2 = e_1^1 \cup e_2^1 \cup e_3^2, \\
e_3 & = & e_1^2 \cup e_2^1 \cup e_3^1, \quad 	e_4 = e_1^2 \cup e_2^1 \cup e_3^2.
\end{array}
\end{equation}
\noindent The four focal total elements can be represented as `elastic bands' as in Fig. \ref{fig:focal-elements}. 

\begin{figure}[ht!]
\begin{center}
\includegraphics[width = 0.8\textwidth]{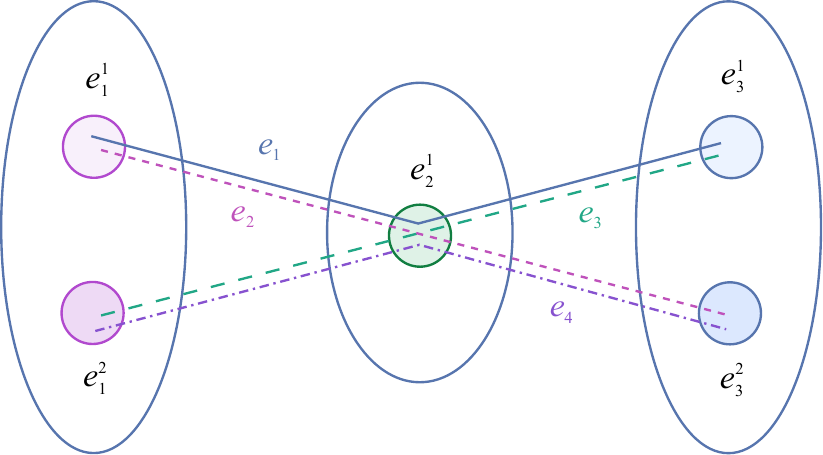}
\end{center}
\caption{Graphical representation of the four possible focal elements (\ref{eq:elastic-bands}) of the total belief function (\ref{eq:total-belief-function}) in our case study. \label{fig:focal-elements}} 
\end{figure} 

Without loss of generality, we assume that a prior $Bel_0$ on $\Omega$ has each subset of $\Omega$  as a focal element, i.e., $\mathcal{E}_{\Omega} = 2^{\Omega}$.    
It follows that each focal element $e$ of the total belief function $Bel \doteq \overrightarrow{Bel} \oplus Bel_0^{\uparrow \Theta} $ is the union of some focal elements from \emph{different} conditional belief functions $Bel_i$. 
So, the set $\mathcal{E}$ of the focal elements of $Bel$ is  
\[
\mathcal{E} = \{ e = \bigcup_{1\leq i \leq I} e_i:  1\leq I \leq 3, e_i\in \mathcal{E}_i \}
\]
and is the union of the following three sets:
	\begin{align}
		\mathcal{E}_{I=1}: &=\mathcal{E}_1 \cup \mathcal{E}_2\cup \mathcal{E}_3, 
		\nonumber
		\\
		\mathcal{E}_{I=2}: &=\{e\cup e': (e,e')\in \mathcal{E}_i\times \mathcal{E}_j, 1\leq i,j\leq 3,i\neq j \}, 
		\nonumber
		\\
		\mathcal{E}_{I=3}: &=\{e_1\cup e_2\cup e_3: e_i\in \mathcal{E}_i, 1\leq i\leq 3\}.\nonumber
	\end{align}
Hence $|\mathcal{E}| = 5+ 8+ 4= 17$.
  According to (\ref{mass_total}), it is very easy to compute the corresponding total mass function $m$.  For example, for the two focal elements $e_1^1 \cup e_3^2$ and $e_1^1 \cup e_2^1 \cup e_3^2$, we have
	\begin{align}
		m(e_1^1 \cup e_3^2) &= m_0(\{w_1, w_3\}) m_1(e_1^1) m_3(e_3^2), 
		\nonumber 
		\\
		m(e_1^1 \cup e_2^1 \cup e_3^2) &= m_0(\Omega) m_1(e_1^1) m_2(e_2^1) m_3(e_3^2). \nonumber
	\end{align}

\end{exam}

\subsubsection{Number of solutions} 

The total belief function $Bel$ obtained in Theorem \ref{the:total-belief} is not the only solution to the total belief problem (note that the Dempster combination (\ref{eq:total-belief-function}) is itself unique, as the Dempster sum is unique).

Assume that $Bel^*$ is a total belief function posssessing the two properties in Theorem \ref{the:total-belief}. Let $m^*$ and $\mathcal{E}^*$ denote its mass function and the set of its focal elements, respectively.  Without loss of generality, we may still assume that the prior $Bel_0$ has every subset of $\Omega$ as its focal element, i.e., $\mathcal{E}_{\Omega}= 2^{\Omega}$.  From property (P2), 
\[
Bel^* \oplus Bel_{\Pi_i} = Bel_i, \quad 1 \leq i \leq |\Omega|, 
\]
we derive that each focal element of $Bel^*$ must be a union of focal elements of a number of conditional belief functions $Bel_i$.  
For,  if $e^*$ is a focal element of $Bel^*$ and $e^* = e_l \cup e'$, where $\emptyset \neq e_l \subseteq \Pi_l$ and $e' \subseteq \Theta \setminus \Pi_l$ for some $1\leq l \leq |\Omega|$, then $m_l(e_l) =( m^* \oplus m_{\Pi_l})(e_l) >0$ and hence $e_l \in \mathcal{E}_l$.  So we must have that $\mathcal{E}^* \subseteq \mathcal{E}$, where $\mathcal{E}$ is the set of focal elements of the total belief function $Bel$ (\ref{eq:total-belief-function}) obtained in Theorem \ref{the:total-belief}, 
\[
\mathcal{E}= \left \{\bigcup_{j\in J}e_j: J \subseteq \{1, \ldots, |\Omega|\}, e_j \in \mathcal{E}_j\right \}.
\] 

In order to find $Bel^*$ (or $m^*$), we need to solve a group of linear equations which correspond to the constraints dictated by the two properties, in which the mass $m^*(e)$ of each focal element $e\in \mathcal{E}$ of the total solution (\ref{eq:total-belief-function}) is treated as an unknown variable. There are $|\mathcal{E}|$ variables in the group. 

From properties (P1) and (P2), we know that $Pl_0(\omega_i) = Pl^*(\Pi_i)$, $1 \leq i \leq |\Omega|$, where $Pl_0$ and $Pl^*$ are the plausibility functions associated with $Bel_0$ and $Bel^*$, respectively.
In addition, property (P1) implies the system of linear constraints
\begin{equation} \label{eq:constraints-p1}
\left \{
\sum_{e \cap \Pi_i = e_i, e \in \mathcal{E}} m^*(e) = m_i(e_i) Pl_0(\omega_i), 
\quad
\forall i=1, \ldots, n, \; \forall e_i \in \mathcal{E}_i.
\right.	
\end{equation}
The total number of such equations is $\sum_{j=1}^{|\Omega|}|\mathcal{E}_j|$.  Since, for each $1\leq i \leq |\Omega|$, $\sum_{e \in \mathcal{E}_i} m_i(e) =1$, the
system (\ref{eq:constraints-p1}) includes a group of $\sum_{j=1}^{|\Omega|}|\mathcal{E}_j| - |\Omega|$ \emph{independent} linear equations, which we denote as $G_1$. 

From property (P2) (the marginal of $Bel^*$ on $\Omega$ is $Bel_0$), we have the following constraints:
\begin{equation} \label{eq:constraints-p2}
\left \{
\sum_{e \in \mathcal{E}, \overline{\rho}(e) =C } m^*(e) = m_0(C), 
\quad 
\forall \emptyset \neq C \subseteq \Omega.
\right.	
\end{equation}
The total number of linear equations in (\ref{eq:constraints-p2}) is the number of non-empty subsets of $\Omega$. Since $\sum_{C \subseteq \Omega} m_0(C)=1$, there is a subset of $|2^{\Omega}|-2$ {independent} linear equations in (\ref{eq:constraints-p2}), denoted by $G_2$. 

As the groups of constraints $G_1$ and $G_2$ are independent, 
their union $G:=G_1 \cup G_2 $ completely specifies properties (P1) and (P2) in Theorem \ref{the:total-belief}.  Since $|G_1| =\sum_{j=1}^{|\Omega|}|\mathcal{E}_j| - |\Omega| $ and $|G_2| = |2^{\Omega}|-2$, the cardinality of the union group $G$ is $|G|=\sum_{j=1}^{|\Omega|}|\mathcal{E}_j| - |\Omega| +  |2^{\Omega}|-2 $.  

From Theorem \ref{the:total-belief}, we know that the system of equations $G$ is solvable and has at least one positive solution (in which each variable has a positive value). This implies that $|\mathcal{E}| \geq |G|$, i.e., the number of variables must be no less than that of the independent linear equations in $G$.  If $|\mathcal{E}| > |G|$, in particular, we can apply the Fourier--Motzkin elimination method \cite{Khachiyan2009} to show that $G$ 
has another distinct positive solution $Bel^*$ (i.e., such that $m^*(e)\geq 0$ $\forall e\in \mathcal{E}$). 

\begin{exam}\label{exam:non-unique}  We employ Example \ref{exam:total_belief} to illustrate the whole process of finding such an alternative total belief function $Bel^*$. 
We assume further that $m_0$ and $m_i$, $1 \leq i \leq 3$, take values
\[
\begin{array}{c}
\displaystyle
m_1(e_1^1)  = \frac{1}{2} = m_1(e_1^2), \quad m_2(e_2^1) = 1, \quad m_3(e_3^1) = \frac{1}{3}, \quad m_3(e_3^2) = \frac{2}{3},
\\ \\
\displaystyle
m_0(\{\omega_1\}) = m_0(\{\omega_2\}) = m_0(\{\omega_3\}) = \frac{1}{16}, 
\quad 
m_0(\{\omega_1, \omega_2\}) = \frac{2}{16},
\\ \\
\displaystyle
m_0(\{\omega_2, \omega_3\}) = \frac{4}{16}, 
\quad 
m_0(\{\omega_1,\omega_3\}) = \frac{3}{16}, 
\quad 
m_0(\Omega) =\frac{1}{4}. 
\end{array}
\] 
If we follow the process prescribed above to translate the two properties into a group $G$ of linear equations, we obtain 17 unknown variables $m^*(e)$, $e \in \mathcal{E}$ ($|\mathcal{E}| =17$) and eight independent linear equations ($|G| =8$).  From Theorem \ref{the:total-belief}, we can construct a positive solution $m$ defined according to (\ref{mass_total}). For this example,
	\begin{align}
		m (\{e_1^1,e_2^1, e_3^1\}) & = m_0(\Omega) \; m_1(e_1^1) \; m_2(e_2^1) \; m_3(e_3^1) = \frac{1}{24},\nonumber
		\\
		m (\{e_1^2,e_2^1, e_3^2\}) & = m_0(\Omega) \; m_1(e_1^2) \; m_2(e_2^1) \; m_3(e_3^2) = \frac{1}{12}.
		\nonumber
	\end{align}
When solving the equation group $G$ via the Fourier--Motzkin elimination method, we choose $m^*(\{e_1^1,e_2^1, e_3^1\})$ and $m^* (\{e_1^2,e_2^1, e_3^2\})$ to be the last two variables to be eliminated. 
Moreover, there is a sufficiently small positive number $\epsilon$ such that 
\[
m^*(\{e_1^1,e_2^1, e_3^1\})=\frac{1}{24}-\epsilon >0, \quad m^* (\{e_1^2,e_2^1, e_3^2\})= \frac{1}{12} +\epsilon, 
\]
and all other variables also take positive values.  
It is easy to see that the $m^*$ so obtained is different from the $m$ obtained via Theorem \ref{the:total-belief}. 
\end{exam}

Clearly, this new solution has the same focal elements as the original one: other solutions may exist with a different set of focal elements (as can indeed be empirically shown). We will discuss this issue further in a couple of pages. 

\subsubsection{Case of Bayesian prior} 

Note that when the prior belief function $Bel_0$ is Bayesian, however, the total belief function obtained according to (\ref{mass_total}) is the unique one satisfying the two properties in Theorem \ref{the:total-belief}. 
\begin{corollary} \label{cor:total-belief-bayesian} For the belief function $Bel_0$ over $\Omega$ and the conditional belief functions $Bel_i$ over $\Pi_i$ in Theorem \ref{the:total-belief}, if $Bel_0$ is Bayesian (i.e., a probability measure), then there exists a unique total belief function $Bel: 2^{\Theta}\rightarrow [0,1]$ such that (P1) and (P2) are satisfied.
Moreover, the total mass function $m$ of $Bel$ is
\begin{equation*}
m(e)= \left\{
\begin{array}{ll}
 m_i(e) \; m_0(\omega_i)  & \quad \text{if } e \in \mathcal{E}_i \text{ for some } i,\\
0 & \hspace{5mm} \text{otherwise}.
\end{array} \right.
\end{equation*}
\end{corollary} 

\begin{proof}  It is easy to check that the total mass function $m$   defined above satisfies the two properties.   Now we need to show that it is unique.  
Since $Bel_0$ is Bayesian, 
$\mathcal{E}_{\Omega} \subseteq \{ \{ \omega_1\}, \ldots, \{ \omega_{|\Omega|} \} \}$. In other words, all focal elements of $Bel_0$ are singletons.  
It follows that $\mathcal{E}^{\uparrow\Theta}_{\Omega} \subseteq \{\Pi_1, \ldots, \Pi_{|\Omega|}\}$. 
If $e$ is a focal element of $Bel$, we obtain from (\ref{mass_total}) that $e \in \mathcal{E}_i$ for some $i \in \{1, \ldots, |\Omega|\}$ and $m(e)= m_0(\omega_i) m_i(e)$. 
This implies that, if $e\not \in \bigcup_{i=1}^{|\Omega|} \mathcal{E}_i$, i.e., $e$ is not a focal element of any conditional belief function $Bel_i$, $ m(e) =0$.  So we have shown that the total mass function $m$ is the unique one satisfying (P1) and (P2).
\end{proof}

\subsubsection{Generalisation}

If the first requirement (P1) is modified to include conditional constraints with respect to \emph{unions} of equivalence classes, the approach used to prove Theorem \ref{the:total-belief} is no longer valid.  

For each non-empty subset $\{i_1, \ldots, i_J\} \subseteq \{1, \ldots, |\Omega|\}$, let 
\[
Bel_{\bigcup_{j=1}^J \Pi_{i_j}}, \quad Bel_{i_1\cdots i_J} 
\]
denote the categorical belief function with $\bigcup_{j=1}^J \Pi_{i_j}$ as its only focal element, and a conditional belief function on $\bigcup_{j=1}^J \Pi_{i_j}$, respectively. 
\\
We can then introduce a new requirement by generalising property (P1) as follows: 
\begin{itemize}
\item  $(P1'): Bel \oplus Bel_{\bigcup_{j=1}^J \Pi_{i_j}} = Bel_{i_1 \cdots i_J}$ $\forall \emptyset \neq \{i_1, \ldots, i_J\} \subseteq \{1, \ldots, |\Omega|\}$.
\end{itemize} 
Let $\overrightarrow{Bel}_{i_1\cdots i_J}$ denote the conditional embeddings of $Bel_{i_1\cdots i_J}$, and $\overrightarrow{Bel}_\text{new}$ the Dempster combination of all these conditional embeddings. Let $Bel_\text{new} = \overrightarrow{Bel}_\text{new} \oplus Bel_0^{\uparrow \Theta}$.
It is easy to show that $Bel_\text{new}$ satisfies neither $(P1')$ nor $(P2)$. 

\subsubsection{Relation to generalised Jeffrey's rules} \label{sec:total-jeffrey} 

In spirit, our approach in this chapter is similar to Spies's Jeffrey's rule for belief functions in \cite{spies94conditional}. His total belief function is also the Dempster combination of the prior on the subalgebra generated by the partition $\Pi$ of $\Theta$ with the conditional belief functions on all the equivalence classes $\Pi_i$. 
Moreover, he showed that this total belief function satisfies the two properties in Theorem \ref{the:total-belief}. However, his definition of the conditional belief function is different from the one used here, which is derived from Dempster's rule of combination. His definition falls within the framework of random sets, in which a conditional belief function is a \emph{second-order} belief function whose focal elements are conditional events, defined as sets of subsets of the underlying frame of discernment.
The biggest difference between Spies's approach and ours is thus that his framework depends on probabilities, whereas ours does not. It would indeed be interesting to explore the connection of our total belief theorem with his Jeffrey's rule for belief functions.  

Smets  \cite{smets93jeffrey} also generalised Jeffrey's rule within the framework of models based on belief functions, without relying on probabilities.  Recall that $\rho$ is a refining mapping from $2^{\Omega}$ to $2^{\Theta}$, and $\mathbb{A}^{\rho}$ is the Boolean algebra generated by the set of equivalence classes $\Pi_i$ associated with $\rho$. Contrary to our total belief theorem, which assumes conditional constraints only with respect to the equivalence classes $\Pi_i$ (the atoms of $\mathbb{A}^{\rho}$), Smets's generalised Jeffrey's rule considers constraints with respect to unions of equivalence classes, i.e., arbitrary elements of $\mathbb{A}^{\rho}$ (see the paragraph `Generalisation' above). 

Given two belief functions $Bel_1$ and $Bel_2$ over $\Theta$, his general idea is to find a BF $Bel_3$ there such that:
\begin{itemize}
	\item (Q1) its marginal on $\Omega$ is the same as that of $Bel_1$, i.e., $Bel_3 \upharpoonright_{\Omega} = Bel_1 \upharpoonright_{\Omega}$;
	\item (Q2) its conditional constraints with respect to elements of $\mathbb{A}^{\rho}$ are the same as those of $Bel_2$. 
\end{itemize} 
In more detail, let $m$ be a mass function over $\Theta$. Smets defines two kinds of conditioning for conditional constraints: for any $E \in \mathbb{A}^{\rho}$ and $e\subseteq E$ such that $\rho(\bar{\rho}(e)) = E$,
\[ 
m^\text{in} ( e | E ) \doteq \frac{m(e)}{\sum_{\rho(\bar{\rho}(e')) = E} m(e')},
\quad
m^\text{out} ( e | E ) \doteq \frac{m(e |E)}{\sum_{\rho(\bar{\rho}(e')) = E} m(e'|E )}.
\]
The first one is the well-known geometric conditioning \cite{smets93deductibility}, whereas the second is called \emph{outer conditioning}. Both are distinct from Dempster's rule of conditioning. From these two conditioning rules, he obtains two different forms of generalised Jeffrey's rule, namely, for any $e\subseteq \Theta$,
\begin{itemize}
	\item $ m_3^\text{in} (e) = m_1^\text{in} (e |E)  m_2(E) $, where $E = \rho(\bar{\rho}(e))$;
	\item $ m_3^\text{out} (e) = m_1^\text{out} (e |E)  m_2(E)$. 
\end{itemize}
Both $m_3^\text{in}$ and $m_3^\text{out}$ satisfy (Q1).  As for (Q2), $m_3^\text{in}$ applies, whereas $m_3^\text{out}$ does so only partially, since $(m_3^\text{out})^\text{in} (e |E) = m_1^\text{out}(e|E)$ \cite{ZhouWQ14}. 

In \cite{76dbd9a53fe5495695d6fc9e7d8fcd3f}, Ma et al. defined a new Jeffrey's rule where the conditional constraints are indeed defined according to Dempster's rule of combination, and with respect to the whole power set of the frame instead of a subalgebra as in Smets's framework. In their rule, however, the conditional constraints are not preserved by their total belief functions. 

\subsubsection{Exploring the space of solutions} 

Lemma \ref{the:ConditionalBeliefs} ensures that the solution (\ref{eq:total-belief-function}) of the total belief theorem has focal elements which are unions of 
one focal element $e_i$ of a subset of conditional belief functions $Bel_i$.

The following result prescribes that \emph{any} 
solution to the total belief problem must have focal elements which obey the above structure (see \cite{cuzzolin14lap} for a proof).
\begin{proposition} \label{pro:focal-elements}
Each focal element $e$ of a total belief function $Bel$ meeting the requirements of Theorem \ref{the:total-belief} is the union of \emph{exactly one} focal element for each of the conditional belief functions whose domain $\Pi_i$ is a subset of $\rho(E)$, where $E$ is the smallest focal element of the a priori belief function $Bel_0$ such that $e \subset \rho(E)$. Namely,
\begin{equation} \label{eq:structure}
\displaystyle e = \bigcup_{i : \Pi_i \subset \rho(E)} e_i^{j_i},
\end{equation}
where $e_i^{j_i}\in{\mathcal{E}}_i$ $\forall i$, and ${\mathcal{E}}_i$ denotes the list of focal elements of $Bel_i$.
\end{proposition}

\paragraph{The special total belief theorem}

If we require the a priori function $Bel_0$ to have only \emph{disjoint} focal elements (i.e., we force $Bel_0$ to be the vacuous extension of a Bayesian function defined on some coarsening of $\Omega$), we have what we call the \emph{restricted} or \emph{special} total belief theorem \cite{cuzzolin14lap}. The situation generalises Corollary \ref{cor:total-belief-bayesian}.

In this special case, it suffices to solve separately the $|\mathcal{E}_{\Omega}|$ subproblems obtained by considering each focal element $E$ of $Bel_0$, and then combine the resulting partial solutions by simply weighting the resulting basic probability assignments using the a priori mass $m_0(E)$, to obtain a fully normalised total belief function.  

\paragraph{Candidate solutions as linear systems}

For each individual focal element of $Bel_0$, the task of finding a suitable solution translates into a linear algebra problem.

Let $N = |E|$ be the cardinality of $E$. A candidate solution to the subproblem of the restricted total belief problem associated with $E \in{\mathcal{E}}_{\Omega}$ is the solution to a linear system with $n_{\min} = \sum_{i = 1,...,N}(n_i - 1) + 1$ equations and $n_{\max} = \prod_i n_i$ unknowns,
\begin{equation} \label{eq:candidate-solution}
A \vec{x} = \vec{b},
\end{equation}
where $n_i$ is the number of focal elements of the conditional belief function $Bel_i$.
Each column of $A$ is associated with an admissible (i.e., satisfying the structure of Proposition \ref{pro:focal-elements}) focal element $e_j$ of the candidate total belief function with mass assignment $m$, $\vec{x} = [m (e_1), \ldots, m(e_n)]'$, and $n = n_{\min}$ is the number of equalities generated by the $N$ conditional constraints. 

\paragraph{Minimal solutions}
Since the rows of the solution system (\ref{eq:candidate-solution}) are linearly independent, any system of equations obtained by selecting $n_{\min}$ columns from $A$ has a unique solution. A \emph{minimal} (i.e., with the minimum number of focal elements) solution to the special total belief problem is then uniquely determined by the solution of a system of equations obtained by selecting $n_{\min}$ columns from the $n_{\max}$ columns of $A$.

\paragraph{A class of linear transformations}

Let us represent each focal element of the candidate total function $e = \bigcup_{i : \Pi_i \subset \rho(E)} e_i^{j_i}$, where $e_i^{j_i}\in{\mathcal{E}}_i$ $\forall i$ (see (\ref{eq:structure})), as the vector of indices of its constituent focal elements $e = [ j_1, \ldots, j_N ]'$.

\begin{definition} \label{def:column-transformation}
We define a class $\mathcal{T}$ of transformations acting on columns $e$ of a candidate minimal solution system via the following formal sum:
\begin{equation} \label{eq:column-transformation}
e \mapsto e' = -e + \sum_{i\in \mathcal{C}} e_i -\sum_{j\in\mathcal{S}} e_j,
\end{equation}
where $\mathcal{C}$, $|\mathcal{C}|<N$, is a covering set of `companions' of $e$ (i.e., such that every component of $e$ is present in at least one of them), and a number of `selection' columns $\mathcal{S}$, $|\mathcal{S}|=|\mathcal{C}|-2$, are employed to compensate for the side effects of $\mathcal{C}$ to yield an admissible column (i.e., a candidate focal element satisfying the structure of Proposition \ref{pro:focal-elements}).
\end{definition}

We call the elements of $\mathcal{T}$ \emph{column substitutions}. 

A sequence of column substitutions induces a discrete path in the solution space: the values $m(e_i)$ of the solution components associated with the columns $e_i$ vary, and in a predictable way. If we denote by $s<0$ the (negative) solution component associated with the old column $e$:
\begin{enumerate}
\item 
The new column $e'$ has as its solution component $- s>0$.
\item 
The solution component associated with each companion column is decreased by $|s|$.
\item 
The solution component associated with each selection is increased by $|s|$.
\item 
All other columns retain the old values of their solution components.
\end{enumerate}
The proof is a direct consequence of the linear nature of the transformation (\ref{eq:column-transformation}).

Clearly, if we choose to substitute the column with the most negative solution component, the overall effect is that the most negative component is changed into a positive one, components associated with selection columns become more positive (or less negative), and, as for companion columns, while some of them may end up being assigned negative solution components, these will be smaller in absolute value than $|s|$ (since their initial value was positive). Hence we have the following proposition.
\begin{proposition} \label{the:column-transformation}
Column substitutions of the class $\mathcal{T}$ reduce the absolute value of the most negative solution component.
\end{proposition}

\paragraph{An alternative existence proof} 

We can then use Theorem \ref{the:column-transformation} to prove that there always exists a selection of columns of $A$ (the possible focal elements of the total belief function $Bel$) such that the resulting square linear system has a positive vector as a solution. This can be done in a constructive way, by applying a transformation of the type (\ref{eq:column-transformation}) recursively to the column associated with the most negative component, to obtain a path in the solution space leading to the desired solution.

Such an existence proof for the special total belief theorem, as an alternative to that provided earlier in this section, exploits the effects on solution components of column substitutions of type $\mathcal{T}$:
\begin{enumerate}
\item 
By Theorem \ref{the:column-transformation}, at each column substitution, the most negative solution component decreases.
\item 
If we keep replacing the most negative variable, we keep obtaining \emph{distinct} linear systems, for at each step the transformed column is assigned a positive solution component and, therefore, if we follow the proposed procedure, \emph{cannot be changed back to a negative one} by applying transformations of class $\mathcal{T}$.
\item 
This implies that there can be no cycles in the associated path in the solution space.
\item 
The number $\binom{n_{\max}}{n_{\min}}$ of solution systems is obviously finite, and hence the procedure must terminate.
\end{enumerate}
Unfortunately, counter-examples show that there are `transformable' columns (associated with negative solution components) which do not admit a transformation of the type (\ref{eq:column-transformation}). Although they do have companions on every partition $\Pi_i$, such counter-examples do not admit a complete collection of `selection' columns.

\paragraph{Solution graphs} 

The analysis of significant particular cases confirms that, unlike the classical law of total probability, the total belief theorem possesses more than one admissible minimal solution, even in the special case.

In \cite{cuzzolin14lap}, we noticed that all the candidate minimal solution systems related to a total belief problem of a given size $\{ n_i, \; i = 1, \ldots, N \}$ can be arranged into a \emph{solution graph}, whose
structure and symmetries can be studied to compute the number of minimal total belief functions for any given instance of the problem. We focused on the group of permutations of focal elements of the conditional belief functions $Bel_i$,
\begin{equation} \label{eq:group}
G = S_{n_1} \times \cdots \times S_{n_N},
\end{equation}
namely the product of the permutation groups $S_{n_i}$ acting on the collections of focal elements of each individual conditional belief function $Bel_i$. The group $G$ acts on the solution graph by generating orbits, i.e., the set of all candidate solution systems (nodes of the graph) obtained by some permutation of the focal elements within at least some of the partition elements $\Pi_i$. We hypothesised that such orbits are in 1--1 
correspondence with the number of admissible solutions.

\paragraph{Future agenda} 

A number of research directions remain open, based on the results obtained so far.
The analysis of the special law of total belief needs to be completed, by finalising the alternative proof of existence and, through the latter, completing the analysis of solution graphs and the number of admissible solutions.

Further down the road, the full description of all minimal solutions needs to be extended to the general case of an arbitrary prior $Bel_0$, which our analysis earlier in this section sheds some light on but fails to provide a comprehensive understanding. 
Distinct versions of the law of total belief may arise from replacing Dempster conditioning with other accepted forms of conditioning for belief functions, such as credal \cite{fagin91new}, geometric \cite{suppes1977}, conjunctive and disjunctive \cite{smets93jeffrey} conditioning.
As belief functions are a special type of coherent lower probabilities, which in turn can be seen as a special class of lower previsions (see \cite{walley91book}, Section 5.13), marginal extension \cite{miranda07marginal} can be applied to them to obtain a total lower prevision. The relationship between marginal extension and the law of total belief therefore needs to be understood.

Finally, fascinating relationships exist between the total belief problem and transversal matroids \cite{Oxley92}, on the one hand, and the theory of positive linear systems \cite{farina2011positive}, on the other, which we also plan to investigate in the near future.

\subsection{Limit theorems for random sets} \label{sec:future-theorems} 

The law of total probability is only one important result of classical probability theory that needs to be generalised to the wider setting of random sets. 

In order to properly define a Gaussian belief function, for instance, we would need to generalise the classical central limit theorem to random sets. The old proposal of Dempster and Liu merely transfers normal distributions on the real line by the Cartesian product with $\mathbb{R}^m$
\cite{liu99local}.
In fact, both the central limit theorem and the law(s) of large numbers have already been generalised to imprecise probabilities \cite{DeCooman20082409,Cozman20101069}.\footnote{See {\url{http://onlinelibrary.wiley.com/book/10.1002/9781118763117}}.} Here we review the most recent and relevant attempts at formulating similar laws for belief functions.

\subsubsection{Central limit theorems} \label{sec:future-central} 

Chareka \cite{CIS-234958} and later Ter\'an \cite{Terán2015185} have conducted interesting work on the central limit theorem for capacities.
More recently, specific attention has been directed at central limit theorem results for belief measures \cite{Epstein2011CLT,Shi2015clt}.

\paragraph{Choquet's definition of belief measures}

Let $\Theta$ be a Polish space\footnote{A separable completely metrisable topological space; that is, a space homeomorphic to a complete metric space that has a countable dense subset.} and $\mathcal{B}(\Theta)$ be the Borel $\sigma$-algebra on $\Theta$. Let us denote by $\mathcal{K}(\Theta)$ the collection of compact subsets of
$\Theta$, and by $\mathcal{P}(\Theta)$ the space of all probability measures on $\Theta$.
The set $\mathcal{K}(\Theta)$ will be endowed with the Hausdorff topology\footnote{A \emph{Hausdorff space}, or `separated' space, is a topological space in which distinct points have disjoint neighbourhoods. It implies the uniqueness of limits of sequences, nets and filters.} generated by the topology of $\Theta$.

\begin{definition} \label{def:belief-measure-choquet}
A belief measure on $(\Theta,\mathcal{B}(\Theta))$ is defined as a set function $Bel: \mathcal{B}(\Theta) \rightarrow [0,1]$ satisfying:
\begin{itemize}
\item
$Bel(\emptyset) = 0$ and $Bel(\Theta) = 1$;
\item
$Bel(A) \leq Bel(B)$ for all Borel sets $A \subset B$;
\item
$Bel(B_n)\downarrow Bel(B)$ for all sequences of Borel sets $B_n \downarrow B$;
\item
$Bel(G) = \sup \{ Bel(K) : K \subset G, K \in \mathcal{K}(\Theta) \}$, for all open sets G;
\item
$Bel$ is totally monotone (or $\infty$-monotone ): for all Borel sets $B_1, \ldots, B_n$,
\[
Bel \left ( \bigcup_{i=1}^n B_i \right ) \geq \sum_{\emptyset \neq I \subset \{1,...,n\}} (-1)^{|I|+1} Bel \left ( \bigcap_{i \in I} B_i \right ) .
\]
\end{itemize}
\end{definition} 

By \cite{Philippe99decision}, for all $A \in \mathcal{B}(\Theta)$ the collection $\{K \in \mathcal{K}(\Theta): K \subset A\}$ is universally measurable.\footnote{A subset $A$ of a Polish space $X$ is called \emph{universally measurable} if it is measurable with respect to every complete probability measure on $X$ that measures all Borel subsets of $X$.} Let us denote by $\mathcal{B}_u(\mathcal{K}(\Theta))$ the $\sigma$-algebra of all subsets of $\mathcal{K}(\Theta)$ which are universally measurable. The following result
is due to Choquet \cite{Choquet53}.
\begin{proposition} \label{pro:choquet}
The set function $Bel: \mathcal{B}(\Theta) \rightarrow [0,1]$ is a belief measure if and only if there exists a probability measure $P_{Bel}$ on $(\mathcal{K}(\Theta), \mathcal{B}(\mathcal{K}(\Theta)))$ such that
\[
Bel(A) = P_{Bel}(\{K \in \mathcal{K}(\Theta) : K \subset A \}), \quad \forall A \in \mathcal{B}(\Theta). 
\]
Moreover, there exists a unique extension of $P_{Bel}$ to $(\mathcal{K}(\Theta), \mathcal{B}_u(\mathcal{K}(\Theta)))$, which we can still denote by $P_{Bel}$.
\end{proposition}

By comparison with the usual random-set formulation,
we can appreciate that Proposition \ref{pro:choquet} confirms that a belief measure is the total probability induced by a probability measure on a source space -- in this case, the source space $\Omega = \mathcal{K}(\Theta)$ is the collection of all compact subspaces of $\Theta$.

For any $A \in \mathcal{B}(\Theta^\infty)$, we define
\begin{equation} \label{eq:belief-infinity}
Bel^\infty (A) \doteq P_{Bel}^\infty (\{ K = K_1 \times K_2 \times \cdots \in (\mathcal{K}(\Theta))^\infty : K \subset A \}),
\end{equation}
where $P_{Bel}^\infty$ is the i.i.d. product probability measure. $Bel^\infty$ is the unique belief measure on $(\Theta^\infty,\mathcal{B}(\Theta^\infty)$ induced by $P_{Bel}^\infty$.

\paragraph{Bernoulli variables}  

A central limit theorem for belief functions, in their Choquet formulation, was recently proposed by  Epstein and Seo for the case of Bernoulli random variables \cite{Epstein2011CLT}.

Let\footnote{We adopt our own notation here, in place of Epstein and Seo's original one.} $\Theta = \{T,F\}$, and consider the set $\Theta^\infty$ of all infinite series of samples $\theta^\infty = (\theta_1,\theta_2, \ldots)$ extracted from $\Theta$. Let $\Phi_n(\theta^\infty)$ be the empirical frequency of the outcome $T$ in the first $n$ experiments in sample $\theta^\infty$.
Let $Bel$ be a belief function on $\Theta$ induced by a measure $P_{Bel}$, and $Bel^\infty$ the belief measure (\ref{eq:belief-infinity}) on $\Theta^\infty$.

The law of large numbers asserts certainty that asymptotic empirical frequencies will lie in the
interval $[Bel (T), 1 - Bel (F)]$, that is,
\[
Bel^\infty \Big \{\theta^\infty : [\lim \inf  \Phi_n(\theta^\infty), \lim \sup \Phi_n(\theta^\infty) ] \subset [Bel (T), 1 - Bel (F)] \Big \} = 1.
\]
The authors of \cite{Epstein2011CLT} proved the following proposition.
\begin{proposition}
The following hold:
\[
\begin{array}{c}
\lim_{n\rightarrow \infty} Bel^\infty \left (\left \{ \theta^\infty : \sqrt{n} \displaystyle \frac{\Phi_n(\theta^\infty) - Pl(T)}{\sqrt{Bel(F)(1-Bel(F))}} \leq \alpha \right \} \right) = \mathcal{N}(\alpha),
\\ \\
\lim_{n\rightarrow \infty} Bel^\infty \left (\left \{ \theta^\infty : \sqrt{n} \displaystyle \frac{\Phi_n(\theta^\infty) - Bel(T)}{\sqrt{Bel(T)(1-Bel(T))}} \leq \alpha \right \} \right) = 1 - \mathcal{N}(\alpha),
\end{array}
\]
where $\mathcal{N}$ denotes the usual normal distribution on real numbers. 
\end{proposition}

The proof follows from showing that, for the events indicated, the minimising measures are i.i.d. As a result, classical limit theorems applied to these measures deliver corresponding limit theorems for the i.i.d. product $Bel^\infty$.

The central result of \cite{Epstein2011CLT}, however, is the following.
\begin{proposition}
Suppose that $G: \mathbb{R} \rightarrow \mathbb{R}$ is bounded, quasi-concave and upper-
semicontinuous. Then
\[
\int G(\Phi_n(\theta^\infty)) \; \mathrm{d} Bel^\infty(\theta^\infty) = E [G(X'_{1n}), G(X'_{2n}) ] + O\left (\frac{1}{\sqrt{n}} \right ),
\]
where $(X'_{1n},X'_{2n})$ is normally distributed with mean $(Bel(T), Pl(T))$ and covariance
\[
\frac{1}{n}
\left [ 
\begin{array}{cc}
Bel(T)(1-Bel(T)) & Bel(T)Bel(F)
\\
Bel(T)Bel(F) & (1-Bel(F)) Bel(F)
\end{array}
\right ].
\]
That is,
\[
\lim \sup_{n\rightarrow \infty} \sqrt{n} \left |  \int G(\Phi_n(\theta^\infty)) \; \mathrm{d} Bel^\infty(\theta^\infty)  -
E [G(X'_{1n}), G(X'_{2n}) ] \right | \leq K
\]
for some constant $K$.
\end{proposition}

\paragraph{Generalisation} 

Xiaomin Shi \cite{Shi2015clt} has recently (2015)
generalised the result of Epstein and Seo \cite{Epstein2011CLT} from Bernoulli random variables to
general bounded random variables,\footnote{\url{https://arxiv.org/pdf/1501.00771.pdf}.} thanks to a theorem in a seminar paper \cite{Choquet53} by Choquet. 

\begin{proposition}
Let $Y_i$, $i=1,\ldots$, be equally distributed random variables on $(\Theta,\mathcal{B}(\Theta))$, bounded by a constant $M$, and let us define $X_i(\theta_1,\theta_2,..) \doteq Y_i(\theta_i)$, $i\geq 1$. 
Then we have, for all $\alpha \in \mathbb{R}$,
\[
\lim_{n \rightarrow \infty} Bel^\infty \left ( 
\frac{\sum_{i=1}^n X_i - n \underline{\mu}}{\sqrt{n} \underline{\sigma}} \geq \alpha
\right ) =  1 - \mathcal{N}(\alpha),
\]
and
\[
\lim_{n \rightarrow \infty} Bel^\infty \left ( 
\frac{\sum_{i=1}^n X_i - n \overline{\mu}}{\sqrt{n} \overline{\sigma}} < \alpha
\right ) =  1 - \mathcal{N}(\alpha),
\]
where $\underline{\mu} = E_{P_{Bel}^\infty} [\underline{Z}_i]$, $\overline{\mu} = E_{P_{Bel}^\infty} [\overline{Z}_i]$ and 
\[
\begin{array}{lll}
\underline{Z}_i (K_1 \times K_2 \times \cdots) & = & \inf_{\theta_i \in K_i} X_i(\theta_1,\theta_2, \ldots ),
\\
\overline{Z}_i (K_1 \times K_2 \times \cdots) & = & \sup_{\theta_i \in K_i} X_i(\theta_1,\theta_2, \ldots ),
\end{array}
\]
for $K_i \in \mathcal{K}(\Theta)$, $i\geq 1$. The expressions for $\underline{\sigma}$ and $\overline{\sigma}$ are given in \cite{Shi2015clt}, Theorem 3.1.
\end{proposition}

A central limit theorem for two-sided intervals was also provided (\cite{Shi2015clt}, Theorem 4.1).

\paragraph{Gaussian random sets} \label{sec:future-gaussian}

One can note, however, that the above results do not really address the question of whether there exists, among all (compact) random sets, a special class (playing the role of Gaussian probability measures) characterised by all sums of independent and equally distributed random sets converging to an object of that class (which it would make sense to call a proper `Gaussian random set').
The issue remains open for future investigation.

\begin{question} \label{que:gaussian-random-sets}
Does there exist a (parameterisable) class of random sets, playing a role similar to that which Gaussian distributions play in standard probability, such that sample averages of (the equivalent of) i.i.d. random sets converge in some sense to an object of such a class?
\end{question}

\subsubsection{Laws of large numbers} \label{sec:future-large-numbers}

Several papers have been published on laws of large numbers for capacities \cite{10.2307/3481723}, monotone measures \cite{Agahi20131213} and non-additive measures \cite{REBILLE2009872,Terán2014lln}, although not strictly for belief measures.
Molchanov \cite{molchanov1988uniform}, among others, studied the Glivenko--Cantelli theorem for capacities induced by random sets. A strong law of large numbers for set-valued random variables in a $G_\alpha$ space was recently produced by Guan Li \cite{li2006strong,li2015strong}, based on Taylor's result for single-valued random variables. The issue was also recently addressed by Kerkvliet in his PhD dissertation \cite{Kerkvliet2017}. 

\begin{definition}
The \emph{weak law of large numbers} states that the sample average converges in probability towards the expected value. Consider an infinite sequence of i.i.d. Lebesgue-integrable random variables $X_1,X_2, \ldots$ with expected value $\mu$ and sample average $
{\overline {X}}_{n} = (X_{1}+\cdots +X_{n}) / n$. Then, for any positive number $\varepsilon$,
\begin{equation} \label{eq:weak-lln}
\lim _{n\to \infty } P \!\left(\,|{\overline {X}}_{n}-\mu |>\varepsilon \,\right)=0.
\end{equation}
The \emph{strong law of large numbers} states that the sample average converges almost surely to the expected value. Namely,
\begin{equation} \label{eq:strong-lln}
{\displaystyle P \!\left(\lim _{n\to \infty }{\overline {X}}_{n}=\mu \right)=1.}
\end{equation}
\end{definition} 

When the probability measure is replaced by a set function $\nu$ assumed to be completely monotone, the empirical frequencies $f_n(A)$ of an event $A$ over a sequence of Bernoulli
trials satisfy
\[
\nu \left ( 
\nu(A) \leq \lim \inf_n f_n(A) \leq \lim \sup_n f_n(A) \leq 1 - \nu(A^c) \right ) = 1.
\]
In order to replace the event by a random variable $X$ on a measurable space $(\Theta,\mathcal{B})$, one needs to calculate expectations with respect to $\nu$ and $\overline{\nu} = 1 - \nu(A^c)$ in such a way that 
\[
E_\nu[I_A] = \nu(A), \quad E_{\overline{\nu}}[I_A] = 1 - \nu(A^c), 
\]
and both expectations equal the usual definition when $\nu$ is a probability measure. That is achieved by using the Choquet integral \cite{Choquet53}.

The resulting type of limit theorem has already been considered by Marinacci \cite{Marinacci1999limit} and Maccheroni and Marinacci \cite{10.2307/3481723}, and in more recent papers \cite{Chen2010strong,DeCooman20082409,Cozman20101069,REBILLE2009872,doi:10.1142/S0218488509006212}. In those papers, the weakening of the axiomatic properties of probability are balanced by the incorporation of extra technical assumptions about the properties of $\Theta$ and/or the random variables, for instance that $\Theta$ is a compact or Polish \cite{10.2307/3481723} topological space or that the random variables are continuous functions, are bounded, have sufficiently high moments, are independent in an unreasonably strong sense or satisfy ad hoc regularity requirements. 

In \cite{Terán2014lln}, it was shown that additivity can be replaced by complete monotonicity in the law of large numbers under the reasonable first-moment condition that both of the Choquet expectations $E_\nu [X]$ and $E_{\overline{\nu}}[X]$ are finite, with no additional assumptions. Under the assumption that the sample space is endowed with a topology, the continuity condition for monotone sequences can also be relaxed. 
The main result reads as follows (\cite{Terán2014lln}, Theorem 1.1).
\begin{proposition}
Let $\nu$ be a completely monotone set function on a measurable space
$(\Theta,\mathcal{B})$. Let $X$ be an integrable\footnote{That is, its lower and upper expectations $E_\nu [X]$ and $E_{\overline{\nu}} [X]$ are finite.} random variable, and let $\{X_n, n\}$ be pairwise preindependent and identically distributed to $X$. 
If $\nu$ is a capacity, then for every $\epsilon >0$,
\[
\nu \left ( 
E_\nu [X] - \epsilon < {\overline {X}}_{n} < E_{\overline{\nu}} [X] + \epsilon
\right ) \rightarrow 1.
\]
If $(\Theta,\mathcal{B})$ is a topological space with a Borel $\sigma$-algebra, $\nu$ is a topological capacity and the $X_n$ are continuous functions on $\Theta$, then
\[
\nu \left ( 
E_\nu [X] \leq \lim \inf_n {\overline {X}}_{n} \leq \lim \sup_n {\overline {X}}_{n} \leq E_{\overline{\nu}} [X]
\right ) = 1.
\]
\end{proposition}

\subsection{Frequentist inference with random sets} \label{sec:future-frequentist} 

Random sets are mathematical objects detached from any specific interpretation. Just as probability measures are used by both Bayesians and frequentists for their analyses, random sets can also be employed in different ways according to the interpretation they are provided with. 

In particular, it is natural to imagine a generalised frequentist framework in which random experiments are designed by assuming a specific random-set distribution, rather than a conventional one, in order to better cope with the ever-occurring set-valued observations (see the Introduction).

\subsubsection{Parameterised families of random sets} \label{sec:future-families}

The first necessary step is to introduce parametric models based on random sets.

\begin{figure}[ht!]
\centering
\includegraphics[width = \textwidth]{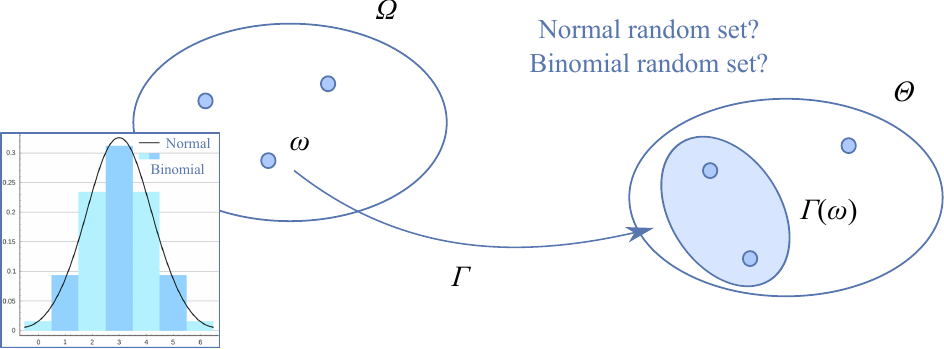}
\caption{Describing the family of random sets (right) induced by families of probability distributions in the source probability space (left) is a likely first step towards a generalisation of frequentist inference to random sets.}\label{fig:parameterised}
\end{figure}

Recall Dempster's random-set interpretation (Fig. \ref{fig:parameterised}).
Should the multivalued mapping $\Gamma$ which defines a random set be `designed', or derived from the problem?
For instance, in the cloaked die example discussed in the Introduction to \cite{cuzzolin2021springer}, 
it is the occlusion which generates the multivalued mapping and we have no control over it.
In other situations,
however, it may make sense to impose a parameterised family of mappings
\[
\Gamma(.|\pi) : \Omega \rightarrow 2^{\Theta}, \quad \pi \in \Pi,
\]
which, given a (fixed) probability on the source space $\Omega$, would yield as a result a parameterised family of random sets.

The alternative is to fix the multivalued mapping (e.g., when it is given by the problem), and model the source probability by a classical parametric model. A Gaussian or binomial family of source probabilities would then induce a family of `Gaussian'\footnote{Note that the Gaussian random sets so defined would be entirely distinct from Gaussian random sets as advocated in the context of the generalised central limit theorem setting.} or `binomial' random sets (see Fig. \ref{fig:parameterised} again). 

\begin{question} \label{que:binomial-random-sets}
What is the analytical form of random sets induced by say, Gaussian, binomial or exponential families of (source) probability distributions?
\end{question}

\subsubsection{Hypothesis testing with random sets} \label{sec:future-hypothesis}

As we know, in hypothesis testing \cite{berger1997unified}
designing an experiment amounts to choosing a family of probability distributions which is assumed to generate the observed data. If parameterised families of random sets can be constructed, they can then be plugged into the frequentist inference machinery, after an obvious generalisation of some of the steps involved. 

Hypothesis testing with random sets would then read as follows:
\begin{itemize}
\item
State the relevant null hypothesis $H_0$ and the alternative hypothesis.
\item
State the assumptions about the form of the 
\emph{random set} (or \emph{mass assignment}, as opposed to a classical distribution) describing the observations.
\item
State the relevant test statistic $T$ (a quantity derived from the sample) -- only {this time the sample may contain set-valued observations!}
\item
Derive the 
\emph{mass assignment} (rather than probability distribution) of the test statistic under the null hypothesis (from the assumptions).
\item
Set a significance level ($\alpha$).
\item
Compute from the observations the observed value $t_\text{obs}$ of the test statistic $T$ -- this will also now be set-valued.
\item
Calculate the 
\emph{conditional belief value} (playing the role of the classical p-value) under $H_0$ of sampling a test statistic at least as extreme as the observed value.
\item
Reject the null hypothesis, in favour of the alternative hypothesis, if and only if 
the above {conditional belief value} is less than the significance level.
\end{itemize}

\begin{question} \label{que:hypothesis-testing}
How would results from random-set hypothesis testing relate to classical hypothesis testing applied to parameterised distributions within the assumed random set?
\end{question}

\subsection{Random-set random variables} \label{sec:future-random-variables}

We know that random sets are set-valued random variables: nevertheless, the question remains as to whether one can build random variables on top of random-set (belief) spaces, rather than the usual probability space. It would be natural to term such objects \emph{random-set random variables}.

Just as in the classical case, we would need a mapping from $\Theta$ to a measurable space (e.g. the positive real half-line)
\[
f : \Theta \rightarrow \mathbb{R}_+ = [0,+\infty],
\]
where this time $\Theta$ is itself the codomain of a multivalued mapping $\Gamma: \Omega \rightarrow 2^\Theta$, with source probability space $\Omega$.


\subsubsection{Generalising the Radon--Nikodym derivative} \label{sec:future-radon}

For a classical continuous random variable $X$ on a measurable space $(\Theta,\mathcal{F}(\Theta))$, we can compute its PDF as its \emph{Radon--Nikodym derivative} (RND), namely the measurable function $p : \Theta \rightarrow [0,\infty)$ such that for any measurable subset $A \subset \Theta$,
\[
P[X \in A] = \int_A p \; \mathrm{d} \mu,
\]
where $\mu$ is a $\sigma$-finite measure on $(\Theta,\mathcal{F}(\Theta))$.
We can pose ourselves the following research question. 
\begin{question} \label{que:rsrv}
Can a (generalised) PDF be defined for a random-set random variable as defined above? 
\end{question}

Answering Question \ref{que:rsrv} requires extending the notion of a Radon--Nikodym derivative to random sets.
Graf \cite{Graf1980} pioneered the study of Radon--Nikodym derivatives for the more general case of \emph{capacities}. The extension of the RND to the even more general case of \emph{set functions} was studied by Harding et al. in 1997 \cite{Harding1997}. The problem was investigated more recently by Rebille in 2009 \cite{RePEc:hal:wpaper:hal-00441923}.
The following summary of the problem is abstracted from Molchanov's Theory of Random Sets \cite{Molchanov05}.

\subsubsection{Absolute continuity} 

Let us assume that the two capacities $\mu$, $\nu$ are monotone, subadditive and continuous from below.
\begin{definition}
A capacity $\nu$ is \emph{absolutely continuous} with respect to another capacity $\mu$ if, for every $A \in \mathcal{F}$, $\nu(A) = 0$ whenever $\mu(A)=0$.
\end{definition}

This definition is the same as for standard measures. However, whereas for
standard measures absolute continuity is equivalent to the integral relation $\mu = \int \nu$ (the existence of the corresponding RND), this is no longer true for general capacities.
To understand this, consider the case of a finite $\Theta$, $|\Theta|=n$. Then any measurable function $f:\Theta \rightarrow \mathbb{R}_+$ is determined by just $n$ numbers, which do not suffice to uniquely define a capacity on $2^\Theta$ (which has $2^n$ degrees of freedom) \cite{Molchanov05}.

\subsubsection{Strong decomposition} \label{sec:future-strong-decomposition}

Nevertheless, a Radon--Nikodym theorem for capacities can be established after introducing the notion of \emph{strong decomposition}.

Consider a pair of capacities $(\mu,\nu)$ that are monotone, subadditive and continuous from below.
\begin{definition} \label{def:strong-decomposition}
The pair $(\mu,\nu)$ is said to have the \emph{strong decomposition} property if, $\forall \alpha \geq 0$, there exists a measurable set $A_\alpha \in \mathcal{F}$ such that
\[
\begin{array}{ll}
\alpha (\nu(A)-\nu(B)) \leq \mu(A) - \mu(B) & \text{if} \; B \subset A \subset A_\alpha, 
\\
\alpha (\nu(A) - \nu(A \cap A_\alpha)) \geq \mu(A) - \mu(A \cap A_\alpha) & \forall A.
\end{array}
\]
\end{definition}

Roughly speaking, the strong decomposition condition states that, for each bound $\alpha$, the `incremental ratio' of the two capacities is bounded by $\alpha$ in the sub-power set capped by some event $A_\alpha$.
Note that all standard measures possess the strong decomposition property.

\subsubsection{A Radon--Nikodym theorem for capacities}

The following result was proved by Graf \cite{Graf1980}.
\begin{theorem} \label{the:rnd-capacities}
For every two capacities $\mu$ and $\nu$, $\nu$ is an indefinite integral of $\mu$ if and only if the pair $(\mu,\nu)$ has the strong decomposition property \emph{and} $\nu$ is absolutely continuous with respect to $\mu$.
\end{theorem} 

Nevertheless, a number of issues remain open. Most relevantly to our ambition to develop a statistical theory based on random sets, the conditions of Theorem \ref{the:rnd-capacities} (which holds for general capacities) need to be elaborated for the case of completely alternating capacities (distributions of random closed sets).
As a first step, Molchanov notes (\cite{Molchanov05}, page 75) that the strong decomposition property for two capacity functionals $\nu = T_X$ and $\mu=T_Y$ associated with the random sets $X$, $Y$ implies that
\[
\alpha {T}_X(\mathcal{C}_A^B) \leq {T}_Y(\mathcal{C}_A^B)
\]
if $B \subset A \subset A_\alpha$, and
\[
\alpha {T}_X(\mathcal{C}_A^{A \cap A_\alpha}) \geq {T}_Y(\mathcal{C}_A^{A \cap A_\alpha}) \; \forall A,
\]
where $\mathcal{C}_A^B = \{C \in \mathcal{C}, B \subset C \subset A \}$.

A specific result on the Radon--Nikodym derivative for random (closed) sets is still missing, and with it the possibility of generalising the notion of a PDF to these more complex objects. 

\section{Extending the geometric approach} \label{sec:future-geometric}

\subsection{The geometry of uncertainty}

As we argued in \cite{cuzzolin2021springer},
\emph{geometry} may be a unifying language, not only for random set and evidential theory, but for uncertainty measures more generally 
\cite{cuzzolin18belief-maxent,Cuzzolin99,cuzzolin05isipta,cuzzolin00mtns,cuzzolin13fusion,gennari02-integrating,gong2017belief,black97geometric,rota97book,ha98geometric,wang91geometrical}, possibly in conjunction with an algebraic perspective \cite{cuzzolin00rss,cuzzolin01bcc,cuzzolin08isaim-matroid,cuzzolin01lattice,cuzzolin05amai,cuzzolin07bcc,cuzzolin14algebraic}.

{In our 
geometric approach to uncertainty \cite{cuzzolin2021springer}, in particular, uncertainty measures 
are assimilated to points of a geometric space, where that can be processed via combination, conditioning and so on
\cite{cuzzolin01thesis,cuzzolin2008geometric}.
Most 
{efforts have focused}
on the geometry of belief functions, 
{as elements of a}
\emph{belief space} 
{representable as} either
a simplex 
({the generalisation of a triangle in higher dimensions})
or a recursive bundle structure \cite{cuzzolin01space,cuzzolin03isipta,cuzzolin14annals,cuzzolin14lap}.}
The approach can be applied to combination rules, such as Dempster's sum, resulting in rather elegant geometric procedures for them \cite{cuzzolin02fsdk,cuzzolin04smcb}.
The combinatorial 
{features}
of plausibility and commonality functions, 
{set functions which carry the same information as the associated BF,}
have also been 
{touched upon}
\cite{cuzzolin08pricai-moebius,cuzzolin10ida}.
\\
The geometric approach 
{has been subsequently extended}
to other {types} uncertainty measures, 
{such as}
possibility measures 
\cite{cuzzolin10fss}, and {special classes of belief functions such as} consistent {BFs} \cite{cuzzolin11-consistent,cuzzolin2008lp,cuzzolin09isipta-consistent,cuzzolin08isaim-simplicial}, 
{via the notion of}
simplicial {complex} \cite{cuzzolin04ipmu}. 
{The geometry of BFs as} credal sets {has also been investigated}
\cite{cuzzolin08-credal,antonucci10-credal,burger10brest}.

{The geometry of the mappings between different forms of uncertainty measures has been considered}
\cite{cuzzolin05hawaii,cuzzolin09-intersection,cuzzolin07ecsqaru,cuzzolin2010credal},
{specifically the task of mapping a}
belief function 
{to a}
probability measure \cite{Cobb03isf,voorbraak89efficient,Smets:1990:CPP:647232.719592}. 
Belief functions can be approximated using possibility measures too
\cite{aregui08constructing}, a problem which
can be approached in geometric terms as well \cite{cuzzolin09ecsqaru,cuzzolin11isipta-consonant,Cuzzolin2014tfs}. Namely,
one can analytically compute and analyse approximations obtained by minimising standard $L_p$ norms and compare with outer consonant approximations \cite{Dubois90}.
Consistent approximations of belief functions induced by classical Minkowski norms may also be derived \cite{cuzzolin11-consistent}
and compared with classical outer consonant approximations \cite{Dubois90}. 

{Geometric reasoning}
can be applied to most if not all aspects of making inferences from data in the framework of epistemic uncertainty,
{including}
conditioning 
\cite{lehrer05updating},  
{by associating a}
conditioning event 
{$A$ with a}
`conditioning simplex' 
\cite{cuzzolin10brest,cuzzolin11isipta-conditional}. 
{Transformations mapping belief functions to probabilities can be interpreted}
in terms of credal sets 
\cite{cuzzolin2010credal}.

\subsection{Open questions}

Nevertheless, the geometric approach to uncertainty 
still
has much room for further development. 
For instance:
\begin{question} \label{que:geometric-inference}
Can the inference problem can be posed in a geometric setting too, by representing both data and belief measures in a common geometric space?
\end{question}

An answer to such an intriguing question would require a geometric representation general enough to encode both the data driving the inference \emph{and} the (belief) measures possibly resulting from the inference, in such a way that the inferred measure minimises some sort of distance from the empirical data.
The question of what norm is the most appropriate to minimise for inference purposes would also arise.

On the other hand, while \cite{cuzzolin2021springer} has mainly concerned itself with the geometric representation of finite belief measures, as we move on from belief fuctions on finite frames to random sets on arbitrary domains, the nature of the geometry of these continuous formulations poses new questions.
The formalism needs to tackle (besides probability, possibility and belief measures) other important mathematical descriptions of uncertainty, first of all general monotone capacities and gambles (or variations thereof).

Finally, new, more sophisticated geometric representations of belief measures can be sought, in terms of either exterior algebras or areas of projections of convex bodies.

Some of these aspects are briefly considered in this section.

\subsection{Geometry of general combination} \label{sec:future-combination} 

The study of the geometry of the notion of evidence combination and belief updating, which started with
the analysis of Dempster's rule provided in \cite{cuzzolin04smcb}, 
will need to be extended to the geometric behaviour of the other main combination operators.
A comparative geometric analysis of combination rules would allow us to describe the `cone' of possible future belief states under stronger or weaker assumptions about the reliability and independence of sources.

We can start by giving a general definition of a conditional subspace (see \cite{cuzzolin2021springer}, Chapter 8).
\begin{definition} \label{def:conditional-subspace}
Given a belief function $Bel \in \mathcal{B}$, we define the \emph{conditional subspace} $\langle Bel \rangle_\odot$ as the set of all $\odot$ combinations of $Bel$ with {any} other belief function on the same frame, where $\odot$ is an arbitrary combination rule, assuming such a combination exists. Namely,
\begin{equation}
\langle Bel \rangle_\odot \doteq \Big \{ Bel \odot Bel', \; Bel' \in \mathcal{B} \; s.t. \; \exists \; Bel \odot Bel' \Big \}.
\end{equation}
\end{definition} 

\subsubsection{Geometry of Yager's and Dubois's rules} \label{sec:future-combination-yager}

Alternatives to the Dempster rule of combination were proposed by both Yager and Dubois. 
The latter's proposal is based on the view that conflict is generated by non-reliable information sources. In response, the conflicting mass (denoted here by $m_\cap(\emptyset)$) is reassigned to the whole frame of discernment $\Theta$:
\begin{equation} \label{eq:combination-yager}
m_Y(A) = \left \{ \begin{array}{ll} m_\cap (A) & \emptyset \neq A \subsetneq \Theta, 
\\ 
m_\cap (\Theta) + m_\cap(\emptyset) & A = \Theta, \end{array} \right .
\end{equation}
where $m_\cap (A) \doteq \sum_{B \cap C = A} m_1(B) m_2(C)$ and $m_1,m_2$ are the masses to be combined.

The combination operator proposed by Dubois and Prade\footnote{We follow here the same notation as in \cite{Lefevre2002149}.} \cite{dubois88representation} comes from applying the minimum specificity principle \cite{dubois87principle} to cases in which the focal elements $B,C$ of two input belief functions do not intersect. 

This results in assigning their product mass to $B \cup C$, namely
\begin{equation} \label{eq:combination-dubois}
m_D(A) = m_\cap (A) + \sum_{B \cup C = A, B \cap C = \emptyset} m_1(B) m_2(C).
\end{equation}

\paragraph{Analysis on binary frames}

On binary frames $\Theta = \{x,y\}$, Yager's rule (\ref{eq:combination-yager}) and Dubois's rule (\ref{eq:combination-dubois}) coincide, 
as the only conflicting focal elements there are $\{x\}$ and $\{y\}$, whose union is $\Theta$ itself. Namely,
\begin{equation}\label{eq:yager-binary}
\begin{array}{lll}
m_{\text{\oyager}}(x) & = & m_1(x) (1-m_2(y)) + m_1(\Theta) m_2(x),
\\
m_{\text{\oyager}}(y) & = & m_1(y) (1-m_2(x)) + m_1(\Theta) m_2(y),
\\
m_{\text{\oyager}}(\Theta) & = & m_1(x)m_2(y) + m_1(y) m_2(x) + m_1(\Theta) m_2(\Theta).
\end{array}
\end{equation}
Using (\ref{eq:yager-binary}), we can easily show that
\begin{equation}\label{eq:vertices-yager}
\begin{array}{lll}
Bel \text{\oyager} Bel_x & = & [m(x) + m(\Theta), 0, m(y)]',
\\
Bel \text{\oyager} Bel_y & = & [0, m(y) + m(\Theta), m(x)]',
\\
Bel \text{\oyager} Bel_\Theta & = & Bel = [m(x), m(y), m(\Theta)],
\end{array}
\end{equation}
adopting the usual vector notation $Bel = [Bel(x), Bel(y), Bel(\Theta)]'$.

The global behaviour of Yager's (and Dubois's) rule in the binary case is then pretty clear: the conditional subspace $\langle Bel \rangle_{\text{\text{\oyager}}}$ is the convex closure
\[
\langle Bel \rangle_{\text{\oyager}} = Cl(Bel, Bel \text{\oyager} Bel_x, Bel \text{\oyager} Bel_y)
\]
of the points (\ref{eq:vertices-yager}) in the belief space: see Fig. \ref{fig:yager}. 

\begin{figure}[ht!]
\centering
\includegraphics[width=0.75\textwidth]{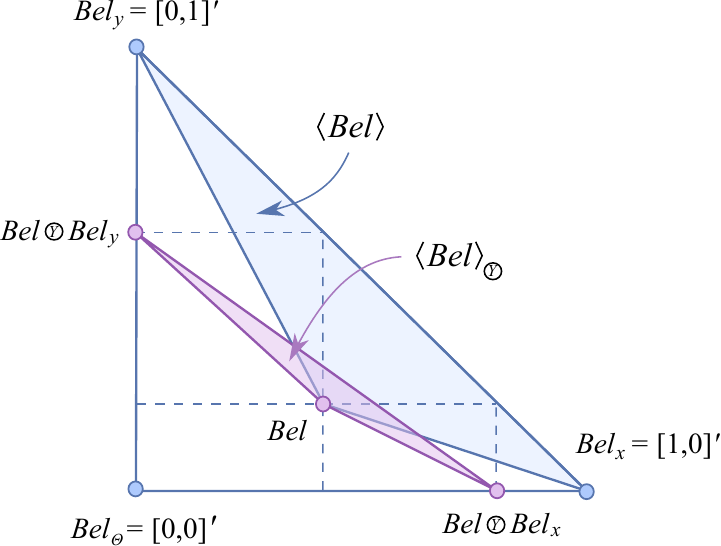}
\caption{Conditional subspace $\langle Bel \rangle_{\text{\oyager}}$ for Yager's (and Dubois's) combination rule on a binary frame $\Theta = \{x,y\}$. Dempster's conditional subspace $\langle Bel \rangle$ is also shown for the sake of comparison.} \label{fig:yager}
\end{figure} 

Comparing (\ref{eq:yager-binary}) with (\ref{eq:vertices-yager}), it is easy to see that
\[
Bel_1 \text{\oyager} Bel_2 = m_2(x) Bel_1 \text{\oyager} Bel_x + m_2(y) Bel_1 \text{\oyager} Bel_y
+ m_2(\Theta) Bel_1 \text{\oyager} Bel_\Theta,
\]
i.e., the simplicial coordinates of $Bel_2$ in the binary belief space $\mathcal{B}_2$ and of the Yager combination $Bel_1 \text{\oyager} Bel_2$ in the conditional subspace $\langle Bel_1 \rangle_{\text{\oyager}}$ coincide.

We can then conjecture the following.
\begin{conjecture} \label{con:yager-commute}
Yager combination and affine combination commute. Namely,
\[
Bel {\text{\oyager}} \left ( \sum_{i} \alpha_i Bel_i \right ) =  \sum_{i} \alpha_i Bel {\text{\oyager}} Bel_i.
\]
\end{conjecture}

As commutativity is the basis of our geometric analysis of Dempster's rule (see \cite{cuzzolin2021springer} Chapter 8), this opens the way for a similar geometric construction for Yager's rule.

\begin{figure}[ht!]
\centering
\includegraphics[width=1\textwidth]{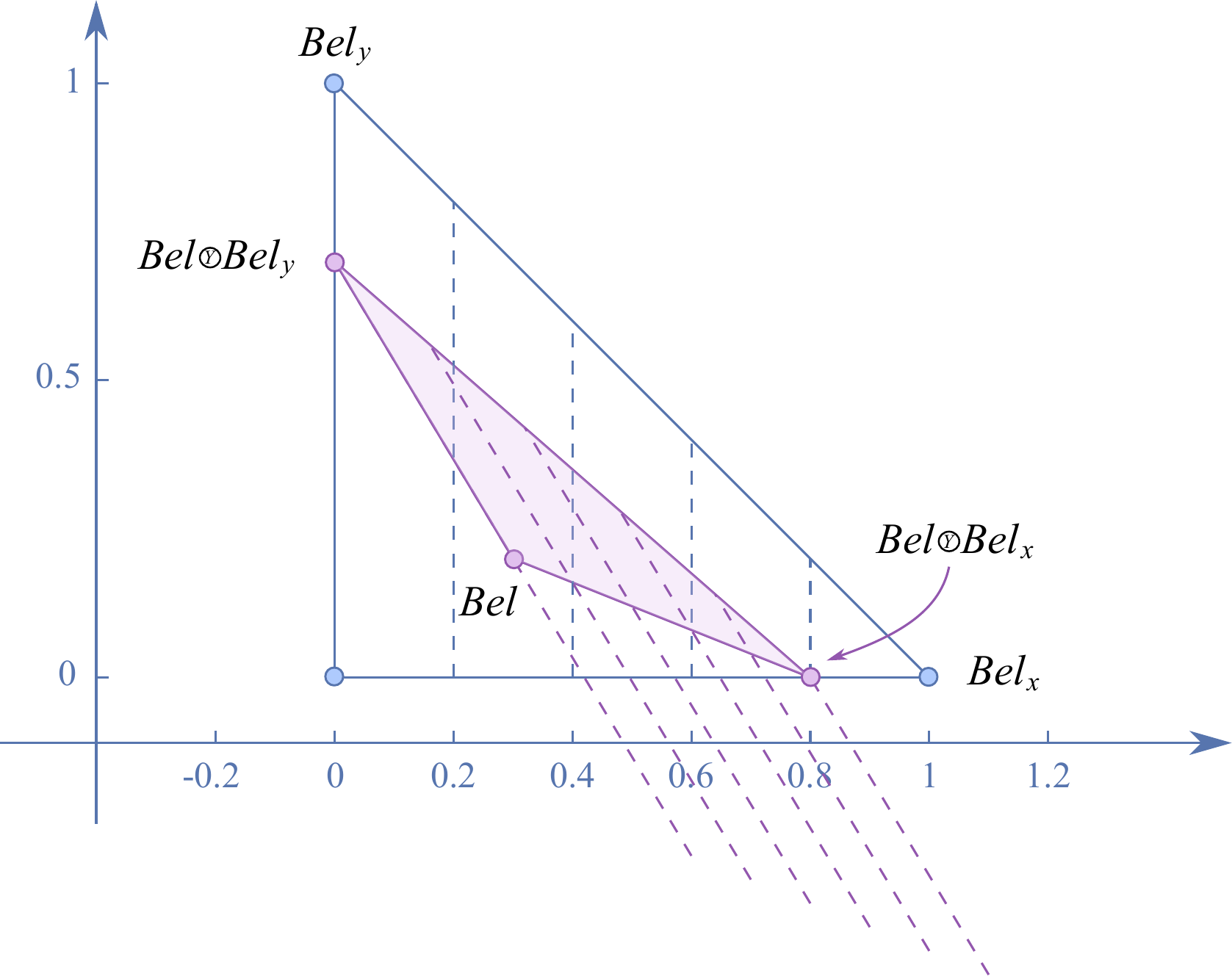}
\caption{In the case of the Yager combination, the images of constant-mass loci (dashed blue segments) do not converge to a focus, but are parallel lines (dashed purple lines; compare \cite{cuzzolin2021springer}, Fig 8.5).} \label{fig:yager-constant-mass}
\end{figure}

\paragraph{Geometric construction}

However, as shown in Fig. \ref{fig:yager-constant-mass}, the images of constant-mass loci under Yager's rule are parallel, so that the related conditional subspaces have no foci.
Indeed, from (\ref{eq:yager-binary}),
\[
\lim_{m_2(y)\rightarrow - \infty} \frac{m_{\text{\oyager}}(y)}{m_{\text{\oyager}}(x)} = \frac{m_1(y) (1-m_2(x)) + m_1(\Theta) m_2(y)}{m_1(x) (1-m_2(y)) + m_1(\Theta) m_2(x)} = - \frac{m_1(\Theta)}{m_1(x)},
\]
and similarly for the loci with $m_2(y) =$ const.

Nevertheless, the general principle of geometric construction of intersecting the linear spaces which are images of constant-mass loci still holds. 

\subsubsection{Geometry of disjunctive combination} \label{sec:future-combination-disjunctive}

Disjunctive combination is the natural, cautious dual of conjunctive/Dempster combination. The operator follows from the assumption that the consensus between two sources of evidence is best represented by the union of the supported hypotheses, rather than by their intersection.

Again, here we will briefly analyse its geometry in the case of a binary frame, and make conjectures about its general behaviour. We will first consider normalised belief functions, to later extend our analysis to the unnormalised case.

\paragraph{Global behaviour} Let us first understand the shape of the disjunctive conditional subspace. By definition,
\begin{equation}\label{eq:disjunctive-binary}
\begin{array}{c}
m_{\ocup}(x) = m_1(x) m_2(x), \quad m_{\ocup}(y) = m_1(y) m_2(y),
\\
m_{\ocup}(\Theta) = 1 -  m_1(x)m_2(x) - m_1(y) m_2(y),
\end{array}
\end{equation}
so that, adopting the usual vector notation $[Bel(x), Bel(y), Bel(\Theta)]'$,
\begin{equation}\label{eq:vertices-disjunctive}
\begin{array}{c}
Bel \ocup Bel_x = [m(x), 0, 1 - m(x)]',
\quad
Bel \ocup Bel_y = [0, m(y), 1- m(y)]',
\\
Bel \ocup Bel_\Theta = Bel_\Theta.
\end{array}
\end{equation}
The conditional subspace $\langle Bel \rangle_{\ocup}$ is thus the convex closure
\[
\langle Bel \rangle_{\ocup} = Cl(Bel, Bel \ocup Bel_x, Bel \ocup Bel_y)
\]
of the points (\ref{eq:vertices-disjunctive}): see Fig. \ref{fig:ocup}.
As in the Yager case,
\[
\begin{array}{lll}
& & 
Bel \ocup [\alpha Bel' + (1-\alpha) Bel''] 
\\
& = & \Big [ m(x) (\alpha m'(x) + (1-\alpha) m''(x)),
m(y) (\alpha m'(y) + (1-\alpha) m''(y)) \Big ]'
\\
& = & \alpha Bel \ocup Bel' + (1-\alpha) Bel \ocup Bel'',
\end{array}
\]
i.e., \emph{$\ocup$ commutes with affine combination}, at least in the binary case.

\begin{question} \label{que:ocup-commute}
Does disjunctive combination commute with affine combination in general belief spaces?
\end{question}

\begin{figure}[ht!]
\centering
\includegraphics[width=0.8\textwidth]{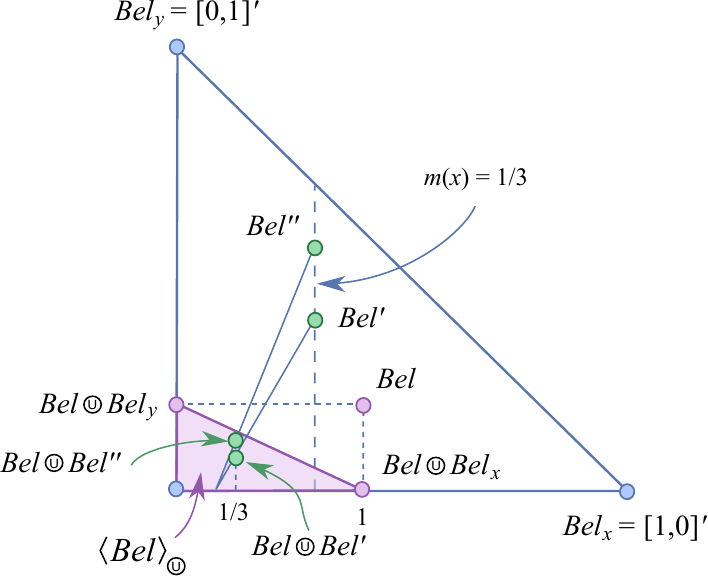}
\caption{Conditional subspace $\langle Bel \rangle_{\ocup}$ for disjunctive combination on a binary frame. The pointwise behaviour of $\ocup$ is also illustrated (see text).} \label{fig:ocup}
\end{figure}

\paragraph{Pointwise behaviour} As in the Yager case, for disjunctive combination the images of constant-mass loci are parallel to each other. In fact, they are parallel to the corresponding constant-mass loci and the coordinate axes (observe, in Fig. \ref{fig:ocup}, the locus $m(x) = \frac{1}{3}$ and its image in the conditional subspace $\langle Bel \rangle_{\ocup}$). 

We can prove the following.
\begin{theorem}
In the binary case $\Theta = \{x,y\}$, the lines joining $Bel'$ and $Bel \ocup Bel'$ for any $Bel' \in \mathcal{B}$ intersect in the point
\begin{equation} \label{eq:intersection-ocup}
\overline{m(x)} = m'(x) \frac{m(x) - m(y)}{1 - m(y)}, \quad \overline{m(y)} = 0.
\end{equation}
\end{theorem}
\begin{proof}
Recalling the equation of the line joining two points $(\chi_1,\upsilon_1)$ and $(\chi_2,\upsilon_2)$ of $\mathbb{R}^2$, with coordinates $(\chi,\upsilon)$,
\[
\upsilon - \upsilon_1 = \frac{\upsilon_2 - \upsilon_1}{\chi_2 - \chi_1} (\chi - \chi_1),
\]
we can identify the line joining $Bel'$ and $Bel \ocup Bel'$ as
\[
\upsilon - m'(y) = \frac{m(y)m'(y) - m'(y)}{m(x)m'(x) - m'(x)} (\chi - m'(x)).
\]
Its intersection with $\upsilon = 0$ is the point (\ref{eq:intersection-ocup}), which does not depend on $m'(y)$ (i.e., on the vertical location of $Bel'$ on the constant-mass loci).
\end{proof} 

A valid geometric construction for the disjunctive combination $Bel \ocup Bel'$ of two belief functions in $\mathcal{B}_2$ is provided by simple trigonometric arguments. Namely (see Fig. \ref{fig:ocup-construction}): 
\begin{enumerate}
\item
Starting from $Bel'$, find its orthogonal projection onto the horizontal axis, with coordinate $m'(x)$ (point 1).
\item
Draw the line with slope $45^\circ$ passing through this projection, and intersect it with the vertical axis, at coordinate $\upsilon = m'(x)$ (point 2).
\item
Finally, take the line $l$ passing through $Bel_y$ and the orthogonal projection of $Bel$ onto the horizontal axis, and draw a parallel line $l'$ through point 2 -- its intersection with the horizontal axis (point 3) is the $x$ coordinate $m(x)m'(x)$ of the desired combination.
\end{enumerate}
A similar construction (shown in magenta) allows us to locate the $y$ coordinate of the combination (as also shown in Fig. \ref{fig:ocup-construction}). 

\begin{figure}[ht!]
\centering
\includegraphics[width=0.9\textwidth]{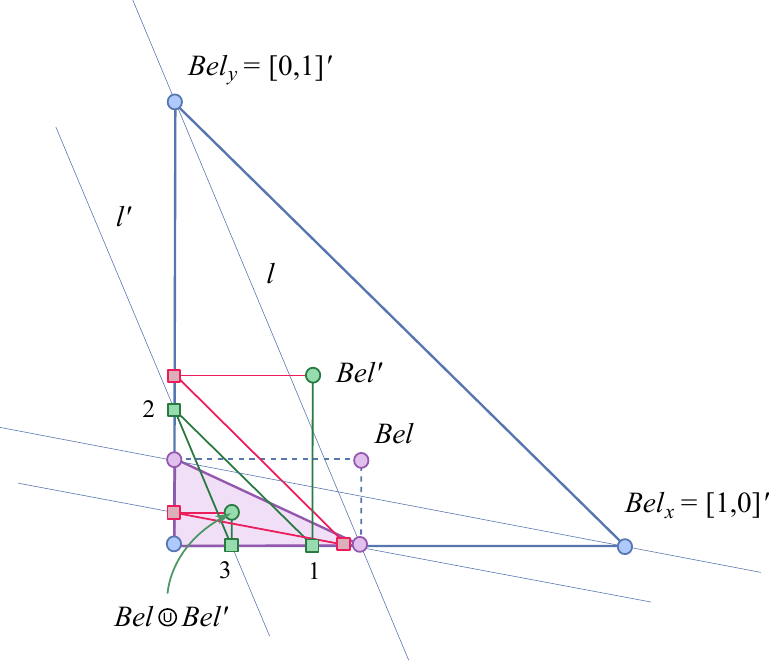}
\caption{Geometric construction for the disjunctive combination of two belief functions $Bel$, $Bel'$ on a binary frame.} \label{fig:ocup-construction}
\end{figure}

\subsubsection{Geometry of combination of unnormalised belief functions} \label{sec:future-combination-ubfs} 

In the case of unnormalised belief functions, Dempster's rule is replaced by the conjunctive combination. The disjunctive combination itself needs to be reassessed for UBFs as well.

In the unnormalised case, a distinction exists between the \emph{belief} measure
\[
Bel(A) = \sum_{\emptyset \neq B \subseteq A} m(B)
\]
and the \emph{believability} (in Smets's terminology) measure of an event $A$, denoted by $b(A)$, 
\begin{equation} \label{eq:believability-measure}
b(A) = \sum_{\emptyset \subseteq B \subseteq A} m(B).  
\end{equation}
Here we analyse the geometric behaviour of the latter, in which case $\emptyset$ is not treated as an exception: the case of belief measures is left for future work.
As $Bel(\Theta) = b(\Theta)=1$, as usual, we neglect the related coordinate and represent believability functions as points of a Cartesian space of dimension $|2^\Theta|-1$ (as $\emptyset$ cannot be ignored any more).

\paragraph{Conjunctive combination on a binary frame} 

In the case of a binary frame, the conjunctive combination of two belief functions $Bel_1$ and $Bel_2$, with BPAs $m_1$, $m_2$, yields
\begin{equation}\label{eq:conjunctive-binary-ubf}
\begin{array}{lll}
m_{\ocap}(\emptyset) & = & m_1(\emptyset) + m_2(\emptyset) - m_1(\emptyset) m_2(\emptyset)   + m_1(x) m_2(y) + m_1(y) m_2(x),
\\
m_{\ocap}(x) & = & m_1(x) (m_2(x) + m_2(\Theta)) + m_1(\Theta) m_2(x), 
\\
m_{\ocap}(y) & = & m_1(y) (m_2(y) + m_2(\Theta)) + m_1(\Theta) m_2(y), 
\\
m_{\ocap}(\Theta) & = & m_1(\Theta) m_2(\Theta).
\end{array}
\end{equation}

\paragraph{Conditional subspace for conjunctive combination}
The global behaviour of $\ocap$ in the binary (unnormalised) case can then be understood in terms of its conditional subspace. We have
\begin{equation} \label{eq:conditional-subspace-ocap}
\begin{array}{lll}
b \ocap b_\emptyset & = & b_\emptyset = [1,1,1]',
\\
b \ocap b_x & = & (m_1(\emptyset) + m_1(y)) b_\emptyset + (m_1(x) + m_1(\Theta)) b_x 
\\
&= & 
[m_1(\emptyset) + m_1(y), 1, m_1(\emptyset) + m_1(y)]'
\\
& = & b_1(y) b_\emptyset + (1 - b_1(y)) b_x,
\\
b \ocap b_y & = & (m_1(\emptyset) + m_1(x)) b_\emptyset + (m_1(y) + m_1(\Theta)) b_y 
\\
& = & 
[m_1(\emptyset) + m_1(x), m_1(\emptyset) + m_1(x), 1]'
\\
& = & b_1(x) b_\emptyset + (1 - b_1(x)) b_y,
\\
b \ocap b_\Theta & = & b,
\end{array}
\end{equation}
as $b_x = [0,1,0]'$, $b_y = [0,0,1]'$, $b_\emptyset = [1,1,1]'$ and $b_\Theta = [0,0,0]'$. 

From (\ref{eq:conditional-subspace-ocap}), we can note that the vertex $b \ocap b_x$ belongs to the line joining $b_\emptyset$ and $b_x$, with its coordinate given by the believability assigned by $b$ to the other outcome $y$. Similarly, the vertex $b \ocap b_y$ belongs to the line joining $b_\emptyset$ and $b_y$, with its coordinate given by the believability assigned by $b$ to the complementary outcome $x$ (see Fig. \ref{fig:conditional-subspace-ubf}).

\begin{figure}[ht!]
\centering
\includegraphics[width=0.8\textwidth]{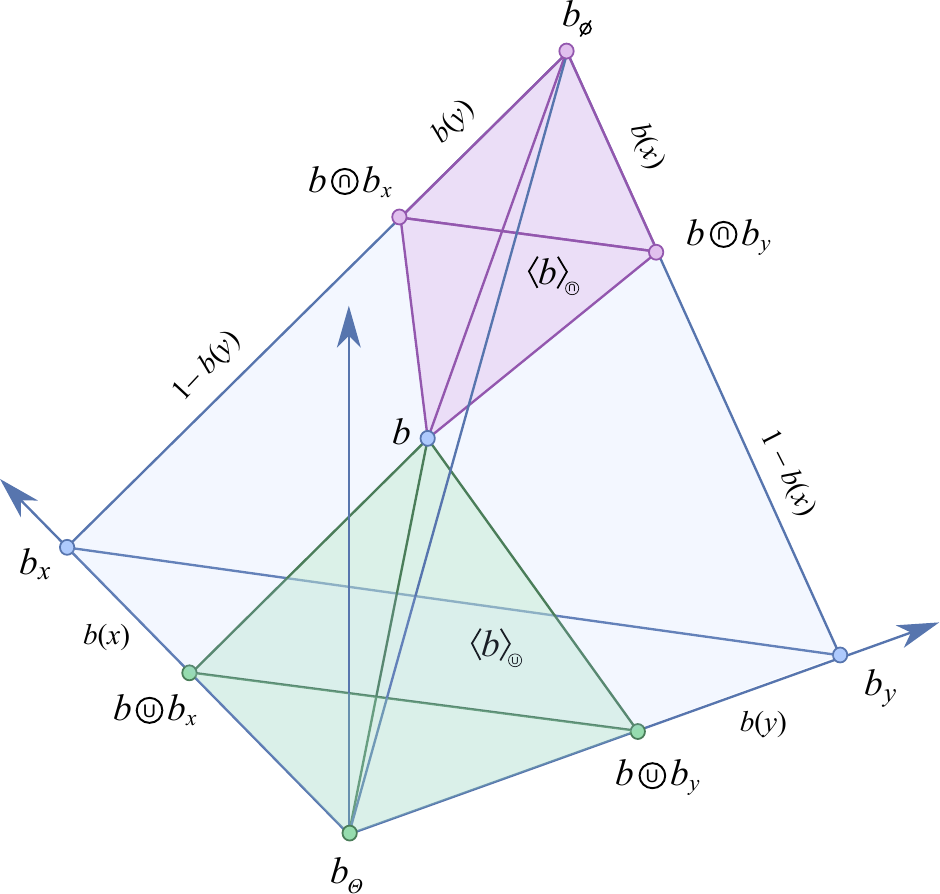}
\caption{Conditional subspaces induced by $\ocap$ and $\ocup$ in a binary frame, for the case of unnormalised belief functions.} \label{fig:conditional-subspace-ubf}
\end{figure}

\paragraph{Conditional subspace for disjunctive combination}

As for the disjunctive combination, it is easy to see that in the unnormalised case we get
\[
\begin{array}{llll}
b \ocup b_\Theta &= & b_\Theta, & \quad b \ocup b_x = b(x) b_x + (1 - b(x)) b_\Theta, 
\\
b \ocup b_\emptyset & = & b, & \quad
b \ocup b_y = b(y) b_y + (1 - b(y)) b_\Theta,
\end{array}
\]
so that the conditional subspace is as in Fig. \ref{fig:conditional-subspace-ubf}. Note that, in the unnormalised case, there is a unit element with respect to $\ocup$, namely $b_\emptyset$. 

In the unnormalised case, we can observe a clear symmetry between the conditional subspaces induced by disjunctive and conjunctive combination. 

\subsubsection{Research questions} \label{sec:future-combination-issues} 

A number of questions remain open after this preliminary geometric analysis of other combination rules on binary spaces, and its extension to the case of unnormalised belief functions.

\begin{question} \label{que:geometry-combination-disjunctive-pointwise}
What is the general pointwise geometric behaviour of disjunctive combination, in both the normalised and the unnormalised case?
\end{question}

A dual question concerns the conjunctive rule, as the alter ego of Dempster's rule in the unnormalised case.
\begin{question} \label{que:geometry-combination-conjunctive-pointwise}
What is the general pointwise geometric behaviour of conjunctive combination in the unnormalised case?
\end{question}

Bold and cautious rules \cite{denoeux07ai}, which are also inherently defined for unnormalised belief functions, are still to be geometrically understood.
\begin{question} \label{que:geometry-combination-bold-cautious}
What is the geometric behaviour of bold and cautious rules?
\end{question} 

\subsection{Geometry of general conditioning} \label{sec:future-conditioning} 

The author's analysis of geometric conditioning (see \cite{cuzzolin2021springer}, Chapter 15) is also still in its infancy. There, we analysed the case of classical Minkowski norms. It is natural to wonder what happens when we plug other norms into the optimisation problem that defines geometric conditioning:
\begin{equation} \label{eq:geometric-conditioning}
Bel_d(.|A) \doteq \arg \min_{Bel' \in \mathcal{B}_A} d(Bel,Bel'),
\end{equation}
where $\mathcal{B}_A$ is the \emph{conditioning event} $A$, i.e., the set of belief functions whose BPA assigns mass to subsets of $A$ only, and $d$ is a suitable distance function in the belief space.

Even more relevantly, the question of whether
geometric conditioning is a framework flexible enough to encompass general conditioning in belief calculus arises.  

\begin{question} \label{que:geometric-conditioning}
Can any major conditioning operator be interpreted as geometric (in the sense defined above) conditioning, i.e., as producing the belief function within the conditioning simplex at a minimum distance from the original belief function, with respect to an appropriate norm in the belief space?
\end{question}


\subsection{A true geometry of uncertainty} \label{sec:future-uncertainty}

A true geometry of uncertainty will require the ability to manipulate in our geometric language any (or most) forms of uncertainty measures \cite{cuzzolin2021uncertainty}. 

Probability and possibility measures are, as we know, special cases of belief functions: therefore, their geometric interpretation does not require any extension of the notion of a belief space. 
Most other uncertainty measures, however, are not special cases of belief functions -- in fact, a number of them are more
general than belief functions, such as probability intervals (2-monotone capacities), general monotone capacities and lower previsions.

Tackling these more general measures requires, therefore, extending the concept of a geometric belief space in order to encapsulate the most general such representation. In the medium term, the aim is to develop a general geometric theory of imprecise probabilities, encompassing capacities, random sets induced by capacity functionals, and sets of desirable gambles. 

\subsubsection{Geometry of capacities and random sets} \label{sec:future-uncertainty-capacities}

Developing a geometric theory of general monotone capacities appears to be rather challenging. We can gain some insight into this issue, however, from an analysis of 2-monotone capacities (systems of probability intervals) and, more specifically, when they are defined on binary frames.

\paragraph{Geometry of 2-monotone capacities} \label{sec:future-capacities} 

Recall that a capacity $\mu$ is called \emph{2-monotone} iff, for any $A,B$ in the relevant $\sigma$-algebra 
\[
\mu(A\cup B)+ \mu(A \cap B) \geq \mu(A) + \mu(B),
\]
whereas it is termed \emph{2-alternating} whenever
\[
\mu(A\cup B)+ \mu(A \cap B) \leq \mu(A) + \mu(B). 
\]
Belief measures are 2-monotone capacities, while their dual plausibility measures are 2-alternating ones.
More precisely, belief functions are infinitely monotone capacities, as they satisfy 
\begin{equation} \label{eq:infinite-order}
\displaystyle \mu \left (\bigcup_{j=1}^k A_j \right) \geq \sum_{\emptyset \neq K \subseteq [1,...,k]} (-1)^{|K|+1} \mu \left (\bigcap_{j \in K} A_j \right )
\end{equation}
\noindent for every value of $k \in \mathbb{N}$.

A well-known lemma by Chateauneuf and Jaffray states the following.
\begin{proposition} \label{pro:2-monotone}
A capacity $\mu$ on $\Theta$ is 2-monotone if and only if
\begin{equation} \label{eq:2-monotone}
\sum_{\{x,y\} \subseteq E \subseteq A } m(E) \geq 0 
\quad 
\forall x, y \in \Theta, \; \forall A \ni x,y,
\end{equation}
where $m$ is the M\"obius transform (\ref{moebius-capacity}) of $\mu$:
\begin{equation} \label{moebius-capacity}
m(A) = \sum_{B \subset A} (-1)^{|A - B|} \mu(B).
\end{equation}
\end{proposition} 

We can use Proposition \ref{pro:2-monotone} to show that the following theorem is true.
\begin{theorem} \label{the:2-monotone-convex}
The space of all 2-monotone capacities defined on a frame $\Theta$, which we denote by $\mathcal{M}^2_\Theta$, is convex.
\end{theorem}
\begin{proof}
Suppose that $\mu_1$ and $\mu_2$, with M\"obius inverses $m_1$ and $m_2$, respectively, are 2-monotone capacities (therefore satisfying the constraint (\ref{eq:2-monotone})). 
Suppose that a capacity $\mu$ is such that $\mu = \alpha \mu_1 + (1-\alpha) \mu_2$, $0 \leq \alpha \leq 1$.

Since M\"obius inversion is a linear operator, we have that
\[
\begin{array}{lll}
\displaystyle \sum_{\{x,y\} \subseteq E \subseteq A } m(E) & = & \displaystyle \sum_{\{x,y\} \subseteq E \subseteq A }
\alpha m_1 + (1-\alpha) m_2 
\\ \\
& = & \displaystyle
\alpha  \sum_{\{x,y\} \subseteq E \subseteq A } m_1
+ (1-\alpha) \sum_{\{x,y\} \subseteq E \subseteq A } m_2 \geq 0,
\end{array}
\]
since 
\[
\sum_{\{x,y\} \subseteq E \subseteq A } m_1\geq 0,
\quad
\sum_{\{x,y\} \subseteq E \subseteq A } m_2 \geq 0
\]
for all $x, y \in \Theta$, $\forall A \ni x,y$.
\end{proof} 

Note that Theorem \ref{the:2-monotone-convex} states the convexity of the \emph{space} of 2-monotone capacities:
it is, instead, well known that every 2-monotone capacity corresponds to a convex set of probability measures in their simplex.

It is useful to get an intuition about the problem by considering the case of 2-monotone capacities defined on the ternary frame $\Theta = \{x,y,z\}$.
There, the constraints (\ref{eq:2-monotone}) read as
\begin{equation} \label{eq:2-monotone-ternary1}
\left \{
\begin{array}{l}
m(\{x,y\}) \geq 0, \\
m(\{x,y\}) + m(\Theta) \geq 0, \\
m(\{x,z\}) \geq 0, \\
m(\{x,z\}) + m(\Theta) \geq 0, \\
m(\{y,z\}) \geq 0, \\
m(\{y,z\}) + m(\Theta) \geq 0,
\end{array}
\right.
\end{equation}
whereas for $\Theta = \{x,y,z,w\}$ they become
\[
\begin{array}{ll}
\left \{
\begin{array}{l}
m(\{x_i,x_j\}) \geq 0,
\\
m(\{x_i,x_j\}) + m(\{x_i,x_j,x_k\}) \geq 0,
\\
m(\{x_i,x_j\}) + m(\{x_i,x_j,x_l\}) \geq 0,
\\
m(\{x_i,x_j\}) + m(\{x_i,x_j,x_k\}) + m(\{x_i,x_j,x_l\}) + m(\Theta) \geq 0,
\end{array}
\right.
&
\end{array}
\]
for each possible choice of $x_i,x_j,x_k,x_l \in \Theta$. 

The constraints (\ref{eq:2-monotone-ternary1}) are clearly equivalent to
\begin{equation} \label{eq:2-monotone-ternary2}
\left \{
\begin{array}{l}
m(\{x,y\}) \geq 0, \\
m(\{x,z\}) \geq 0, \\
m(\{y,z\}) \geq 0, \\
m(\Theta) \geq  - \min \Big \{ m(\{x,y\}), m(\{x,z\}), m(\{y,z\}) \Big \},
\end{array}
\right.
\end{equation}
so that the space $\mathcal{M}^2$ of 2-monotone capacities on $\{x,y,z\}$ is the set of capacities
\[
\begin{array}{ll} 
\mathcal{M}^2_{\{x,y,z\}} = \bigg \{ 
&
\displaystyle
\mu : 2^\Theta \rightarrow [0,1] \; \bigg | \; \sum_{A \subseteq \Theta} m(A) =1, m(A)\geq 0 \; \forall
A : |A| = 2,
\\
&
\displaystyle
m(\Theta) \geq  - \min \Big \{ m(\{x,y\}), m(\{x,z\}), m(\{y,z\}) \Big \}
\bigg \}.
\end{array}
\] 

We can begin to understand the structure of this convex set and its vertices by studying the shape of the convex subset of $\mathbb{R}^3$
\begin{equation} \label{eq:toy-geometry-capacities}
z \geq - \min \{x,y\}, \quad x,y \geq 0, \quad x+y+z =1.
\end{equation}
The latter is the intersection of the following loci:
\[
\begin{array}{c}
V = \{x+y+z = 1\}, \quad H_1 = \{x \geq 0\}, \quad H_2 = \{y \geq 0\}, 
\\ \\
H'_3 = \{ z \geq - x, x \leq y \}, \quad H''_3 = \{ z \geq - y, y \leq x \},
\end{array}
\]
the first of which is a linear space, whereas the others are half-spaces determined by a plane in $\mathbb{R}^3$.
The resulting convex body is the convex closure 
\[
Cl([1,0,0],[0,1,0],[0,0,1],[1,1,-1]) 
\]
depicted in Fig. \ref{fig:2-monotone-toy}. The four vertices are associated with all possible binary configurations of the variables $x$ and $y$, with the $z$ component determined by the other constraints. 
One can notice that vertices are allowed to have negative mass values on some component (focal elements of $\Theta$). As a result, we can conjecture that the space of 2-monotone capacities on a frame $\Theta$, while being embedded in a Cartesian space of the same dimension as the belief space, will contain the belief space by virtue of its extra vertices (e.g. $[1,1,-1]$ in Fig. \ref{fig:2-monotone-toy}).

\begin{figure}[ht!]
\centering
\includegraphics[width=0.6\textwidth]{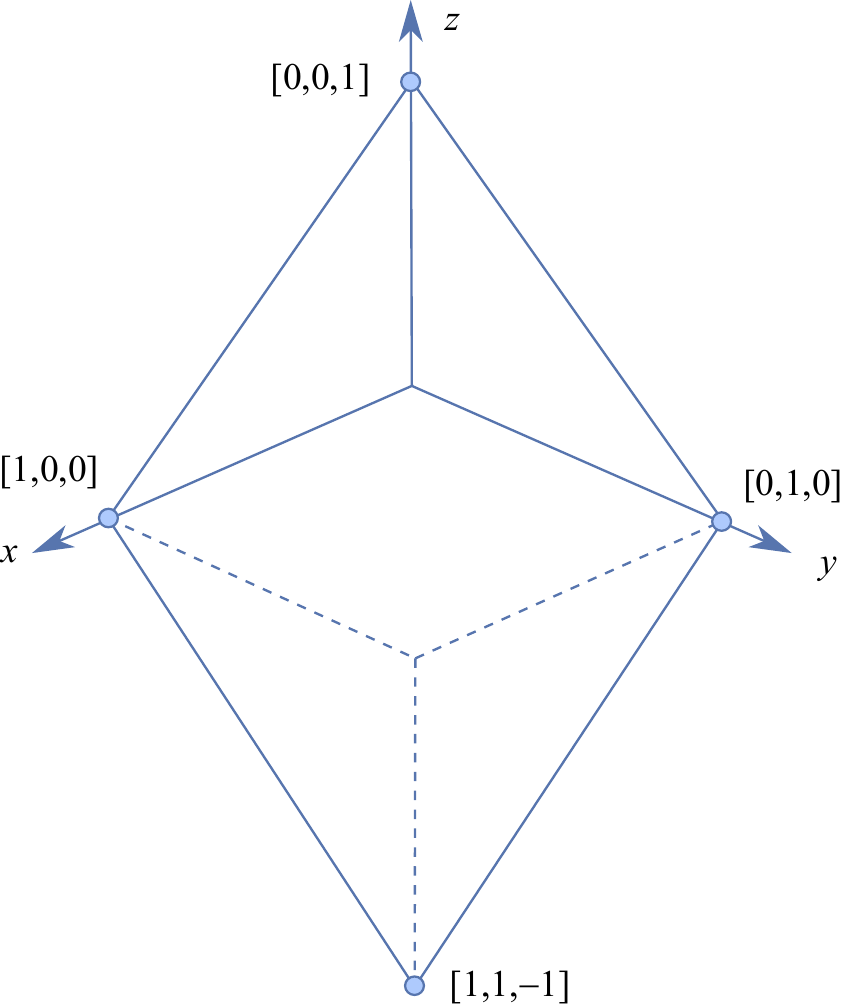}
\caption{The geometry of the convex body (\ref{eq:toy-geometry-capacities}) in $\mathbb{R}^3$.} \label{fig:2-monotone-toy}
\end{figure}

As the geometry of 2-monotone capacities involves a hierarchical set of constraints involving events of increasing cardinality, however, the analysis of the toy problem (\ref{eq:toy-geometry-capacities}) cannot be directly generalised to it.  

\paragraph{Geometry of capacities and random sets}

The geometry of capacities is related to that of random sets, via the capacity functional $T_X$. 
Recall that
\begin{definition} \label{def:capacity-functional}
A functional $T_X : \mathcal{K} \rightarrow [0, 1]$ given by
\begin{equation} \label{eq:capacity-functional}
T_X (K) = P (\{X \cap K \neq \emptyset \}), 
\quad 
K \in \mathcal{K},
\end{equation}
is termed the \emph{capacity functional} of a random closed set $X$. 
\end{definition}
We will analyse this point further in the near future.

\subsubsection{Geometry of functionals} \label{sec:future-uncertainty-gambles}
Any study of the geometry of gambles (widely employed in imprecise probability \cite{walley91book}) requires an analysis of the geometry of \emph{functionals}, i.e., mappings from a space of functions to the real line. The same holds for other uncertainty representations which make use of functions on the real line, such as MV algebras \cite{dubuc2010representation}
and belief measures on fuzzy events \cite{smets81degree}.
For a very interesting recent publication (2016) on the geometry of MV algebras in terms of rational polyhedra, see \cite{Mundici2016}.

\paragraph{Geometric functional analysis} 
\cite{Holmes1975,holmes2012geometric}\footnote{
\url{https://link.springer.com/book/10.1007/978-1-4684-9369-6}.} is a promising tool from this perspective.
A central question of geometric functional analysis is: what do typical $n$-dimensional structures look
like when $n$ grows to infinity? This is exactly what happens when we try to generalise the notion of a belief space to continuous domains. One of the main tools of geometric functional analysis is the theory of \emph{concentration of measure}, which offers a geometric view of the limit theorems of probability theory. 
Geometric functional analysis thus bridges three areas: { functional analysis, convex geometry and
probability theory \cite{Vershynin2009}.

In more detail, a \emph{norm} on a vector space $X$ is a function $\|\| : X \rightarrow \mathbb{R}$ that satisfies (i) non-negativity, (ii) homogeneity and (iii) the triangle inequality. A \emph{Banach space} \cite{fabian2013functional} is a complete normed space.
A convex set $K$ in $\mathbb{R}^n$ is called \emph{symmetric} if $K = -K$ (i.e., $x \in K$ implies $-x \in K$).
\begin{proposition} \label{pro:banach} (Banach spaces and symmetric convex bodies) Let $X = (\mathbb{R}^n, \|\cdot \|)$ be a
Banach space. Then its unit ball $B_X$ is a symmetric convex body in $\mathbb{R}^n$. 
Further, let $K$ be a symmetric convex body in $\mathbb{R}^n$. Then $K$ is the unit ball of some normed space with norm (its \emph{Minkowski functional}, Fig. \ref{fig:minkowski}(a))
\[
\|x\|_K = \inf \left \{t>0, \frac{x}{t} \in K \right \}.
\]
\end{proposition}

The correspondence between unit balls in Banach spaces and convex bodies established in Proposition \ref{pro:banach} allows arguments from convex geometry to be used in
functional analysis and vice versa. 
\begin{figure}[ht!]
\centering
\begin{tabular}{cc}
\includegraphics[width=0.35\textwidth]{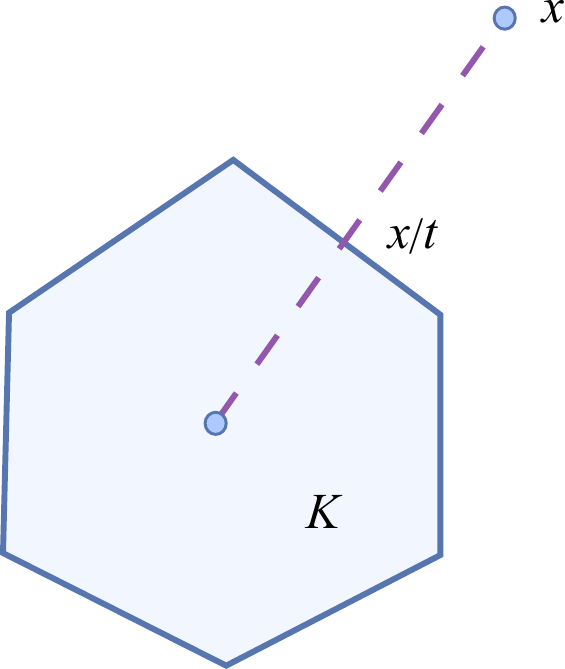}
&
\includegraphics[width=0.56\textwidth]{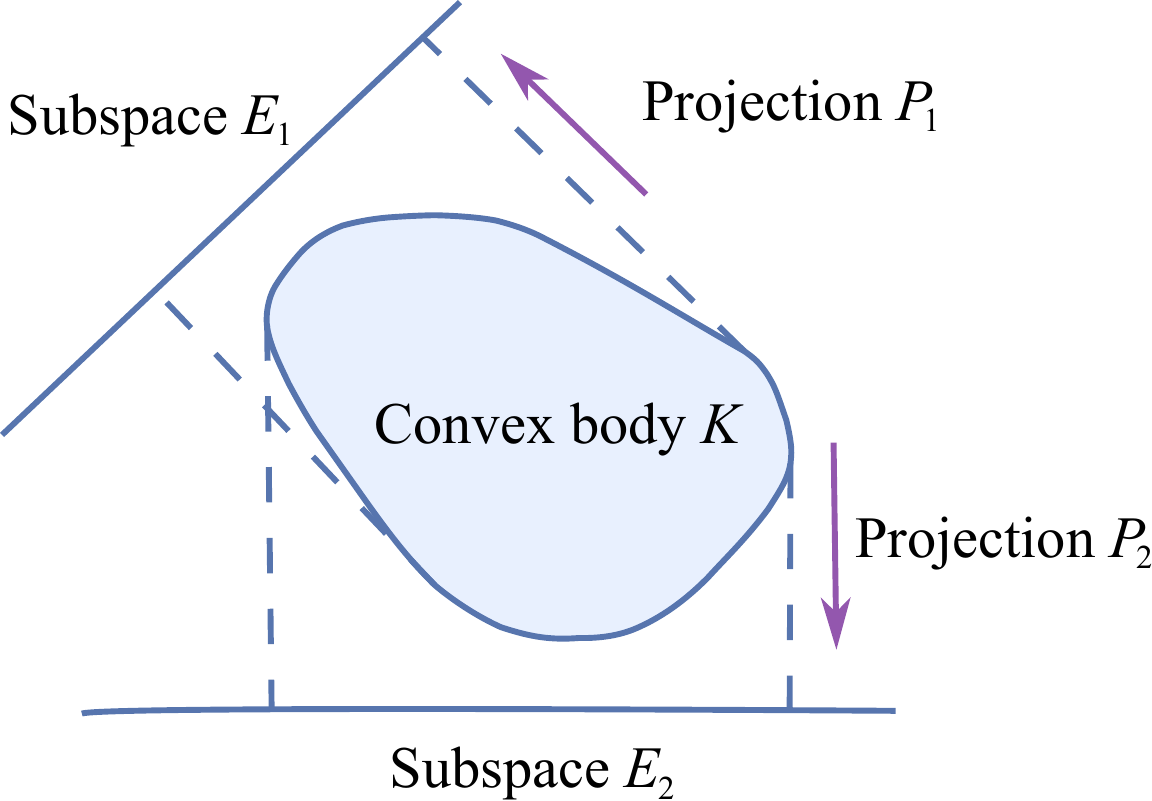}
\\
(a) & (b)
\end{tabular}
\caption{(a) The Minkowski functional of a symmetric convex body in $\mathbb{R}^2$ \cite{Vershynin2009}. (b) Projections of a convex body.} \label{fig:minkowski}
\end{figure}
Using Proposition \ref{pro:banach}, one can show that subspaces of Banach spaces correspond
to \emph{sections} of convex bodies, and quotient spaces correspond to \emph{projections} of convex bodies.

The phenomenon of concentration of measure was the driving force for the early development of geometric functional analysis. It tells us that anti-intuitive phenomena occur in high
dimensions. For example, the `mass' of a high-dimensional ball is concentrated only in a thin band around any equator. This reflects the philosophy that metric and measure should be treated very differently: a set can have a large diameter but carry little mass \cite{Vershynin2009}.

The shape of random projections of symmetric convex bodies onto $k$-dimensional subspaces  (Fig. \ref{fig:minkowski}(b)) can also be studied within geometric functional analysis. It is easy to see that a projection captures only the extremal points of $K$. Hence, one quantity affected by projections is the diameter of the body. 
\begin{proposition} (Diameters under random projections, \cite{Giannopoulos2004}) \label{pro:diameters}
Let $K$ be a symmetric convex body in $\mathbb{R}^n$, and let $P$ be a projection onto a random subspace of dimension $k$. Then, with probability at least $1 - e^{-k}$,
\[
\text{diam}(PK) \leq C \left ( M^*(K) + \sqrt{\frac{k}{n}} \text{diam}(K) \right ),
\]
where $M^*(K)$ is the mean width of $K$.
\end{proposition}

\subsubsection{Research questions} \label{sec:future-uncertainty-issues}

As this is one of the most challenging topics for future research, most issues are still wide open, in particular the following.
\begin{question}
What is the most appropriate framework for describing the geometry of monotone capacities?
\end{question}
\begin{question} 
Can we provide a representation in the theory of convex bodies of the set of 2-monotone capacities similar to that of the belief space, i.e., by providing an analytical expression for the vertices of this set?
\end{question}
\begin{question} 
How can the geometry of monotone capacities be exploited to provide a geometric theory of general random sets?
\end{question}

Finally, we have the following question.
\begin{question} 
How can we frame the geometry of (sets of desirable) gambles and MV algebras in terms of functional spaces?
\end{question}

\subsection{Fancier geometries} \label{sec:future-fancier}

Representing belief functions as mere vectors of mass or belief values is not entirely satisfactory.
Basically, when this is done all vector components are indistinguishable, whereas they in fact correspond to values assigned to subsets of $\Theta$ of different cardinality. 

\subsubsection{Exterior algebras} \label{sec:future-exterior}

Other geometrical representations of belief functions on finite spaces can nevertheless be imagined, which take into account the qualitative difference between events of different cardinalities.

For instance, the \emph{exterior algebra} \cite{Grassmann1844} is a promising tool, as it allows us to encode focal elements of cardinality $k$ with exterior powers (\emph{k-vectors}) of the form $x_1 \wedge x_2 \wedge \cdots \wedge x_k$ and belief functions as linear combinations of decomposable $k$-vectors, namely
\[
\vec{bel} = \sum_{A = \{x_1,...,x_k\} \subset \Theta} m(A) \; x_1 \wedge x_2 \wedge \cdots \wedge x_k.
\] 
Under an exterior algebra representation, therefore, focal elements of different cardinalities are associated with qualitatively different objects (the $k$-vectors) which encode the geometry of volume elements of different dimensions (e.g. the area of parallelograms in $\mathbb{R}^2$).

The question remains of how Dempster's rule may fit into such a representation, and what its relationship with the exterior product operator $\wedge$ is.

\subsubsection{Capacities as isoperimeters of convex bodies} \label{sec:future-isoperimeter}

\emph{Convex bodies}, as we have seen above, 
are also the subject of a fascinating field of study \cite{brazitikos2014geometry}. 

In particular, any convex body in $\mathbb{R}^n$ possesses $2^n$ distinct orthogonal projections onto the $2^n$ subspaces generated by all possible subsets of coordinate axes.
It is easy to see that, given a convex body ${K}$ in the Cartesian space $\mathbb{R}^n$, endowed with coordinates $x_1, \ldots, x_n$, the function $\nu$ that assigns to each subset of coordinates $S = \{ x_{i_1}, \ldots, x_{i_k} \}$ the (hyper)volume $\nu(S)$ of the orthogonal projection ${K}|S$ of $\mathcal{K}$ onto the linear subspace generated by $S = \{ x_{i_1}, \ldots, x_{i_k} \}$ is potentially a capacity. In order for that to happen, volumes in linear spaces of different dimensions need to be compared, and conditions need to be enforced in order for volumes of projections onto higher-dimensional subspaces to be greater than those of projections on any of their subspaces.

As we have just seen, symmetric convex bodies are in a 1--1 relationship with (balls in) Banach spaces, which can then be used as a tool to represent belief measures as projections of symmetric convex bodies.

The following research questions arise.

\begin{question} \label{que:projections-monotone}
Under what conditions is the above capacity monotone?
\end{question}

\begin{question} \label{que:projections-belief}
Under what conditions is this capacity a belief function (i.e. an infinitely monotone capacity)?
\end{question}

\subsection{Relation to integral and stochastic geometry}

\subsubsection{Integral geometry, or geometric probability} 
This is the theory of invariant measures \cite{Matheron75} (with respect to continuous groups of transformations of a space onto itself) on sets consisting of submanifolds of the space (for example lines, planes, geodesics and convex surfaces; in other words, manifolds preserving their type under the transformations in question). Integral geometry has been constructed for various spaces, primarily Euclidean, projective and homogeneous spaces \cite{santalo1950integral}, and 
arose in connection with refinements of statements of problems in geometric probabilities.

In order to introduce an invariant measure, one needs to identify a function depending on the coordinates of the space under consideration whose integral over some region of the space is not changed under any continuous coordinate transformation belonging to a specified Lie group. This requires finding an integral invariant of the Lie group. 

A classic problem of integral geometry is Buffon's \emph{needle problem} (dating from 1733). 
\begin{proposition}
Some parallel lines on a wooden floor are a distance $d$ apart from each other. A needle of length $l <d$ is randomly dropped onto the floor (Fig. \ref{fig:buffon}, left). Then the probability that the needle will touch one of the lines is
\[
P= \frac{2l}{\pi d}.
\]
\end{proposition}

\begin{figure}[ht!]
\centering
\includegraphics[width=0.85\textwidth]{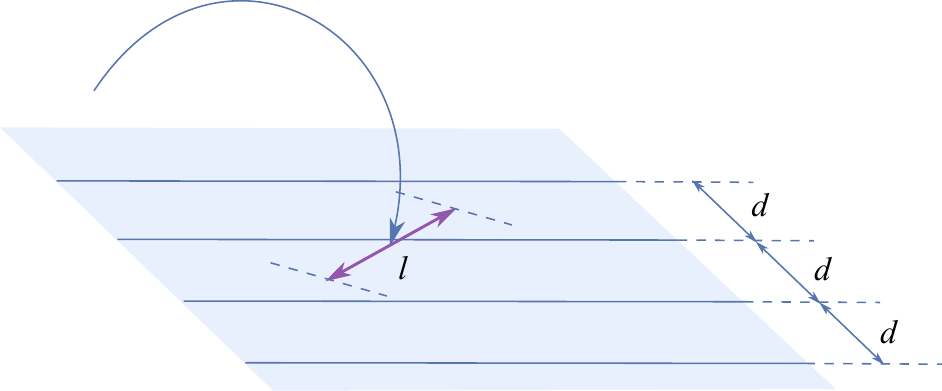}
\caption{Buffon's needle problem.} \label{fig:buffon}
\end{figure}

Another classical problem is Firey's \emph{colliding dice} problem (from 1974). 
\begin{proposition}
Suppose $\Omega_1$ and $\Omega_2$ are disjoint unit cubes in $\mathbb{R}^3$. In a random collision,
the probability that the cubes collide edge-to-edge slightly exceeds the
probability that the cubes collide corner-to-face. In fact,
\[
0.54 \approxeq P(\text{edge-to-edge collision}) > P(\text{corner-to-face collision}) \approxeq 0.46.
\]
\end{proposition}

Roughly speaking, integral geometry is the problem of computing probabilities of events which involve geometric loci, as opposed to the geometric approach to uncertainty supported in \cite{cuzzolin2021springer}, which is about understanding the geometric properties of (uncertainty) measures.
The main technical tool is Poincar\'e's integral-geometric formula.

After representing lines in the plane using coordinates $(p,\theta)$, $\cos(\theta) x + \sin(\theta) y = p$,
the \emph{kinematic measure} for such lines is defined as $\mathrm{d}K = \mathrm{d}p \wedge \mathrm{d} \theta$, as the measure on a set of lines which is invariant under rigid motion, i.e., the Jacobian of the coordinate transformation is equal to 1.
Now, let $C$ be a piecewise $\mathcal{C}^1$ curve in the plane. Given a line $L$ in the plane, let $n(L \cap C)$ be the number of intersection points. If $C$ contains a linear segment and $L$ agrees with that segment,
$n(C \cap L) = \infty$.
\begin{proposition} (Poincar\'e formula for lines, 1896)
Let $C$ be a piecewise $\mathcal{C}^1$ curve in the plane. Then the measure of unoriented lines
meeting $C$, counted with multiplicity, is given by
\[
2 L(C) = \int_{ L: L \cap C \neq \emptyset } n(C \cap L) \; \mathrm{d} K(L).
\]
\end{proposition}

Integral-geometric formulae hold for convex sets $\Omega$ as well. 
Since $n(L \cap \partial \Omega)$, where $\partial \Omega$ is the boundary of $\Omega$, is either zero or two for $\mathrm{d}K$-almost all lines $L$, the measure of unoriented lines that meet the convex set $\Omega$ is given by
\[
L(\partial \Omega) = \int_{L: L \cap \Omega \neq \emptyset} \; \mathrm{d}K.
\]
One can then prove the following. 
\begin{proposition} (Sylvester's problem, 1889)
Let $\Theta \subset \Omega$ be two bounded convex sets in the plane. Then the probability that a random line meets $\Theta$ given that it meets $\Omega$ is $P = {L(\partial \Theta)} / {L(\partial \Omega)}$.
\end{proposition}
\begin{corollary}
Let $C$ be a piecewise $\mathcal{C}^1$ curve contained in a compact convex set $\Omega$. 
The expected number of intersections with $C$ of all random lines that meet $\Omega$ is
\[
\mathbb{E}(n) = \frac{2 L(C)}{L(\partial \Omega)}. 
\]
\end{corollary}

Integral geometry potentially provides a powerful tool for designing new geometric representations of belief measures, on the one hand. 
On the other hand, we can foresee generalisations of integral geometry in terms of non-additive measures and capacities.

\begin{question} \label{que:geometric-probability-extended}
How does geometric probability generalise when we replace standard probabilities with other uncertainty measures?
\end{question}

\subsubsection{Stochastic geometry}

Stochastic geometry is a branch of probability theory which deals with set-valued random elements. It describes the behaviour of random configurations such as random graphs, random networks, random cluster processes, random unions of convex sets, random mosaics and many other random geometric structures \cite{Hug2016}.
The topic is very much intertwined with the theory of random sets, and Matheron's work in particular.
The name appears to have been coined by David Kendall and Klaus Krickeberg \cite{chiu2013stochastic}.

Whereas geometric probability
considers a fixed number of random objects of a fixed shape and studies their interaction
when some of the objects move randomly, since the 1950s the focus has switched to models involving a random number of randomly chosen geometric objects. 
As a consequence, the notion of a \emph{point process} started to play a
prominent role in this field.

\begin{definition} \label{def:point-process} 
A \emph{point process} $\eta$ is a measurable map from some probability space $(\Omega,\mathcal{F},P)$ to the locally finite subsets of a Polish space $\X$ (endowed with a suitable $\sigma$-algebra), called the state space. The \emph{intensity measure} of $\eta$, evaluated at a measurable set $A \subset \X$, is defined by $\mu(A) = E[\eta(A)]$, and equals the mean number of elements of $\eta$ lying in $A$.
\end{definition}

Typically, $\X$ is either $\mathbb{R}^n$, the space of compact subsets of $\mathbb{R}^n$ or the set of all affine subspaces there (the Grassmannian).
A point process can be written as $\eta = \sum_{i=1}^\tau \delta_{\zeta_i}$, where $\tau$ is a random variable taking values in $\mathbb{N}_0 \cup \{\infty\}$, and $\zeta_1,\zeta_2, \ldots$ is a sequence of random points in $\X$ (Fig. \ref{fig:point-process}). `Random point field' is sometimes considered a more appropriate terminology, as it avoids confusion with stochastic processes. 

A point process is, in fact, a special case of a \emph{random element}, a concept introduced by Maurice Fr\'echet \cite{frechet1948elements} as a generalisation of the idea of a random variable to situations in which the outcome is a complex structure (e.g., a vector, a function, a group \cite{celler1995generating}, a series or a subset). 

\begin{figure}[ht!]
\centering
\includegraphics[width=1\textwidth]{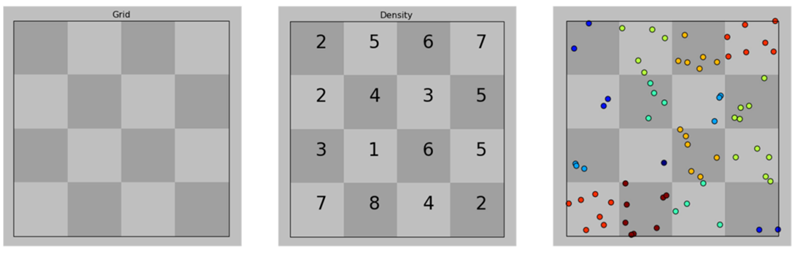}
\\
\begin{tabular}{ccc}
(a) & \hspace{34mm} (b) & \hspace{34mm} (c)
\end{tabular}
\caption{Examples of point processes (from 
{\url{http://stats.stackexchange.com/questions/16282/how-to-generate-nd-point-process}} \small : (a) grid; (b) Poisson; (c) uniform.} \label{fig:point-process}
\end{figure}

\begin{definition} \label{def:random-element}
Let ${\displaystyle (\Omega ,{\mathcal {F}},P)}$ be a probability space, and $(E,{\mathcal {E}})$ a measurable space. A \emph{random element} with values in $E$ is a function $X: \Omega \rightarrow E$ which is $({\mathcal {F}},{\mathcal {E}})$-measurable, that is, a function $X$ such that for any $B\in {\mathcal {E}}$, the pre-image of $B$ lies in ${\mathcal {F}}$.
\end{definition}
\noindent Random elements with values in $E$ are sometimes called $E$-valued random variables (hence the connection with random sets).

Geometric processes \cite{lam2007geometric} are defined as point processes on manifolds that represent spaces of events. Thus, processes of straight lines in the plane are defined as point processes on a M\"obius strip (as the latter represents the space of straight lines in $\mathbb{R}^2$).
Processes on manifolds form a more general concept; here, once again, stochastic geometry is linked with the theory of random sets \cite{Matheron75}. The `intrinsic volume' of convex bodies and Minkowski combinations of them (\emph{polyconvex} bodies) is a topic of stochastic geometry which relates to the questions we posed in Section \ref{sec:future-isoperimeter}.

A  peculiarity which distinguishes stochastic geometry from the theory of random sets is the interest stochastic geometry has for geometric processes with distributions that are invariant relative to groups acting on the fundamental space $\X$.
For instance, one may examine
the class of those processes of straight lines on $\mathbb{R}^2$ which are invariant relative to the group of Euclidean motions of the plane.  

\section{High-impact developments} \label{sec:future-applications} 

The theory of random sets is very well suited to tackling a number of high-profile applications, which range from the formulation of cautious climate change prediction models (Section \ref{sec:future-climate}) to the design of novel machine learning tool generalising classical mathematical frameworks (e.g. max-entropy classifiers and random forests, Section \ref{sec:future-machine-learning}) and to the analysis of the very foundations of statistical machine learning (Section \ref{sec:future-slt}).

Exciting work has recently been conducted on Bayesian formulations of the foundations of quantum mechanics \cite{PhysRevA.65.022305,Benavoli2017QPL}. Further analyses of the problem in the extended framework of imprecise probabilities are likely to follow.

\subsection{Climate change} \label{sec:future-climate} 

Climate change\footnote{\url{https://link.springer.com/journal/10584}.}
is a paramount example of a problem which requires predictions to be made under heavy uncertainty, due to the imprecision associated with any climate model, the long-term nature of the predictions involved and the second-order uncertainty affecting the statistical parameters involved.

A typical question a policymaker may ask a climate scientist is, for example \cite{Rougier2007}, \\

\noindent ``\emph{What is the probability that a doubling of atmospheric CO$_2$ from pre-industrial levels
will raise the global mean temperature by at least $2^\circ C$?}'' \\ 

Rougier \cite{Rougier2007} nicely outlined a Bayesian approach to climate modelling and prediction, in which a predictive distribution for a future climate is found by conditioning the future
climate on observed values for historical and current climates.
More recent papers on the Bayesian modelling of climate change include, for instance, \cite{min2007probabilistic,smith2009bayesian,katzfuss2017bayesian}.

A number of challenges arise when modelling climate change in a Bayesian framework:
\begin{itemize}
\item
in climate prediction, the collection of uncertain quantities for which the
climate scientist must specify prior probabilities can be large;
\item
specifying a prior distribution over climate vectors is very challenging.
\end{itemize}
According to Rougier, considering that people are spending thousands of hours on collecting climate data and constructing climate models, it is surprising that little attention is being devoted to quantifying our judgements about how the two are related. 

In this section, climate is represented as a vector of measurements $y$, collected at a given time. Its components include, for instance, the level of CO$_2$ concentration on the various points of a grid.
More precisely, the \emph{climate vector} $y = (y_h,y_f)$ collects together both historical and present ($y_h$) and future ($y_f$) climate values.
A \emph{measurement error} $e$ is introduced to take into account errors due to, for instance, a seasick technician or atmospheric turbulence. The actual measurement vector is therefore
\[
z \doteq y_h + e.
\]

The Bayesian treatment of the problem makes use of a number of assumptions. For starters, we have the following axioms.
\begin{axiom}
Climate and measurement error are independent: $e \bot y$.
\end{axiom}
\begin{axiom} \label{ax:error-gaussian}
The measurement error is Gaussian distributed, with mean $\vec{0}$ and covariance $\Sigma^e$: $e \sim \mathcal{N}(\vec{0},\Sigma^e)$.
\end{axiom}

Thanks to these assumptions, the predictive distribution for the climate given the measured values $z = \tilde{z}$ is
\begin{equation} \label{eq:climate-prediction-bayesian}
p(y| z = \tilde{z}) \sim \mathcal{N}(\tilde{z} - y_h|\vec{0},\Sigma^e) p(y),
\end{equation}
which requires us to {specify a prior distribution} for the climate vector $y$ itself.

\subsubsection{Climate models}

The
choice of such a prior $p(y)$ is extremely challenging, because $y$ is such a large collection of quantities, and these component quantities are linked by complex interdependencies, such as those arising from the laws of nature.
The role of the climate model is then to induce a distribution for the climate itself, and plays the role of a {parametric model in statistical inference}. 

Namely, a \emph{climate model} is a deterministic mapping from a collection of parameters $x$ (equation coefficients, initial conditions, forcing functions) to a vector of measurements (the `climate'), namely
\begin{equation} \label{eq:climate-mapping}
x \to y = g(x),
\end{equation}
where $g$ belongs to a predefined `model space' $\mathcal{G}$. 
Climate scientists call an actual value of $g(x)$, computed for some specific set of parameter values $x$, a \emph{model evaluation}. The reason is that the analytical mapping (\ref{eq:climate-mapping}) is generally not known, and only images of specific input parameters can be computed or sampled (at a cost).
A climate scientist considers, on a priori grounds based on their past experience, that some choices of $x$ are better than others, i.e., that there exists a set of parameter values $x^*$ such that
\[
y = g(x^*) + \epsilon^*,
\]
where $\epsilon^*$ is termed the `model discrepancy'.

\subsubsection{Prediction via a parametric model}

The difference between the climate vector and any model evaluation can be decomposed into two parts:
\[
y - g(x) = g(x^*) - g(x) + \epsilon^*.
\]
The first part is a contribution that may be reduced by a better choice of the model $g$;
the second part is, in contrast, an irreducible contribution that arises from the model's own imperfections.
Note that $x^*$ is not just a statistical parameter, though, for it relates to physical quantities, so that climate scientists have a clear intuition about its effects.
Consequently, scientists may be able to exploit their expertise to provide a prior $p(x^*)$ on the input parameters. 

In this Bayesian framework for climate prediction, two more assumptions are needed.
\begin{axiom} \label{ax:discrepancy-independence}
The `best' input, discrepancy and measurement error are mutually (statistically) independent: 
\[ x^* \bot \epsilon^* \bot e. \]
\end{axiom}
\begin{axiom} \label{ax:discrepancy-gaussian}
The model discrepancy $\epsilon^*$ is Gaussian distributed, with mean $\vec{0}$ and covariance $\Sigma^\epsilon$.
\end{axiom}

Axioms \ref{ax:discrepancy-independence} and \ref{ax:discrepancy-gaussian} then allow us to compute the desired climate prior as
\begin{equation} \label{eq:climate-prior-bayesian}
p(y) = \int \mathcal{N}(y - g(x^*)|\vec{0},\Sigma^\epsilon) p(x^*) \; \mathrm{d} x^*,
\end{equation}
which can be plugged into (\ref{eq:climate-prediction-bayesian}) to yield a Bayesian prediction of future climate values.

In practice, as we have said, the climate model function $g(.)$ is not known -- we only possess a sample of model evaluations $\{ g(x_1), \ldots, g(x_n) \}$.
We call the process of tuning the covariances $\Sigma^\epsilon, \Sigma^e$ and checking the validity of the Gaussianity assumptions (Axioms \ref{ax:error-gaussian} and \ref{ax:discrepancy-gaussian}) \emph{model validation}. 
This can be done by using (\ref{eq:climate-prediction-bayesian}) to predict past/present climates $p(z)$, and applying some hypothesis testing to the result. 
If the observed value $\tilde{z}$ is in the tail of the distribution, the model parameters (if not the entire set of model assumptions) need to be corrected.

As Rougier admits \cite{Rougier2007}, responding to bad validation results is not straightforward.

\subsubsection{Model calibration} 

Assuming that the model has been validated, it needs to be `{calibrated}', i.e., we need to
find the desired `best' value $x^*$ of the model's parameters.

Under Axioms 3--6, we can compute
\[
p(x^*|z=\tilde{z}) = p(z=\tilde{z}|x^*) = \mathcal{N}(\tilde{z} = g(x^*)| \vec{0}, \Sigma^{\epsilon} + \Sigma^e) p(x^*).
\]
As we know, maximum a posteriori estimation 
could be applied
to the above posterior distribution: however, the presence of multiple modes might make it ineffective.

As an alternative, we can apply full Bayesian inference to compute
\begin{equation} \label{eq:climate-posterior-prediction}
p(y_f | z = \tilde{z}) = \int p(y_f| x^*, z= \tilde{z}) p(x^*|z=\tilde{z}) \; \mathrm{d} x^*,
\end{equation}
where $p(y_f| x^*, z= \tilde{z})$ is Gaussian with a mean which depends on $\tilde{z} - g(x)$.
The posterior prediction (\ref{eq:climate-posterior-prediction}) highlights two routes for climate data to impact on future climate predictions:
\begin{itemize}
\item
by concentrating the distribution $p(x^*|z=\tilde{z})$ relative to the prior $p(x^*)$, depending on both the quantity and the quality of the climate data;
\item 
by shifting the mean of $p(y_f| x^*, z= \tilde{z})$ away from $g(x)$, depending on the size of the difference $\tilde{z} - g(x)$.
\end{itemize}

\subsubsection{Role of model evaluations}

Let us go back to the initial question: what is the probability that a doubling of atmospheric CO$_2$ will raise the global mean temperature by at least $2^\circ C$ by 2100?

Let $Q \subset \mathcal{Y}$ be the set of climates $y$ for which the global mean temperature is at least $2^\circ C$ higher in 2100. The probability of the event of interest can then be computed by integration, as follows:
\[
Pr(y_f \in Q | z = \tilde{z}) = \int f(x^*) p(x^*| z = \tilde{z}) \; \mathrm{d} x^*.
\]
The quantity $f(x) \doteq \int_Q \mathcal{N}(y_f|\mu(x),\Sigma) \; \mathrm{d} y_f$
can be computed directly, while the overall integral in $dx^*$ requires numerical integration, for example via a naive Monte-Carlo approach,
\[
\int \cong \frac{\sum_{i=1}^n f(x_i)}{n}, \quad x_i \sim p(x^*| z=\tilde{z}),
\]
or by weighted sampling
\[
\int \cong \frac{\sum_{i=1}^n w_i f(x_i)}{n}, \quad x_i \sim p(x^*| z=\tilde{z}),
\]
weighted by the likelihood $w_i \propto p(z=\tilde{z}|x^* = x_i)$.
A trade-off appears, in the sense that sophisticated models which take a long time to evaluate may not provide enough samples for the prediction to be statistically significant -- however,
they may make the prior $p(x^*)$ and covariance $\Sigma^\epsilon$ easier to specify.

\subsubsection{Modelling climate with belief functions} 

As we have seen, a number of issues arise in making climate inferences in the Bayesian framework.

Numerous assumptions are necessary to make calculations practical, which are only weakly related to the actual problem to be solved.
Although the prior on \emph{climates} is reduced to a prior $p(x^*)$ on the parameters of a climate \emph{model}, there is no obvious way of picking the latter (confirming our criticism of Bayesian inference in the Introduction).
In general,
it is far easier to say what choices are definitively wrong (e.g. uniform priors) than to specify `good' ones. Finally,
significant parameter tuning is required (e.g. to set the values of $\Sigma^\epsilon$, $\Sigma^e$ and so on).

Formulating a climate prediction method based on the theory of belief functions will require quite a lot of work, but a few landmark features can already be identified: 
\begin{itemize}
\item
such a framework will have to
avoid committing to priors $p(x^*)$ on the `correct' climate model parameters;
\item
it will likely
use a climate model as a parametric model to infer either a belief function on the space of climates $\mathcal{Y}$
\item
or a belief function on the space of parameters (e.g. covariances, etc.) of the distribution on $\mathcal{Y}$.
\end{itemize}

\begin{question} \label{que:climate-change}
Can a climate model allowing predictions about future climate be formulated in a belief function framework, in order to better describe the huge Knightian uncertainties involved?
\end{question}

In recent times, efforts to model climate change using belief functions have been conducted \cite{abdallah2013using,kriegler2005imprecise}, for instance by Kriegler and co-authors \cite{kriegler2005utilizing}, but more could and should be done.

\subsection{Machine learning} \label{sec:future-machine-learning} 

\emph{Machine learning} provides a clear case for the mathematical modelling of epistemic uncertainty, in particular via random set and belief function theory, as the issue of learning models capable of making predictions which are robust to variations in the process generating the data is foundational in machine learning.

\subsubsection{Modelling uncertainty in machine learning}

A significant amount of work has been done in this sense over the years.
Denoeux and co-authors, 
{together with others, \cite{Liu2012291}, 
have much contributed to}
unsupervised learning and clustering, in particular  
\cite{Masson20081384}. 
The combination of the outcomes of multiple classifiers has been investigated within the theory of evidence
\cite{155943} 
\cite{Rogova1994777}, 
and evidential approaches to
decision trees \cite{elouedi00decision} and K-nearest neighbour classifiers \cite{Denoeux2008classic} {have been proposed.} 
However, most of the sparse attempts made so far to incorporate uncertainty in neural predictions \cite{hullermeier2021aleatoric,geifman2017selective} have been made from a Bayesian stance, which only captures aleatory uncertainty, and do not fully model epistemic uncertainty.

A recent wave of papers have brought forward various distinct approaches to evidential deep classification \cite{tong2021evidential}, in particular via a second-order Dirichlet {parameter} representation \cite{sensoy2018evidential}, imprecise neural networks \cite{dutta2023distributionally}, 
{interval neural networks \cite{shariatmadar2022introduction} or convolutional neural networks outputting random sets \cite{manchingal2023random,manchingal2022epistemic} rather than conventional probability vectors.}

Here, we focus on a specific type of classifier as a representative of the issues involved.

\subsubsection{Maximum entropy classification}

\emph{Maximum entropy classifiers}\footnote{\url{http://www.kamalnigam.com/papers/maxent-ijcaiws99.pdf}.} are a widespread tool in machine learning. There, the Shannon entropy of a sought joint (or conditional) probability distribution $p(C_k|x)$ of observed data $x$ and an associated class $C_k \in \mathcal{C} = \{C_1, \ldots, C_K\}$ is maximised, following the maximum entropy principle that the least informative distribution matching the available evidence should be chosen. Once one has picked a set of \emph{feature functions}, chosen to efficiently encode the training information, the joint distribution is subject to the constraint that the empirical expectation of the feature functions equals the expectation associated with it.

Clearly, the two assumptions that (i) the training and test data come from the same probability distribution, and that (ii) the empirical expectation of the training data is correct, and the model expectation should match it, are rather strong, and work against generalisation power.
A more robust learning approach can instead be sought by generalising max-entropy classification in a belief function setting. We take the view here that a training set does not provide, in general, sufficient information to precisely estimate the joint probability distribution of the class and the data. We assume instead that a belief measure can be estimated, providing lower and upper bounds on the joint probability of the data and class.

As in the classical case, an appropriate measure of entropy for belief measures is maximised. In opposition to the classical case, however, the empirical expectation of the chosen feature functions is assumed to be merely `compatible' with the lower and upper bounds associated with the sought belief measure. This leads to a constrained optimisation problem with inequality constraints, rather than equality ones, which needs to be solved by looking at the Karush-Kuhn-Tucker conditions.

A shorter version of what follows was published in \cite{cuzzolin18belief-maxent}.

\subsubsection{Generalised max-entropy classifiers} \label{sec:future-maxentropy}

Given a training set in which each observation is attached to a class, namely
\begin{equation} \label{eq:maxent-training-set}
\mathcal{D} = \Big \{ (x_i,y_i), i = 1, \ldots, N \Big | x_i \in X, y_i \in \mathcal{C} \Big \},
\end{equation}
a set $M$ of \emph{feature maps} is designed,
\begin{equation} \label{eq:feature-maps}
\phi(x,C_k) = [\phi_1(x,C_k), \ldots , \phi_M(x,C_k)]',
\end{equation}
whose values depend on both the object observed and its class. Each feature map $\phi_m :  X \times \mathcal{C} \rightarrow \mathbb{R}$ is then a random variable, whose expectation is obviously
\begin{equation} \label{eq:feature-expectation-theoretical} 
E[\phi_m] = \sum_{x,k} p(x,C_k) \phi_m(x,C_k).
\end{equation}
In opposition, the \emph{empirical} expectation of $\phi_m$ is
\begin{equation} \label{eq:feature-expectation-empirical}
\hat{E}[\phi_m] = \sum_{x,k} \hat{p}(x,C_k) \phi_m(x,C_k),
\end{equation}
where $\hat{p}$ is a histogram constructed by counting occurrences of the pair $(x,C_k)$ in the training set:
\[
\hat{p}(x,C_k) = \frac{1}{N} \sum_{(x_i,y_i) \in \mathcal{D}} \delta(x_i = x \wedge y_i = C_k).
\]

The theoretical expectation (\ref{eq:feature-expectation-theoretical}) can be approximated by decomposing via Bayes' rule $p(x,C_k) = p(x) p(C_k|x)$, and approximating the (unknown) prior of the observations $p(x)$ by the empirical prior $\hat{p}$, i.e., the histogram of the observed values in the training set:
\[
\tilde{E}[\phi_m] = \sum_{x,k} \hat{p}(x) p(C_k|x) \phi_m(x,C_k). 
\]
For values of $x$ that are not observed in the training set, $\hat{p}(x)$ is obviously equal to zero.

\begin{definition} \label{def:maxent-classifier}
Given a training set (\ref{eq:maxent-training-set}) related to the problem of classifying $x \in X$ as belonging to one of the classes $\mathcal{C} = \{ C_1, \ldots, C_K \}$, the \emph{max-entropy classifier} is the conditional probability $p^*(C_k|x)$ which solves the following constrained optimisation problem:
\begin{equation} \label{eq:maxent-classifier-1}
p^*(C_k|x) \doteq \arg\max_{p(C_k|x)} H_s(P),
\end{equation}
where $H_s$ is the traditional Shannon entropy, subject to
\begin{equation} \label{eq:maxent-classifier-2}
\tilde{E}_{p}[\phi_m] = \hat{E}[\phi_m] \quad \forall m = 1, \ldots, M.
\end{equation}
\end{definition}

The constraint (\ref{eq:maxent-classifier-2}) requires the classifier to be consistent with the empirical frequencies of the features in the training set, while seeking the least informative probability distribution that does so (\ref{eq:maxent-classifier-1}).

The solution of the maximum entropy classification problem (Definition \ref{def:maxent-classifier}) is the so-called \emph{log-linear model},
\begin{equation} \label{eq:log-linear-model}
p^*(C_k|x) = \frac{1}{Z_\lambda(x)} e^{\sum_m \lambda_m \phi_m(x,C_k)}, 
\end{equation}
where $\lambda = [ \lambda_1, \ldots, \lambda_M ]'$ are the Lagrange multipliers associated with the linear constraints (\ref{eq:maxent-classifier-2}) in the above constrained optimisation problem, and $Z_\lambda(x)$ is a normalisation factor.  

The classification function associated with (\ref{eq:log-linear-model}) is
\[
y(x) = \arg\max_k \sum_m \lambda_m \phi_m(x,C_k),
\]
i.e., $x$ is assigned to the class which maximises the linear combination of the feature functions with coefficients $\lambda$.

\paragraph{Generalised maximum entropy optimisation problem}

When generalising the maximum entropy optimisation problem (Definition \ref{def:maxent-classifier}) to the case of belief functions, we need to (i) choose an appropriate measure of entropy for belief functions as the objective function, and (ii) revisit the constraints that the (theoretical) expectations of the feature maps are equal to the empirical ones computed over the training set.

As for (i), 
manifold generalisations of the Shannon entropy to the case of belief measures exist (\cite{cuzzolin2021springer}, Chapter 4). 
As for (ii), it is sensible to require that the empirical expectation of the feature functions is bracketed by the lower and upper expectations associated with the sought belief function $Bel: 2^{X \times \mathcal{C}} \rightarrow [0,1]$, namely
\begin{equation} \label{eq:maxent-lower-upper-expectations}
\sum_{(x,C_k) \in X \times \mathcal{C}} Bel(x,C_k) \phi_m(x,C_k) \leq \hat{E}[\phi_m] \leq \sum_{(x,C_k) \in X \times \mathcal{C}} Pl(x,C_k) \phi_m(x,C_k).
\end{equation}
Note that this makes use only of the 2-monotonicity of belief functions, as only probability intervals on singleton elements of $X \times \mathcal{C}$ are considered. 

\begin{question} \label{que:maxent-events}
Should constraints of this form be enforced on all possible subsets $A \subset X \times \mathcal{C}$, rather than just singleton pairs $(x,C_k)$? This is linked to the question of what information a training set actually carries (see Section \ref{sec:future-slt}).
\end{question}

More general constraints would require extending the domain of feature functions to set values -- we will investigate this idea in the near future.

We can thus pose the generalised maximum-entropy optimisation problem as follows \cite{cuzzolin18belief-maxent}. 

\begin{definition} \label{def:maxent-classifier-generalised}
Given a training set (\ref{eq:maxent-training-set}) related to the problem of classifying $x \in X$ as belonging to one of the classes $\mathcal{C} = \{ C_1, \ldots, C_K \}$, the \emph{maximum random set entropy classifier} is the joint belief measure $Bel^*(x,C_k) : 2^{X \times \mathcal{C}} \rightarrow [0,1]$ which solves the following constrained optimisation problem:
\begin{equation} \label{eq:maxent-classifier-generalised-1}
Bel^*(x,C_k) \doteq \arg\max_{Bel(x,C_k)} H(Bel),
\end{equation}
subject to the constraints
\begin{equation} \label{eq:maxent-classifier-generalised-2}
\sum_{x,k} Bel(x,C_k) \phi_m(x,C_k) \leq \hat{E}[\phi_m] \leq \sum_{x,k} Pl(x,C_k) \phi_m(x,C_k) \;\; \forall m=1, \ldots, M,
\end{equation}
where $H$ is an appropriate measure of entropy for belief measures.
\end{definition}

\begin{question} \label{que:least-committed-classifier}
An alternative course of action is to pursue the \emph{least committed}, rather than the maximum-entropy, belief function satisfying the constraints (\ref{eq:maxent-classifier-generalised-2}). This is left for future work.
\end{question}

\paragraph{KKT conditions and sufficiency} 

As the optimisation problem in Definition \ref{def:maxent-classifier-generalised} involves inequality constraints (\ref{eq:maxent-classifier-generalised-2}), as opposed to the equality constraints of traditional max-entropy classifiers (\ref{eq:maxent-classifier-2}), we need to analyse the KKT (recall Definition \ref{def:kkt}) \cite{Karush1939} necessary conditions for a belief function $Bel$ to be an optimal solution to the problem.

Interestingly, the KKT conditions are also sufficient whenever the objective function $f$ is concave, the inequality constraints $g_i$ are continuously differentiable convex functions and the equality constraints $h_j$ are affine.\footnote{More general sufficient conditions can be given in terms of \emph{invexity} \cite{Ben-Israel1986invexity} requirements. A function $f:\mathbb{R}^n \rightarrow \mathbb{R}$ is `invex' if there exists a vector-valued function $g$ such that $f(x) - f(u) \geq g(x,u) \nabla f(u)$. Hanson \cite{HANSON1981545} proved that if the objective function and constraints are invex with respect to the same function $g$, then the KKT conditions are sufficient for optimality.}
This provides us with a way to tackle the generalised maximum entropy problem.

\paragraph{Concavity of the entropy objective function}

As for the objective function, 
we recalled
that several generalisations of the standard Shannon entropy to random sets are possible \cite{Jirousek2016entropy}.
It is a well-known fact that the Shannon entropy is a concave function of probability distributions, represented as vectors of probability values.\footnote{{\url{http://projecteuclid.org/euclid.lnms/1215465631}}}
A number of other properties of concave functions are well known: any linear combination of concave functions is concave, a monotonic and concave function of a concave function is still concave, and the logarithm is a concave function.

Relevantly,
the transformations which map mass vectors to vectors of belief (and commonality) values are linear, as they can be expressed in the form of matrices.
In particular, $\vec{bel} = BfrM \vec{m}$, where $BfrM$ is a matrix whose $(A,B)$ entry is
\[
BfrM (A,B) =  \left \{ \begin{array}{ll} 1 & B \subseteq A, 
\\ 
0 & \text{otherwise}. \end{array} \right .
\]
The same can be said of the mapping $\vec{q} = QfrM \vec{m}$ between a mass vector and the associated commonality vector.
As a consequence, belief, plausibility and commonality are all linear (and therefore concave) functions of a mass vector.

Using this matrix representation, it is easy to conclude that several of the entropies proposed by various authors 
for belief measures
are indeed concave. In particular, Smets's specificity measure 
\[
H_t = \sum_A \log ({1} / {Q(A)})
\]
is concave, as a linear combination of concave functions. Nguyen's entropy 
\[
H_n = - \sum_A m(A) \log(m(A)) = H_s(m) 
\]
is also concave as a function of $m$, as the Shannon entropy of a mass assignment. Dubois and Prade's measure $H_d = \sum_A m(A) \log(|A|)$ 
is also concave with respect to $m$, as a linear combination of mass values.
Straighforward applications of the Shannon entropy function to $Bel$ and $Pl$,
\begin{equation} \label{eq:shannon-entropy-belief}
\begin{array}{c}
\displaystyle H_{Bel}[m] = H_s[Bel] = \sum_{A \subseteq \Theta} Bel(A) \log \left ( \frac{1}{Bel(A)} \right)
\end{array}
\end{equation}
and
\[
\begin{array}{c} \displaystyle
H_{Pl}[m] = H_s[Pl] = \sum_{A \subseteq \Theta} Pl(A) \log \left ( \frac{1}{Pl(A)} \right ),
\end{array}
\]
are also trivially concave, owing to the concavity of the entropy function and to the linearity of the mapping from $m$ to $Bel, Pl$.

Drawing conclusions about other entropy measures 
is less immediate, as they involve products of concave functions (which are not, in general, guaranteed to be concave). 

\paragraph{Convexity of the interval expectation constraints}

As for the constraints (\ref{eq:maxent-classifier-generalised-2}) of the generalised maximum entropy problem, we first note that (\ref{eq:maxent-classifier-generalised-2}) can be decomposed into the following pair of constraints:
\[
\begin{array}{lll}
g_m^1(m) & \doteq & \displaystyle \sum_{x,k} Bel(x,C_k) \phi_m(x,C_k) - \hat{E}[\phi_m] \leq 0,
\\ \\
g_m^2(m) & = & \displaystyle \sum_{x,k} \phi_m(x,C_k) [\hat{p}(x,C_k) - Pl(x,C_k)]  \leq 0,
\end{array}
\]
for all $m = 1, \ldots, M$.
The first inequality constraint is a linear combination of linear functions of the sought mass assignment $m^* : 2^{X \times \mathcal{C}} \rightarrow [0,1]$ (since $Bel^*$ results from applying a matrix transformation to $m^*$). 
As 
\[
\vec{pl} = 1 - J \vec{bel} = 1 - J BfrM \vec{m},
\]
the constraint on $g_m^2$ is also a linear combination of mass values. Hence, as linear combinations, the constraints on $g_m^1$ and $g_m^2$ are both concave and convex.

\paragraph{KKT conditions for the generalised maximum entropy problem} 

We can conclude the following.
\begin{theorem} \label{the:generalised-max-entropy}
If either $H_t, H_n,H_d,H_{Bel}$ or $H_{Pl}$ is adopted as the measure of entropy, the generalised maximum entropy optimisation problem (Definition \ref{def:maxent-classifier-generalised}) has a concave objective function and convex constraints. Therefore, the KKT conditions are sufficient for the optimality of its solution(s).
\end{theorem} 

Let us then compute the KKT conditions (see Definition \ref{def:kkt}) for the Shannon-like entropy (\ref{eq:shannon-entropy-belief}). Condition 1 (stationarity), applied to the sought optimal belief function $Bel^* : 2^{X \times \mathcal{C}} \rightarrow [0,1]$, reads as
\[
\nabla H_{Bel}(Bel^*) = \sum_{m=1}^M \mu_m^1 \nabla g_m^1(Bel^*) +
\mu_m^2 \nabla g_m^2 (Bel^*).
\]
The components of $\nabla H_{Bel}(Bel)$ are the partial derivatives of the entropy with respect to the mass values $m(\overline{B})$, for all $\overline{B} \subseteq \Theta = X \times \mathcal{C}$. They read as
\[
\begin{array}{lll}
\displaystyle \frac{\partial  H_{Bel}}{\partial  m(\overline{B})} 
& = & \displaystyle 
\frac{\partial}{\partial  m(\overline{B})} 
\sum_{A \supseteq \overline{B}} 
\left [ - \left (  \sum_{B \subseteq A} m(B) \right ) \log \left (  \sum_{B \subseteq A} m(B) \right ) \right ] 
\\ \\
& = & \displaystyle
- \sum_{A \supseteq \overline{B}}
\left [
1 \cdot \log \left (  \sum_{B \subseteq A} m(B) \right ) +  \left (  \sum_{B \subseteq A} m(B) \right ) \frac{1}{ \left (  \sum_{B \subseteq A} m(B) \right )}
\right ]
\\ \\
& = & \displaystyle
- \sum_{A \supseteq \overline{B}}
\left [1 + 
\log \left (  \sum_{B \subseteq A} m(B) \right )
\right ]
=
- \sum_{A \supseteq \overline{B}} [1 + \log Bel(A)].
\end{array} 
\]
For $\nabla g_m^1(Bel)$, we easily have
\begin{equation} \label{eq:deriv-g1m}
\begin{array}{lll}
\displaystyle 
\frac{\partial g_m^1 }{\partial  m(\overline{B})} 
& = & \displaystyle
\frac{\partial }{\partial  m(\overline{B})} \sum_{(x,C_k) \in \Theta} Bel(x,C_k) \phi_m(x,C_k) - \hat{E}[\phi_m]
\\ \\
& = & \displaystyle
\frac{\partial }{\partial  m(\overline{B})} \sum_{(x,C_k) \in \Theta} m(x,C_k) \phi_m(x,C_k) - \hat{E}[\phi_m]
\\ \\
& = & \displaystyle
\left \{ 
\begin{array}{ll}
\phi_m(x,C_k) & \overline{B} = \{(x,C_k)\},
\\
0 & \text{otherwise}.
\end{array}
\right.
\end{array}
\end{equation} 
Note that, if we could define feature functions over non-singleton subsets $A \subseteq \Theta$, (\ref{eq:deriv-g1m}) would simply generalise to
\[
\frac{\partial g_m^1}{\partial  m(\overline{B})}  
=
\phi(\overline{B}) \quad \forall \overline{B} \subseteq \Theta.
\]
As for the second set of constraints,
\[
\begin{array}{lll}
\displaystyle \frac{\partial g_m^2}{\partial  m(\overline{B})} & = &
\displaystyle
\frac{\partial }{\partial  m(\overline{B})} \sum_{(x,C_k) \in \Theta} \phi_m (x,C_k) [\hat{p}(x,C_k) - Pl(x,C_k)] 
\\ \\
& = & \displaystyle 
\frac{\partial }{\partial  m(\overline{B})} \sum_{(x,C_k) \in \Theta} \phi_m (x,C_k) \left [ \hat{p}(x,C_k) - \sum_{B \cap \{(x,C_k)\} \neq \emptyset} m(B) \right ]
\\ \\
& = & \displaystyle 
 \frac{\partial }{\partial  m(\overline{B})} 
 \left (
 - \sum_{(x,C_k) \in \Theta} \phi_m (x,C_k) \sum_{B \supseteq \{(x,C_k)\}} m(B) 
 \right )
\\ \\
& = & \displaystyle  
 - \sum_{(x,C_k) \in \overline{B}} \phi_m (x,C_k).
 \end{array}
\] 

Putting everything together, the KKT stationarity conditions for the generalised maximum entropy problem amount to the following system of equations:
\begin{equation} \label{eq:kkt-stationarity-maxent}
\left \{
\begin{array}{ll}
\displaystyle
- \sum_{A \supseteq \overline{B}} [1 + \log Bel(A)]
=
\sum_{m=1}^M \phi_m (\overline{x},\overline{C}_k) [\mu_m^1 - \mu_m^2],
&
\overline{B} = \{ (\overline{x},\overline{C}_k) \},
\\
\displaystyle
- \sum_{A \supseteq \overline{B}} [1 + \log Bel(A)]
=
\sum_{m=1}^M \mu_m^2 \sum_{(x,C_k) \in \overline{B}} \phi_m (x,C_k),
&
\overline{B} \subset X \times \mathcal{C},
|\overline{B}| > 1.
\end{array}
\right .
\end{equation}
The other conditions are (\ref{eq:maxent-classifier-generalised-2}) (primal feasibility), $\mu_m^1, \mu_m^2 \geq 0$ (dual feasibility) and complementary slackness,
\begin{equation} \label{eq:kkt-slackness-maxent}
\begin{array}{c}
\displaystyle
\mu_m^1 \left [
\sum_{(x,C_k) \in \Theta} Bel(\{(x,C_k)\}) \phi_m(x,C_k) - \hat{E}[\phi_m]
\right ] = 0,
\\ \\
\displaystyle
\mu_m^2 \left [
\sum_{(x,C_k) \in \Theta} \phi_m(x,C_k) \Big [ \hat{p}(x,C_k) - Pl(\{(x,C_k)\}) \Big ]
\right ] = 0
\end{array}
\end{equation}
In the near future we will address the analytical form of the solution of the above system of constraints, and its alternative versions associated with different entropy measures.

\begin{question}
Derive the analytical form of the solution to the above generalised maximum entropy classification problem.
\end{question}

\begin{question}
How does this analytical solution compare with the log-linear model solution to the traditional maximum entropy problem?
\end{question}

\begin{question}
Derive and compare the constraints and analytical solutions for the alternative generalised maximum entropy frameworks obtained by plugging in other concave generalised entropy functions.
\end{question}

\subsection{Generalising statistical learning theory}  \label{sec:future-slt}

Current machine learning paradigms focus on explaining the observable outputs in the training data (`overfitting'), which may, for instance, lead an autonomous driving system to perform well in validation training sessions but fail catastrophically when tested in the real world. Methods are often designed to handle a specific set of testing conditions, and thus are unable to predict how a system will behave in a radically new setting (e.g., how would a smart car \cite{Geiger:2012:WRA:2354409.2354978} cope with driving in extreme weather conditions?). With the pervasive deployment of machine learning algorithms in `mission-critical' artificial intelligence systems for which failure is not an option (e.g. smart cars navigating a complex, dynamic environment, robot surgical assistants capable of predicting the surgeon's needs even before they are expressed, or smart warehouses monitoring the quality of the work conducted: see Fig. \ref{fig:ml-wild}), it is imperative to ensure that these algorithms behave predictably `in the wild'.

Vapnik's classical \emph{statistical learning theory} \cite{vapnik1998statistical,vapnik2013nature,vapnik1999overview} is effectively useless for model selection, as the bounds on generalisation errors that it predicts are too wide to be useful, and rely on the assumption that the training and testing data come from the same (unknown) distribution. As a result, practitioners regularly resort to cross-validation\footnote{{\url{https://en.wikipedia.org/wiki/Cross-validation_(statistics)}}.} to approximate the generalisation error and identify the parameters of the optimal model. 

\begin{figure}[ht!]
\begin{tabular}{cc}
\includegraphics[height=3.8cm]{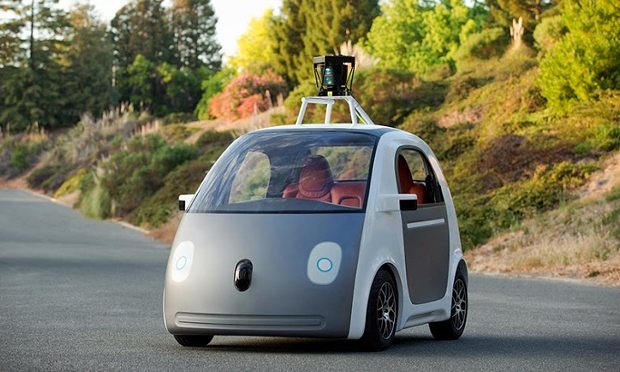}
&
\includegraphics[height=3.8cm]{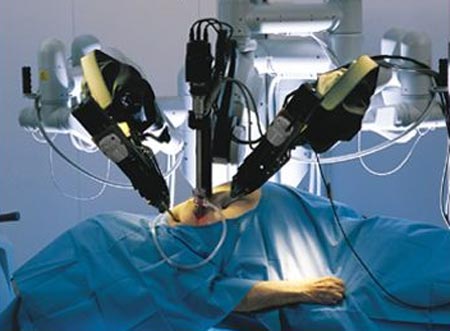}
\end{tabular}
\caption{New applications of machine learning to unconstrained environments (e.g. smart vehicles, left, or robotic surgical assistants, right) require new, robust foundations for statistical learning theory.} \label{fig:ml-wild}
\end{figure}

\subsubsection{PAC learning} \label{sec:future-pac} 

Statistical learning theory \cite{vapnik1998statistical} considers the problem of predicting an output $y \in \mathcal{Y}$ given an input $x \in \mathcal{X}$, by means of a mapping $h : \mathcal{X} \rightarrow \mathcal{Y}$, $h \in \mathcal{H}$ called a \emph{model} or \emph{hypothesis}, which lives in a specific model space $\mathcal{H}$. The error committed by a model is measured by a \emph{loss function} $l:(\mathcal{X} \times \mathcal{Y}) \times \mathcal{H} \rightarrow \mathbb{R}$, for instance the zero--one loss $l((x, y), h) = \mathbb{I}[y \neq h(x)]$.
The input--output pairs are assumed to be generated a probability distribution $p^*$. 

The \emph{expected risk} of a model $h$,
\begin{equation} \label{eq:expected-risk}
L(h) \doteq \mathbb{E}_{(x,y)\sim p^*} [l((x,y), h)],
\end{equation}
is measured as its expected loss $l$ assuming that the pairs $(x_1,y_1), (x_2,y_2), \ldots$ are sampled i.i.d. from the probability distribution $p^*$. The  \emph{expected risk minimiser},
\begin{equation} \label{eq:expected-risk-minimiser}
h^* \doteq \arg \min_{h \in \mathcal{H}} L(h),
\end{equation}
is any hypothesis in the given model space that minimises the expected risk (not a random quantity at all). 
Given $n$ training examples, also drawn i.i.d. from $p^*$, the \emph{empirical risk} of a hypothesis $h$ is the average loss over the available training set $\mathcal{D} = \{(x_1,y_1), \ldots, (x_n,y_n) \}$:
\begin{equation}
\hat{L}(h) \doteq \frac{1}{n} \sum_{i=1}^n l((x_i,y_i),h).
\end{equation}
We can then define the \emph{empirical risk minimiser} (ERM), namely
\begin{equation} \label{eq:empirical-risk-minimiser}
\hat{h} \doteq \arg \min_{h \in \mathcal{H}} \hat{L}(h),
\end{equation}
which is instead a random variable which depends on the collection of training data.

Statistical learning theory is interested in the expected risk of the ERM (which is the only thing we can compute from the available training set). 
The core notion is that of \emph{probably approximately correct} algorithms. 
\begin{definition}
A learning algorithm is \emph{probably approximately correct} (PAC) if it finds with probability at least $1-\delta$ a model $h \in \mathcal{H}$ which is `approximately correct', i.e., it makes a training error of no more than $\epsilon$.
\end{definition}

According to this definition, PAC learning aims at providing bounds of the kind
\[
P[L(\hat{h}) - L(h^*) > \epsilon] \leq \delta,
\]
on the difference between the loss of the ERM and the minimum theoretical loss for that class of models.  
This is to account for the generalisation problem, due to the fact that the error we commit when training a model is different from the error one can expect in general cases in which data we have not yet observed are presented to the system. 

As we will see in more detail later on, statistical learning theory makes predictions about the reliability of a training set, based on simple quantities such as the number of samples $n$. 
In particular, 
the main result of PAC learning, in the case of finite model spaces $\mathcal{H}$, is that we can relate the required size $n$ of a training sample to the size of the model space $\mathcal{H}$, namely
\[
\log{|\mathcal{H}|} \leq n \epsilon - \log{\frac{1}{\delta}},
\]
so that the minimum number of training examples given $\epsilon$, $\delta$ and $|\mathcal{H}|$ is
\[
n \geq \frac{1}{\epsilon} \left ( \log{|\mathcal{H}|} + \log{\frac{1}{\delta}} \right ).
\]

\subsubsection{Vapnik--Chervonenkis dimension}

Obviously, for infinite-dimensional hypothesis spaces $\mathcal{H}$, the previous relation does not make sense. However, a similar constraint on the number of samples can be obtained after we introduce the following notion.

\begin{definition}
The \emph{Vapnik--Chervonenkis (VC) dimension}\footnote{{\url{http://www.svms.org/vc-dimension/}.}} of a model space $\mathcal{H}$ is the maximum number of points that can be successfully `shattered' by a model $h \in \mathcal{H}$ (i.e., they can be correctly classified by some $h \in \mathcal{H}$ for all possible binary labellings of these points).
\end{definition}

Unfortunately, the VC dimension dramatically overestimates the number of training instances required. Nevertheless, arguments based on the VC dimension provide the only justification for max-margin linear SVMs, the dominant force in classification for the last two decades, before the advent of deep learning.
As a matter of fact,
for the space $\mathcal{H}_m$ of linear classifiers with margin $m$, the following expression holds:
\[
VC_\text{SVM} = \min \left \{ d, \frac{4R^2}{m^2} \right \} + 1,
\]
where $d$ is the dimensionality of the data and $R$ is the radius of the smallest hypersphere enclosing all the training data.
As the VC dimension of $\mathcal{H}_m$ decreases as $m$ grows, it is then desirable to select the linear boundaries which have maximum margin, as in support vector machines. 

\subsubsection{Towards imprecise-probabilistic foundations for machine learning} \label{sec:future-foundations}

These issues with Vapnik's traditional statistical learning theory have recently been recognised by many scientists.

There is exciting research exploring a variety of options, for instance 
risk-sensitive reinforcement learning \cite{DBLP:journals/corr/ShenTSO13}; approaches based on situational awareness, i.e., the ability to perceive the current status of an environment in terms of time and space, comprehend its meaning and behave with a corresponding risk attitude; robust optimisation using minimax learning \cite{Zhou:2012:LWC:2999325.2999380}; and budgeted adversary models \cite{DBLP:journals/corr/KhaniRL16}. The latter, rather than learning models that predict accurately on a target distribution, use minimax optimisation to learn models that are suitable for any target distribution within a `safe' family. A portfolio of models have been proposed \cite{Caruana:2004:ESL:1015330.1015432}, including Satzilla \cite{Xu:2008:SPA:1622673.1622687} and IBM Watson (\url{http://www.ibm.com/watson/}), which combines more than 100 different techniques for analysing natural language, identifying sources, finding and generating hypotheses, finding and scoring evidence, and merging and ranking hypotheses. 
By an ensemble of models, however, these methods mean sets of models of different nature (e.g. SVMs, random forests) or models of the same kind learned from different slices of data (e.g. boosting). Exciting recent progress in domain adaptation includes multiview learning \cite{AISTATS2011_BlitzerKF11}, multifaceted understanding \cite{Blum:1998:CLU:279943.279962} and learning of transferable components \cite{GonZhaLiuTaoSch16,DBLP:journals/corr/Long0J16a}. 

Some consensus exists that robust approaches should provide worst-case guarantees, as it is not possible to rule out completely unexpected behaviours or catastrophic failures.
The minimax approach, in particular, evokes the concept of imprecise probability, at least in its credal, robust-statistical incarnation. 
As we know, imprecise probabilities naturally arise whenever the data are insufficient to allow the estimation of a probability distribution. 
Indeed, training sets in virtually all applications of machine learning constitute a glaring example of data which is both:
\begin{itemize}
\item
\emph{insufficient in quantity} (think of a Google routine for object detection from images, trained on a few million images, compared with the thousands of billions of images out there), and
\item
\emph{insufficient in quality} (as they are selected based on criteria such as cost, availability or mental attitudes, therefore biasing the whole learning process).
\end{itemize}

Uncertainty theory therefore has the potential to address model adaptation in new, original ways, so providing a new paradigm for robust statistical learning.
A sensible programme for future research could then be based on the following developments (see Fig. \ref{fig:leverhulme}):
\begin{enumerate}
\item 
Addressing the domain adaptation problem by allowing the test data to be sampled from a different probability distribution than the training data, under the weaker assumption that both belong to the same convex set of distributions (credal set) (Fig. \ref{fig:leverhulme}(a)).
\item 
Extending further the statistical-learning-theory framework by moving away from the selection of a single model from a class of hypotheses to the identification of a convex set of models induced by the available training data (Fig. \ref{fig:leverhulme}(b)).
\item 
Applying the resulting convex-theoretical learning theory (with respect to both the data-generating probabilities and the model space) to the functional spaces associated with convolutional neural networks (e.g., series of convolutions and max-pooling operations, Fig. \ref{fig:leverhulme}(c)) in order to lay solid theoretical foundations for deep learning.
\end{enumerate}

\begin{figure}[ht!]
\begin{center}
\begin{tabular}{ccc}
\includegraphics[height=3.6cm]{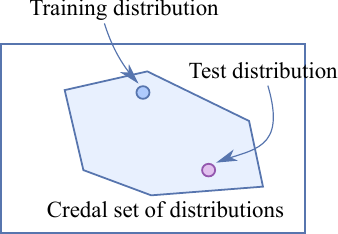}
&
\includegraphics[height=4.5cm]{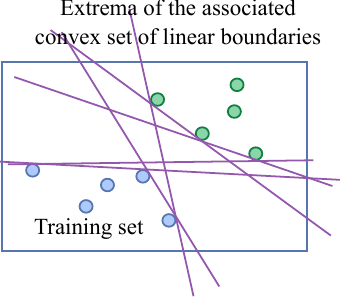}
\\
(a) & (b) 
\end{tabular}
\\
\begin{tabular}{c}
\includegraphics[height=4.5cm]{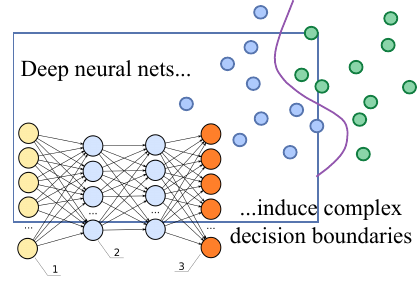}
\\
(c)
\end{tabular}
\end{center}
\caption{(a) Training and test distributions from a common random set. (b) A training set of separable points determines only a convex set of admissible linear boundaries. (c) Going against conventional wisdom, deep neural models have very good generalisation power while generating, at the same time, complex decision boundaries.} \label{fig:leverhulme}
\end{figure}

\subsubsection{Bounds for realisable finite hypothesis classes} 

To conclude this essay, we consider the domain adaptation issue in the case of realisable finite hypothesis classes, and outline a credal solution for it.

Under the assumption that (1) the model space $\mathcal{H}$ is finite, and (2) there exists a hypothesis $h^* \in \mathcal{H}$ that obtains zero expected risk, that is,
\begin{equation} \label{eq:realisability}
L(h^*) = \mathbb{E}_{(x,y)\sim p^*} [l((x,y), h^*)] = 0,
\end{equation}
a property called \emph{realisability}, the following result holds\footnote{{\url{https://web.stanford.edu/class/cs229t/notes.pdf}.}}. 
\begin{theorem} \label{the:the-4}
Let $\mathcal{H}$ be a hypothesis class, where each hypothesis $h \in \mathcal{H}$ maps some $\mathcal{X}$ to $\mathcal{Y}$; let $l$ be the zero--one loss, $l((x, y), h) = \mathbb{I}[y \neq h(x)]$; let $p^*$ be any distribution over $\mathcal{X} \times \mathcal{Y}$; and let $\hat{h}$ be the empirical risk minimiser (\ref{eq:empirical-risk-minimiser}). Assume that (1) and (2) hold.
Then the following two equivalent statements hold, with probability at least $1-\delta$:
\begin{equation} \label{eq:th4-1}
L(\hat{h}) \leq \frac{\log |\mathcal{H}| + \log (1/\delta)}{n},
\end{equation}
\begin{equation} \label{eq:th4-2}
n \geq \frac{\log |\mathcal{H}| + \log (1/\delta)}{\epsilon} \Rightarrow L(\hat{h}) \leq \epsilon.
\end{equation}
\end{theorem}
\begin{proof}
Let $B = \{h \in \mathcal{H} : L(h) > \epsilon \}$ be the set of bad hypotheses. We wish to upper-bound the probability\footnote{Note that this is a probability measure on the space of models $\mathcal{H}$, completely distinct from the data-generating probability $p^*$.} of selecting a bad hypothesis:
\begin{equation} \label{eq:proof1}
P[L(\hat{h}) > \epsilon] = P[\hat{h} \in B].
\end{equation}
Recall that the empirical risk of the ERM is always zero, i.e., $\hat{L}(\hat{h}) =0$, since at least $\hat{L}(h^*) = L(h^*) = 0$. So if we have selected a bad hypothesis ($\hat{h} \in B$), then some
bad hypothesis must have zero empirical risk: $\hat{h} \in B$ implies $\exists h \in B : \hat{L}(h) = 0$, which is equivalent to saying that
\begin{equation} \label{eq:proof2}
P[\hat{h} \in B] \leq P[\exists h \in B : \hat{L}(h) = 0].
\end{equation}

Firstly, we need to bound $P[\hat{L}(h) = 0]$ for any fixed $h \in B$. For each example, hypothesis $h$ does not err, with probability $1 - L(h)$. Since the training examples are i.i.d. and since $L(h) > \epsilon$ for $h \in B$,
\[
P[\hat{L}(h) = 0] = (1 - L(h))^n \leq (1- \epsilon)^n \leq e^{-\epsilon n} \quad \forall h \in B,
\]
where the first equality comes from the examples being i.i.d., and the last step follows since $1- a \leq e^{-a}$. Note that the probability $P[\hat{L}(h) = 0]$ decreases exponentially with $n$.

Secondly, we need to show that the above bound holds simultaneously for all $ h \in B$. Recall that $P(A_1 \cup \cdots \cup A_K) \leq P(\cup_{k=1}^k A_i)$. When applied to the (non-disjoint) events $A_h = \{\hat{L}(h) =0\}$, the union bound yields
\[
P[\exists h \in B : \hat{L}(h) = 0] \leq \sum_{h \in B} P[\hat{L}(h)=0].
\]

Finally,
\[
\begin{array}{lll}
P[L(\hat{h}) \leq \epsilon] = P[\hat{h} \in B] & \leq & P[\exists h \in B : \hat{L}(h) = 0] \leq \displaystyle \sum_{h \in B} P[\hat{L}(h)=0]
\\
& \leq & |B| e^{-en} \leq |\mathcal{H}| e^{-en} \doteq \delta,
\end{array}
\]
and by rearranging we have $\epsilon  = ( {\log |\mathcal{H}| + \log (1/\delta)} ) / {n}$, i.e., (\ref{eq:th4-1}).
\end{proof}

The inequality (\ref{eq:th4-1}) should be interpreted in the following way: with probability at least $1-\delta$, the expected loss of the empirical risk minimiser $\hat{h}$ (the particular model selected after training on $n$ examples) is bounded by the ratio on the right-hand side, so that $L(\hat{h})  = O({\log |\mathcal{H}|} / {n})$.
The inequality (\ref{eq:th4-1}) amounts to saying that if I need to obtain an expected risk of at most $\epsilon$ with confidence at least $1- \delta$, I need to train the model on at least $ ( {\log |\mathcal{H}| + \log (1/\delta)} ) / {\epsilon}$ examples.

The result is \emph{distribution-free}, as it is independent of the choice of $p^*(x,y)$. However, it does rely on the assumption that the training and test distributions are the same.

\subsubsection{Bounds under credal generalisation} 

A credal generalisation of Theorem \ref{the:the-4} would thus read as follows.
\begin{theorem} \label{the:the-4-credal}
Let $\mathcal{H}$ be a finite hypothesis class, let
$\mathcal{P}$ be a credal set over $\mathcal{X} \times \mathcal{Y}$ and let a training set of samples $\mathcal{D} = \{(x_i,y_i), i=1, \ldots, n\}$ be drawn from one of the distributions $\hat{p} \in \mathcal{P}$ in this credal set. Let
\[
\hat{h} \doteq \arg \min_{h \in \mathcal{H}} \hat{L}(h) =  \arg \min_{h \in \mathcal{H}} \frac{1}{n} \sum_{i=1}^n l((x_i,y_i),h)
\]
be the empirical risk minimiser, where $l$ is the 0--1 loss. 
Assume that
\begin{equation}  \label{eq:credal-realisability-1}
\exists h^* \in \mathcal{H}, p^* \in \mathcal{P} : \mathbb{E}_{p^*}[l] = L_{p^*}(h^*) = 0
\end{equation}
(\emph{credal realisability}) holds.

Then, with probability at least $1-\delta$, the following bound holds:
\begin{equation} \label{eq:th4-credal}
P \left [ \max_{p \in \mathcal{P}} L_p(\hat{h})> \epsilon \right ] \leq \epsilon(\mathcal{H},\mathcal{P},\delta),
\end{equation}
where $\epsilon$ is a function of the size of the model space $\mathcal{H}$, the size of the credal set $\mathcal{P}$ and $\delta$.
\end{theorem}

How does the proof of Theorem \ref{the:the-4} generalise to the credal case? We can note that, if we define
\[
B \doteq \left \{h : \max_{p \in \mathcal{P}} L_p(h) > \epsilon \right \}
\]
as the set of models which do not guarantee the bound \emph{in the worst case}, we get, just as in the classical result (\ref{eq:proof1}),
\[
P\left [ \max_{p \in \mathcal{P}} L_p(\hat{h}) > \epsilon \right ] = P[\hat{h} \in B].
\]
However, whereas in the classical case $\hat{L}(\hat{h}) =0$, in the credal case we can only say that
\[
\hat{L}(h^*) = L_{\hat{p}} (h^*) \geq \min_{p \in \mathcal{P}} L_p (h^*) = 0,
\]
by the credal realisability assumption (\ref{eq:credal-realisability-1}). Thus, (\ref{eq:proof2}) does not hold, and the above argument which tries to bound $P[\hat{L}(h) = 0]$ does not apply here.

One option is to assume that $\hat{p} = p^* = \arg \min_p L_p(h^*)$. This would allow (\ref{eq:proof2})  to remain valid.

A different generalisation of realisability can also be proposed, in place of (\ref{eq:credal-realisability-1}):
\begin{equation}  \label{eq:credal-realisability-2}
\forall p \in \mathcal{P}, \exists h_p^* \in \mathcal{H}  : \mathbb{E}_{p}[l(h_p^*)] = L_{p}(h_p^*) = 0,
\end{equation}
which we can call \emph{uniform credal realisability}. Under the latter assumption, we would get
\[
\hat{L}(\hat{h}) = \min_{h \in \mathcal{H}} \hat{L}(h) = \min_{h \in \mathcal{H}} L_{\hat{p}} (h) =
L_{\hat{p}}(h^*_{\hat{p}}) = 0.
\]

We will complete this analysis and extend it to the case of infinite, non-realisable model spaces in the near future.

\section{Conclusions}

In this essay, we briefly touched upon a number of important open questions that concern the application of the theory of random set to statistics, machine learning and other areas of applied science. Much remains to be done, but the direction is clear and the potential for scaling up the impact achieved so far significant.

\bibliographystyle{plain}
\bibliography{agenda-for-the-future}

\begin{thebibliography}{100}

\bibitem{abdallah2013using}
Nadia~Ben Abdallah, Nassima Mouhous-Voyneau, and Thierry Denoeux.
\newblock Using dempster-shafer theory to model uncertainty in climate change and environmental impact assessments.
\newblock In {\em Proceedings of the 16th International Conference on Information Fusion}, pages 2117--2124. IEEE, 2013.

\bibitem{Agahi20131213}
Hamzeh Agahi, Adel Mohammadpour, Radko Mesiar, and Yao Ouyang.
\newblock On a strong law of large numbers for monotone measures.
\newblock {\em Statistics and Probability Letters}, 83(4):1213--1218, 2013.

\bibitem{Aickin2000}
Mikel Aickin.
\newblock Connecting {D}empster--{S}hafer belief functions with likelihood-based inference.
\newblock {\em Synthese}, 123(3):347--364, 2000.

\bibitem{antonucci10-credal}
Alessandro Antonucci and Fabio Cuzzolin.
\newblock Credal sets approximation by lower probabilities: {A}pplication to credal networks.
\newblock In Eyke Hüllermeier, Rudolf Kruse, and Frank Hoffmann, editors, {\em Computational Intelligence for Knowledge-Based Systems Design}, volume 6178 of {\em Lecture Notes in Computer Science}, pages 716--725. Springer, Berlin Heidelberg, 2010.

\bibitem{aregui08constructing}
Astride Aregui and Thierry Den{\oe}ux.
\newblock Constructing consonant belief functions from sample data using confidence sets of pignistic probabilities.
\newblock {\em International Journal of Approximate Reasoning}, 49(3):575--594, 2008.

\bibitem{Augustin96}
Thomas Augustin.
\newblock Modeling weak information with generalized basic probability assignments.
\newblock In H.~H. Bock and W.~Polasek, editors, {\em Data Analysis and Information Systems - Statistical and Conceptual Approaches}, pages 101--113. Springer, 1996.

\bibitem{Ben-Israel1986invexity}
Adi Ben-Israel and Bertram Mond.
\newblock What is invexity?
\newblock {\em Journal of the Australian Mathematical Society. Series B. Applied Mathematics}, 28:1--9, 1986.

\bibitem{Benavoli2017QPL}
Alessio Benavoli, Alessandro Facchini, and Marco Zaffalon.
\newblock {Bayes + Hilbert = Quantum Mechanics}.
\newblock In {\em Proceedings of the 14th International Conference on Quantum Physics and Logic (QPL 2017)}, 2017.

\bibitem{berger1997unified}
James~O Berger, Ben Boukai, and Yinping Wang.
\newblock Unified frequentist and bayesian testing of a precise hypothesis.
\newblock {\em Statistical Science}, 12(3):133--160, 1997.

\bibitem{black97geometric}
Paul~K. Black.
\newblock Geometric structure of lower probabilities.
\newblock In Goutsias, Malher, and Nguyen, editors, {\em Random Sets: Theory and Applications}, pages 361--383. Springer, 1997.

\bibitem{AISTATS2011_BlitzerKF11}
John Blitzer, Sham Kakade, and Dean~P. Foster.
\newblock Domain adaptation with coupled subspaces.
\newblock In Geoffrey~J. Gordon and David~B. Dunson, editors, {\em Proceedings of the 14th International Conference on Artificial Intelligence and Statistics (AISTATS-11)}, volume~15, pages 173--181, 2011.

\bibitem{Blum:1998:CLU:279943.279962}
Avrim Blum and Tom Mitchell.
\newblock Combining labeled and unlabeled data with co-training.
\newblock In {\em Proceedings of the 11th Annual Conference on Computational Learning Theory (COLT'98)}, pages 92--100, New York, NY, USA, 1998. ACM.

\bibitem{bradley2019imprecise}
Seamus Bradley.
\newblock Imprecise probabilities.
\newblock {\em Computer Simulation Validation: Fundamental Concepts, Methodological Frameworks, and Philosophical Perspectives}, pages 525--540, 2019.

\bibitem{brazitikos2014geometry}
Silouanos Brazitikos, Apostolos Giannopoulos, Petros Valettas, and Beatrice-Helen Vritsiou.
\newblock {\em Geometry of isotropic convex bodies}, volume 196.
\newblock American Mathematical Soc., 2014.

\bibitem{burger10brest}
Thomas Burger and Fabio Cuzzolin.
\newblock The barycenters of the k-additive dominating belief functions and the pignistic k-additive belief functions.
\newblock In {\em Proceedings of the First International Workshop on the Theory of Belief Functions (BELIEF 2010)}, 2010.

\bibitem{Caruana:2004:ESL:1015330.1015432}
Rich Caruana, Alexandru Niculescu-Mizil, Geoff Crew, and Alex Ksikes.
\newblock Ensemble selection from libraries of models.
\newblock In {\em Proceedings of the 21st International Conference on Machine Learning (ICML'04)}, pages 18--, New York, NY, USA, 2004. ACM.

\bibitem{PhysRevA.65.022305}
Carlton~M. Caves, Christopher~A. Fuchs, and R\"udiger Schack.
\newblock {Quantum probabilities as Bayesian probabilities}.
\newblock {\em Physical review A}, 65(2):022305, January 2002.

\bibitem{celler1995generating}
Frank Celler, Charles~R Leedham-Green, Scott~H Murray, Alice~C Niemeyer, and Eamonn~A O'brien.
\newblock Generating random elements of a finite group.
\newblock {\em Communications in algebra}, 23(13):4931--4948, 1995.

\bibitem{CIS-234958}
Patrick Chareka.
\newblock The central limit theorem for capacities.
\newblock {\em Statistics and Probability Letters}, 79(12):1456--1462, 2009.

\bibitem{Chateauneuf89}
A.~Chateauneuf and Jean-Yves Jaffray.
\newblock Some characterization of lower probabilities and other monotone capacities through the use of {M}{\"o}ebius inversion.
\newblock {\em Mathematical social sciences}, (3):263--283, 1989.

\bibitem{Chen2010strong}
Zengjing Chen.
\newblock Strong law of large numbers for sub-linear expectations.
\newblock ArXiv preprint arXiv:1006.0749, June 2010.

\bibitem{chiu2013stochastic}
Sung~Nok Chiu, Dietrich Stoyan, Wilfrid~S. Kendall, and Joseph Mecke.
\newblock {\em Stochastic geometry and its applications}.
\newblock John Wiley \& Sons, 2013.

\bibitem{Choquet53}
Gustave Choquet.
\newblock Theory of capacities.
\newblock {\em Annales de l'Institut Fourier}, 5:131--295, 1953.

\bibitem{Cobb03isf}
Barry~R. Cobb and Prakash~P. Shenoy.
\newblock A comparison of {B}ayesian and belief function reasoning.
\newblock {\em Information Systems Frontiers}, 5(4):345--358, 2003.

\bibitem{Cozman20101069}
Fabio~Gagliardi Cozman.
\newblock Concentration inequalities and laws of large numbers under epistemic and regular irrelevance.
\newblock {\em International Journal of Approximate Reasoning}, 51(9):1069--1084, 2010.

\bibitem{cuzzolin01bcc}
Fabio Cuzzolin.
\newblock Lattice modularity and linear independence.
\newblock In {\em Proceedings of the 18th British Combinatorial Conference (BCC'01)}, 2001.

\bibitem{cuzzolin01thesis}
Fabio Cuzzolin.
\newblock {\em Visions of a generalized probability theory}.
\newblock {PhD} dissertation, Universit\`a degli Studi di Padova, 19 February 2001.

\bibitem{cuzzolin04smcb}
Fabio Cuzzolin.
\newblock Geometry of {D}empster's rule of combination.
\newblock {\em IEEE Transactions on Systems, Man and Cybernetics part B}, 34(2):961--977, 2004.

\bibitem{cuzzolin04ipmu}
Fabio Cuzzolin.
\newblock Simplicial complexes of finite fuzzy sets.
\newblock In {\em Proceedings of the 10th International Conference on Information Processing and Management of Uncertainty (IPMU'04)}, volume~4, pages 4--9, 2004.

\bibitem{cuzzolin05amai}
Fabio Cuzzolin.
\newblock Algebraic structure of the families of compatible frames of discernment.
\newblock {\em Annals of Mathematics and Artificial Intelligence}, 45(1-2):241--274, 2005.

\bibitem{cuzzolin07ecsqaru}
Fabio Cuzzolin.
\newblock On the orthogonal projection of a belief function.
\newblock In {\em Proceedings of the International Conference on Symbolic and Quantitative Approaches to Reasoning with Uncertainty (ECSQARU'07)}, volume 4724 of {\em Lecture Notes in Computer Science}, pages 356--367. Springer, Berlin / Heidelberg, 2007.

\bibitem{cuzzolin07bcc}
Fabio Cuzzolin.
\newblock {On the relationship between the notions of independence in matroids, lattices, and Boolean algebras}.
\newblock In {\em Proceedings of the British Combinatorial Conference (BCC'07)}, 2007.

\bibitem{cuzzolin2008geometric}
Fabio Cuzzolin.
\newblock {A geometric approach to the theory of evidence}.
\newblock {\em IEEE Transactions on Systems, Man, and Cybernetics, Part C: Applications and Reviews}, 38(4):522--534, 2008.

\bibitem{cuzzolin08pricai-moebius}
Fabio Cuzzolin.
\newblock Alternative formulations of the theory of evidence based on basic plausibility and commonality assignments.
\newblock In {\em Proceedings of the Pacific Rim International Conference on Artificial Intelligence (PRICAI'08)}, pages 91--102, 2008.

\bibitem{cuzzolin08isaim-matroid}
Fabio Cuzzolin.
\newblock Boolean and matroidal independence in uncertainty theory.
\newblock In {\em Proceedings of the International Symposium on Artificial Intelligence and Mathematics (ISAIM 2008)}, 2008.

\bibitem{cuzzolin08isaim-simplicial}
Fabio Cuzzolin.
\newblock An interpretation of consistent belief functions in terms of simplicial complexes.
\newblock In {\em Proceedings of the International Symposium on Artificial Intelligence and Mathematics (ISAIM 2008)}, 2008.

\bibitem{cuzzolin2008lp}
Fabio Cuzzolin.
\newblock Lp consistent approximations of belief functions.
\newblock {\em IEEE Transactions on Fuzzy Systems (under review)}, 2008.

\bibitem{cuzzolin08-credal}
Fabio Cuzzolin.
\newblock On the credal structure of consistent probabilities.
\newblock In Steffen Hölldobler, Carsten Lutz, and Heinrich Wansing, editors, {\em Logics in Artificial Intelligence}, volume 5293 of {\em Lecture Notes in Computer Science}, pages 126--139. Springer, Berlin Heidelberg, 2008.

\bibitem{cuzzolin09ecsqaru}
Fabio Cuzzolin.
\newblock Complexes of outer consonant approximations.
\newblock In {\em Proceedings of the 10th European Conference on Symbolic and Quantitative Approaches to Reasoning with Uncertainty (ECSQARU'09)}, pages 275--286, 2009.

\bibitem{cuzzolin09-intersection}
Fabio Cuzzolin.
\newblock The intersection probability and its properties.
\newblock In Claudio Sossai and Gaetano Chemello, editors, {\em Symbolic and Quantitative Approaches to Reasoning with Uncertainty}, volume 5590 of {\em Lecture Notes in Computer Science}, pages 287--298. Springer, Berlin Heidelberg, 2009.

\bibitem{cuzzolin2010credal}
Fabio Cuzzolin.
\newblock {Credal semantics of Bayesian transformations in terms of probability intervals}.
\newblock {\em IEEE Transactions on Systems, Man, and Cybernetics, Part B: Cybernetics}, 40(2):421--432, 2010.

\bibitem{cuzzolin10fss}
Fabio Cuzzolin.
\newblock The geometry of consonant belief functions: simplicial complexes of necessity measures.
\newblock {\em Fuzzy Sets and Systems}, 161(10):1459--1479, 2010.

\bibitem{cuzzolin10ida}
Fabio Cuzzolin.
\newblock Three alternative combinatorial formulations of the theory of evidence.
\newblock {\em Intelligent Data Analysis}, 14(4):439--464, 2010.

\bibitem{cuzzolin11isipta-conditional}
Fabio Cuzzolin.
\newblock Geometric conditional belief functions in the belief space.
\newblock In {\em Proceedings of the 7th International Symposium on Imprecise Probabilities and Their Applications (ISIPTA'11)}, 2011.

\bibitem{cuzzolin11-consistent}
Fabio Cuzzolin.
\newblock On consistent approximations of belief functions in the mass space.
\newblock In Weiru Liu, editor, {\em Symbolic and Quantitative Approaches to Reasoning with Uncertainty}, volume 6717 of {\em Lecture Notes in Computer Science}, pages 287--298. Springer, Berlin Heidelberg, 2011.

\bibitem{cuzzolin14algebraic}
Fabio Cuzzolin.
\newblock {Chapter 12: An algebraic study of the notion of independence of frames}.
\newblock In S.~Chakraverty, editor, {\em Mathematics of Uncertainty Modeling in the Analysis of Engineering and Science Problems}. IGI Publishing, 2014.

\bibitem{Cuzzolin2014tfs}
Fabio Cuzzolin.
\newblock Lp consonant approximations of belief functions.
\newblock {\em IEEE Transactions on Fuzzy Systems}, 22(2):420--436, 2014.

\bibitem{cuzzolin14annals}
Fabio Cuzzolin.
\newblock On the fiber bundle structure of the space of belief functions.
\newblock {\em Annals of Combinatorics}, 18(2):245--263, 2014.

\bibitem{cuzzolin18belief-maxent}
Fabio Cuzzolin.
\newblock Generalised max entropy classifiers.
\newblock In S{\'e}bastien Destercke, Thierry Den{\oe}ux, Fabio Cuzzolin, and Arnaud Martin, editors, {\em Belief Functions: Theory and Applications}, pages 39--47, Cham, 2018. Springer International Publishing.

\bibitem{cuzzolin2021springer}
Fabio Cuzzolin.
\newblock {\em The geometry of uncertainty - The geometry of imprecise probabilities}.
\newblock Springer Nature, 2021.

\bibitem{cuzzolin2021uncertainty}
Fabio Cuzzolin.
\newblock Uncertainty measures: The big picture.
\newblock {\em arXiv preprint arXiv:2104.06839}, 2021.

\bibitem{cuzzolin10brest}
Fabio Cuzzolin.
\newblock Geometric conditioning of belief functions.
\newblock In {\em Proceedings of the Workshop on the Theory of Belief Functions (BELIEF'10)}, April 2010.

\bibitem{cuzzolin03isipta}
Fabio Cuzzolin.
\newblock Geometry of upper probabilities.
\newblock In {\em Proceedings of the 3rd Internation Symposium on Imprecise Probabilities and Their Applications (ISIPTA'03)}, July 2003.

\bibitem{cuzzolin11isipta-consonant}
Fabio Cuzzolin.
\newblock Lp consonant approximations of belief functions in the mass space.
\newblock In {\em Proceedings of the 7th International Symposium on Imprecise Probability: Theory and Applications (ISIPTA'11)}, July 2011.

\bibitem{cuzzolin09isipta-consistent}
Fabio Cuzzolin.
\newblock Consistent approximation of belief functions.
\newblock In {\em Proceedings of the 6th International Symposium on Imprecise Probability: Theory and Applications (ISIPTA'09)}, June 2009.

\bibitem{cuzzolin02fsdk}
Fabio Cuzzolin.
\newblock Geometry of {D}empster's rule.
\newblock In {\em Proceedings of the 1st International Conference on Fuzzy Systems and Knowledge Discovery (FSKD'02)}, November 2002.

\bibitem{cuzzolin05hawaii}
Fabio Cuzzolin.
\newblock On the properties of relative plausibilities.
\newblock In {\em Proceedings of the International Conference of the IEEE Systems, Man, and Cybernetics Society (SMC'05)}, volume~1, pages 594--599, October 2005.

\bibitem{cuzzolin00rss}
Fabio Cuzzolin.
\newblock Families of compatible frames of discernment as semimodular lattices.
\newblock In {\em Proceedings of the International Conference of the Royal Statistical Society (RSS 2000)}, September 2000.

\bibitem{cuzzolin14lap}
Fabio Cuzzolin.
\newblock {\em Visions of a generalized probability theory}.
\newblock Lambert Academic Publishing, September 2014.

\bibitem{Cuzzolin99}
Fabio Cuzzolin and Ruggero Frezza.
\newblock An evidential reasoning framework for object tracking.
\newblock In Matthew~R. Stein, editor, {\em Proceedings of SPIE - Photonics East 99 - Telemanipulator and Telepresence Technologies VI}, volume 3840, pages 13--24, 19-22 September 1999.

\bibitem{cuzzolin05isipta}
Fabio Cuzzolin and Ruggero Frezza.
\newblock Evidential modeling for pose estimation.
\newblock In {\em Proceedings of the 4th Internation Symposium on Imprecise Probabilities and Their Applications (ISIPTA'05)}, July 2005.

\bibitem{cuzzolin00mtns}
Fabio Cuzzolin and Ruggero Frezza.
\newblock Integrating feature spaces for object tracking.
\newblock In {\em Proceedings of the International Symposium on the Mathematical Theory of Networks and Systems (MTNS 2000)}, June 2000.

\bibitem{cuzzolin01space}
Fabio Cuzzolin and Ruggero Frezza.
\newblock Geometric analysis of belief space and conditional subspaces.
\newblock In {\em Proceedings of the 2nd International Symposium on Imprecise Probabilities and their Applications (ISIPTA'01)}, June 2001.

\bibitem{cuzzolin01lattice}
Fabio Cuzzolin and Ruggero Frezza.
\newblock Lattice structure of the families of compatible frames.
\newblock In {\em Proceedings of the 2nd International Symposium on Imprecise Probabilities and their Applications (ISIPTA'01)}, June 2001.

\bibitem{cuzzolin13fusion}
Fabio Cuzzolin and Wenjuan Gong.
\newblock Belief modeling regression for pose estimation.
\newblock In {\em Proceedings of the 16th International Conference on Information Fusion (FUSION 2013)}, pages 1398--1405, 2013.

\bibitem{DeCooman20082409}
Gert de~Cooman and Enrique Miranda.
\newblock Weak and strong laws of large numbers for coherent lower previsions.
\newblock {\em Journal of Statistical Planning and Inference}, 138(8):2409--2432, 2008.

\bibitem{dempster67multivalued}
Arthur~P. Dempster.
\newblock Upper and lower probabilities induced by a multivalued mapping.
\newblock {\em Annals of Mathematical Statistics}, 38(2):325--339, 1967.

\bibitem{Dempster01informs}
Arthur~P. Dempster.
\newblock Belief functions in the 21st century: {A} statistical perspective.
\newblock In {\em Proceedings of the Institute for Operations Research and Management Science Annual meeting (INFORMS)}, 2001.

\bibitem{denneberg1994conditioning}
Dieter Denneberg.
\newblock Conditioning (updating) non-additive measures.
\newblock {\em Annals of Operations research}, 52:21--42, 1994.

\bibitem{denneberg99interaction}
Dieter Denneberg and Michel Grabisch.
\newblock Interaction transform of set functions over a finite set.
\newblock {\em Information Sciences}, 121(1-2):149--170, 1999.

\bibitem{denoeux99reasoning}
Thierry Den{\oe}ux.
\newblock Reasoning with imprecise belief structures.
\newblock {\em International Journal of Approximate Reasoning}, 20(1):79--111, 1999.

\bibitem{denoeux07ai}
Thierry Den{\oe}ux.
\newblock Conjunctive and disjunctive combination of belief functions induced by non distinct bodies of evidence.
\newblock {\em Artificial Intelligence}, (2-3):234--264, 2008.

\bibitem{Denoeux2008classic}
Thierry Den{\oe}ux.
\newblock A k-nearest neighbor classification rule based on {D}empster--{S}hafer theory.
\newblock In Roland~R. Yager and Liping Liu, editors, {\em Classic Works of the Dempster-Shafer Theory of Belief Functions}, volume 219 of {\em Studies in Fuzziness and Soft Computing}, pages 737--760. Springer, 2008.

\bibitem{denoeux2021distributed}
Thierry Denoeux.
\newblock Distributed combination of belief functions.
\newblock {\em Information Fusion}, 65:179--191, 2021.

\bibitem{DEMPSTER2008365}
Thierry Den{\oe}ux and Arthur~P. Dempster.
\newblock The {D}empster--{S}hafer calculus for statisticians.
\newblock {\em International Journal of Approximate Reasoning}, 48(2):365--377, 2008.

\bibitem{dezert2018total}
Jean Dezert, Albena Tchamova, and Deqiang Han.
\newblock Total belief theorem and generalized bayes' theorem.
\newblock In {\em 2018 21st International Conference on Information Fusion (FUSION)}, pages 1040--1047. IEEE, 2018.

\bibitem{dubois87principle}
Didier Dubois and Henri Prade.
\newblock The principle of minimum specificity as a basis for evidential reasoning.
\newblock In B.~Bouchon and R.~R. Yager, editors, {\em Uncertainty in Knowledge-Based Systems}, pages 75--84. Springer-Verlag, 1987.

\bibitem{dubois88representation}
Didier Dubois and Henri Prade.
\newblock Representation and combination of uncertainty with belief functions and possibility measures.
\newblock {\em Computational Intelligence}, 4(3):244--264, 1988.

\bibitem{Dubois90}
Didier Dubois and Henri Prade.
\newblock Consonant approximations of belief functions.
\newblock {\em International Journal of Approximate Reasoning}, 4:419--449, 1990.

\bibitem{dubuc2010representation}
Eduardo~J Dubuc and Yuri~A Poveda.
\newblock Representation theory of mv-algebras.
\newblock {\em Annals of Pure and Applied Logic}, 161(8):1024--1046, 2010.

\bibitem{SHAFER2011127}
Jacques Dubucs and Glenn Shafer.
\newblock A betting interpretation for probabilities and {D}empster--{S}hafer degrees of belief.
\newblock {\em International Journal of Approximate Reasoning}, 52(2):127--136, 2011.

\bibitem{dutta2023distributionally}
Souradeep Dutta, Michele Caprio, Vivian Lin, Matthew Cleaveland, Kuk~Jin Jang, Ivan Ruchkin, Oleg Sokolsky, and Insup Lee.
\newblock Distributionally robust statistical verification with imprecise neural networks.
\newblock {\em arXiv preprint arXiv:2308.14815}, 2023.

\bibitem{Eichberger1999}
J{\"u}rgen Eichberger and David Kelsey.
\newblock {E-capacities and the Ellsberg paradox}.
\newblock {\em Theory and Decision}, 46(2):107--138, 1999.

\bibitem{elouedi00classification}
Zied Elouedi, Khaled Mellouli, and Philippe Smets.
\newblock Classification with belief decision trees.
\newblock In {\em Proceedings of the Ninth International Conference on Artificial Intelligence: Methodology, Systems, Architectures (AIMSA 2000)}, pages 80--90, 2000.

\bibitem{elouedi00decision}
Zied Elouedi, Khaled Mellouli, and Philippe Smets.
\newblock Decision trees using the belief function theory.
\newblock In {\em Proceedings of the Eighth International Conference on Information Processing and Management of Uncertainty in Knowledge-based Systems (IPMU 2000)}, volume~1, pages 141--148, Madrid, 2000.

\bibitem{Epstein2011CLT}
Larry~G. Epstein and Kyoungwon Seo.
\newblock A central limit theorem for belief functions, 2011.

\bibitem{fabian2013functional}
Mari{\'a}n Fabian, Petr Habala, Petr H{\'a}jek, Vicente~Montesinos Santaluc{\'\i}a, Jan Pelant, and V{\'a}clav Zizler.
\newblock {\em Functional analysis and infinite-dimensional geometry}.
\newblock Springer Science and Business Media, 2013.

\bibitem{fagin91new}
Ronald Fagin and Joseph~Y. Halpern.
\newblock A new approach to updating beliefs.
\newblock In {\em Proceedings of the Sixth Annual Conference on Uncertainty in Artificial Intelligence (UAI'90)}, pages 347--374, 1990.

\bibitem{Falk04}
Michael Falk, J{\"u}rg H{\"u}sler, and Rolf-Dieter Reiss.
\newblock {\em Laws of small numbers: extremes and rare events}.
\newblock Springer Science and Business Media, 2010.

\bibitem{farina2011positive}
Lorenzo Farina and Sergio Rinaldi.
\newblock {\em Positive linear systems: theory and applications}, volume~50.
\newblock John Wiley \& Sons, 2011.

\bibitem{frechet1948elements}
Maurice Fr{\'e}chet.
\newblock Les {\'e}l{\'e}ments al{\'e}atoires de nature quelconque dans un espace distanci{\'e}.
\newblock In {\em Annales de l'institut Henri Poincar{\'e}}, volume~10, pages 215--310, 1948.

\bibitem{Liu2012291}
Zhun ga~Liu, Jean Dezert, Gr\'egoire Mercier, and Quan Pan.
\newblock Belief {C}-means: {A}n extension of fuzzy {C}-means algorithm in belief functions framework.
\newblock {\em Pattern Recognition Letters}, 33(3):291--300, 2012.

\bibitem{geifman2017selective}
Yonatan Geifman and Ran El-Yaniv.
\newblock Selective classification for deep neural networks.
\newblock {\em arXiv:1705.08500}, 2017.

\bibitem{Geiger:2012:WRA:2354409.2354978}
Andreas Geiger, Philip Lenz, and Raquel Urtasun.
\newblock Are we ready for autonomous driving? {T}he {KITTI} vision benchmark suite.
\newblock In {\em Proceedings of the 2012 IEEE Conference on Computer Vision and Pattern Recognition (CVPR 2012)}, pages 3354--3361, 2012.

\bibitem{gennari02-integrating}
Giambattista Gennari, Alessandro Chiuso, Fabio Cuzzolin, and Ruggero Frezza.
\newblock Integrating shape and dynamic probabilistic models for data association and tracking.
\newblock In {\em Proceedings of the 41st IEEE Conference on Decision and Control (CDC'02)}, volume~3, pages 2409--2414, December 2002.

\bibitem{Giannopoulos2004}
Apostolos~A. Giannopoulos and V.~D. Milman.
\newblock Asymptotic convex geometry short overview.
\newblock In Simon Donaldson, Yakov Eliashberg, and Mikhael Gromov, editors, {\em Different Faces of Geometry}, pages 87--162. Springer US, Boston, MA, 2004.

\bibitem{gilboa1993updating}
Itzhak Gilboa and David Schmeidler.
\newblock Updating ambiguous beliefs.
\newblock {\em Journal of economic theory}, 59(1):33--49, 1993.

\bibitem{GonZhaLiuTaoSch16}
Mingming Gong, Kun Zhang, Tongliang Liu, Dacheng Tao, Clark Glymour, and Bernhard Sch{\"o}lkopf.
\newblock Domain adaptation with conditional transferable components.
\newblock In {\em Proceedings of the 33nd International Conference on Machine Learning (ICML 2016)}, volume~48 of {\em {JMLR} Workshop and Conference Proceedings}, pages 2839--2848, 2016.

\bibitem{gong2017belief}
Wenjuan Gong and Fabio Cuzzolin.
\newblock A belief-theoretical approach to example-based pose estimation.
\newblock {\em IEEE Transactions on Fuzzy Systems}, 26(2):598--611, 2017.

\bibitem{Goubault-Larrecq2007}
Jean Goubault-Larrecq.
\newblock Continuous capacities on continuous state spaces.
\newblock In Lars Arge, Christian Cachin, Tomasz Jurdzi{\'{n}}ski, and Andrzej Tarlecki, editors, {\em Automata, Languages and Programming: 34th International Colloquium, ICALP 2007, Wroc{\l}aw, Poland, July 9-13, 2007. Proceedings}, pages 764--776. Springer, Berlin, Heidelberg, 2007.

\bibitem{Graf1980}
Siegfried Graf.
\newblock {A Radon-Nikodym theorem for capacities}.
\newblock {\em Journal für die reine und angewandte Mathematik}, 320:192--214, 1980.

\bibitem{Grassmann1844}
Hermann Grassmann.
\newblock {\em { Die Ausdehnungslehre von 1844 oder die lineale Ausdehnungslehre }}.
\newblock Wigand, 1878.

\bibitem{li2006strong}
Li~Guan and Shou mei Li.
\newblock A strong law of large numbers for set-valued random variables in rademacher type {P} {B}anach space.
\newblock In {\em Proceedings of the International Conference on Machine Learning and Cybernetics}, pages 1768--1773, 2006.

\bibitem{ha98geometric}
Vu~Ha, AnHai Doan, Van~H. Vu, and Peter Haddawy.
\newblock Geometric foundations for interval-based probabilities.
\newblock {\em Annals of Mathematics and Artical Inteligence}, 24(1-4):1--21, 1998.

\bibitem{HANSON1981545}
Morgan~A. Hanson.
\newblock {On sufficiency of the Kuhn-Tucker conditions}.
\newblock {\em Journal of Mathematical Analysis and Applications}, 80(2):545--550, 1981.

\bibitem{Harding1997}
John Harding, Massimo Marinacci, Nhu~T. Nguyen, and Tonghui Wang.
\newblock Local {R}adon-{N}ikodym derivatives of set functions.
\newblock {\em International Journal of Uncertainty, Fuzziness and Knowledge-Based Systems}, 5(3):379--394, 1997.

\bibitem{Holmes1975}
Richard~B. Holmes.
\newblock {Principles of Banach Spaces}.
\newblock In {\em Geometric Functional Analysis and its Applications}, pages 119--201. Springer New York, New York, NY, 1975.

\bibitem{holmes2012geometric}
Richard~B. Holmes.
\newblock {\em Geometric functional analysis and its applications}, volume~24.
\newblock Springer Science and Business Media, 2012.

\bibitem{Hug2016}
Daniel Hug and Matthias Reitzner.
\newblock Introduction to stochastic geometry.
\newblock In Giovanni Peccati and Matthias Reitzner, editors, {\em Stochastic Analysis for Poisson Point Processes: Malliavin Calculus, Wiener-It{\^o} Chaos Expansions and Stochastic Geometry}, pages 145--184. Springer International Publishing, 2016.

\bibitem{hullermeier2021aleatoric}
Eyke H{\"u}llermeier and Willem Waegeman.
\newblock Aleatoric and epistemic uncertainty in machine learning: An introduction to concepts and methods.
\newblock {\em Machine Learning}, 110(3):457--506, 2021.

\bibitem{itoh95new}
Makoto Itoh and Toshiyuki Inagaki.
\newblock A new conditioning rule for belief updating in the {D}empster--{S}hafer theory of evidence.
\newblock {\em Transactions of the Society of Instrument and Control Engineers}, 31(12):2011--2017, 1995.

\bibitem{Jaffray92}
Jean-Yves Jaffray.
\newblock Bayesian updating and belief functions.
\newblock {\em IEEE Transactions on Systems, Man and Cybernetics}, 22(5):1144--1152, 1992.

\bibitem{Jirousek2016entropy}
Radim Jirousek and Prakash~P. Shenoy.
\newblock A new definition of entropy of belief functions in the {D}empster--{S}hafer theory.
\newblock Working Paper 330, Kansas University School of Business, 2016.

\bibitem{Karush1939}
William Karush.
\newblock {\em Minima of Functions of Several Variables with Inequalities as Side Constraints}.
\newblock {MSc} dissertation, Department of Mathematics, University of Chicago, Chicago, Illinois, 1939.

\bibitem{katzfuss2017bayesian}
Matthias Katzfuss, Dorit Hammerling, and Richard~L Smith.
\newblock A bayesian hierarchical model for climate change detection and attribution.
\newblock {\em Geophysical Research Letters}, 44(11):5720--5728, 2017.

\bibitem{Kerkvliet2017}
Timber Kerkvliet.
\newblock {\em Uniform probability measures and epistemic probability}.
\newblock {PhD} dissertation, Vrije Universiteit, 2017.

\bibitem{Khachiyan2009}
Leonid Khachiyan.
\newblock {Fourier--Motzkin elimination method}.
\newblock In Christodoulos~A. Floudas and Panos~M. Pardalos, editors, {\em Encyclopedia of Optimization}, pages 1074--1077. Springer US, Boston, MA, 2009.

\bibitem{DBLP:journals/corr/KhaniRL16}
Fereshte Khani, Martin Rinard, and Percy Liang.
\newblock Unanimous prediction for 100\% precision with application to learning semantic mappings.
\newblock In {\em Proceedings of the 54th Annual Meeting of the Association for Computational Linguistics (Volume 1: Long Papers)}, pages 952--962, 2016.

\bibitem{King2001}
Gary King and Langche Zeng.
\newblock Logistic regression in rare events data.
\newblock {\em Political Analysis}, 9(2):137--163, 2001.

\bibitem{rota97book}
Daniel~A. Klain and Gian-Carlo Rota.
\newblock {\em Introduction to Geometric Probability}.
\newblock Cambridge University Press, 1997.

\bibitem{kohlas88b}
J{\"u}rg Kohlas.
\newblock Conditional belief structures.
\newblock {\em Probability in Engineering and Information Science}, 2(4):415--433, 1988.

\bibitem{Kramosil02probabilistic-analysis}
Ivan Kramosil.
\newblock Probabilistic analysis of belief functions.
\newblock In {\em ISFR International Series on Systems Science and Engineering}, volume~16. Kluwer Academic / Plenum Publishers, 2002.

\bibitem{kriegler2005imprecise}
Elmar Kriegler.
\newblock {\em Imprecise probability analysis for integrated assessment of climate change}.
\newblock PhD thesis, Universit{\"a}t Potsdam, 2005.

\bibitem{kriegler2005utilizing}
Elmar Kriegler and Hermann Held.
\newblock Utilizing belief functions for the estimation of future climate change.
\newblock {\em International journal of approximate reasoning}, 39(2-3):185--209, 2005.

\bibitem{kruse91reasoning}
Rudolf Kruse, Detlef Nauck, and Frank Klawonn.
\newblock Reasoning with mass distributions.
\newblock In B.~D. D'Ambrosio, Ph. Smets, and P.~P. Bonissone, editors, {\em Uncertainty in Artificial Intelligence}, pages 182--187. Morgan Kaufmann, 1991.

\bibitem{kruse91tool}
Rudolf Kruse, Erhard Schwecke, and Frank Klawonn.
\newblock On a tool for reasoning with mass distribution.
\newblock In {\em Proceedings of the 12th International Joint Conference on Artificial Intelligence (IJCAI'91)}, volume~2, pages 1190--1195, 1991.

\bibitem{lam2007geometric}
Yeh Lam.
\newblock {\em The geometric process and its applications}.
\newblock World Scientific, 2007.

\bibitem{Lefevre2002149}
Eric Lef\'evre, Olivier Colot, and Patrick Vannoorenberghe.
\newblock Belief function combination and conflict management.
\newblock {\em Information Fusion}, 3(2):149--162, 2002.

\bibitem{lehrer05updating}
Ehud Lehrer.
\newblock Updating non-additive probabilities - a geometric approach.
\newblock {\em Games and Economic Behavior}, 50:42--57, 2005.

\bibitem{li2015strong}
Guan Li.
\newblock A strong law of large numbers for set-valued random variables in g$\alpha$ space.
\newblock {\em Journal of Applied Mathematics and Physics}, 3(7):797--801, 2015.

\bibitem{liu99local}
Liping Liu.
\newblock {Local computation of Gaussian belief functions}.
\newblock {\em International Journal of Approximate Reasoning}, 22(3):217--248, 1999.

\bibitem{DBLP:journals/corr/Long0J16a}
Mingsheng Long, Jianmin Wang, and Michael~I. Jordan.
\newblock Deep transfer learning with joint adaptation networks.
\newblock In {\em Proceedings of the International Conference on Machine Learning (ICML 2017)}, pages 2208--2217, 2017.

\bibitem{76dbd9a53fe5495695d6fc9e7d8fcd3f}
Jianbing Ma, Weiru Liu, Didier Dubois, and Henri Prade.
\newblock Bridging {J}effrey's rule, {AGM} revision and {D}empster conditioning in the theory of evidence.
\newblock {\em International Journal on Artificial Intelligence Tools}, 20(4):691--720, 2011.

\bibitem{10.2307/3481723}
Fabio Maccheroni and Massimo Marinacci.
\newblock A strong law of large numbers for capacities.
\newblock {\em The Annals of Probability}, 33(3):1171--1178, 2005.

\bibitem{manchingal2022epistemic}
Shireen~Kudukkil Manchingal and Fabio Cuzzolin.
\newblock Epistemic deep learning.
\newblock {\em arXiv preprint arXiv:2206.07609}, 2022.

\bibitem{manchingal2023random}
Shireen~Kudukkil Manchingal, Muhammad Mubashar, Kaizheng Wang, Keivan Shariatmadar, and Fabio Cuzzolin.
\newblock Random-set convolutional neural network (rs-cnn) for epistemic deep learning.
\newblock {\em arXiv preprint arXiv:2307.05772}, 2023.

\bibitem{Marinacci1999limit}
Massimo Marinacci.
\newblock Limit laws for non-additive probabilities and their frequentist interpretation.
\newblock {\em Journal of Economic Theory}, 84(2):145--195, 1999.

\bibitem{martin2021imprecise}
Ryan Martin.
\newblock An imprecise-probabilistic characterization of frequentist statistical inference.
\newblock {\em arXiv preprint arXiv:2112.10904}, 2021.

\bibitem{martin2010}
Ryan Martin, Jianchun Zhang, and Chuanhai Liu.
\newblock Dempster-{S}hafer theory and statistical inference with weak beliefs.
\newblock {\em Statistical Science}, 25(1):72--87, 2010.

\bibitem{Masson20081384}
Marie-H{\'e}l{\`e}ne Masson and Thierry Den{\oe}ux.
\newblock {ECM: An evidential version of the fuzzy c-means algorithm}.
\newblock {\em Pattern Recognition}, 41(4):1384--1397, 2008.

\bibitem{Matheron75}
Georges Matheron.
\newblock {\em Random sets and integral geometry}.
\newblock Wiley Series in Probability and Mathematical Statistics, New York, 1975.

\bibitem{min2007probabilistic}
Seung-Ki Min, Daniel Simonis, and Andreas Hense.
\newblock Probabilistic climate change predictions applying bayesian model averaging.
\newblock {\em Philosophical Transactions of the Royal Society A: Mathematical, Physical and Engineering Sciences}, 365(1857):2103--2116, 2007.

\bibitem{miranda07marginal}
Enrique Miranda and Gert de~Cooman.
\newblock Marginal extension in the theory of coherent lower previsions.
\newblock {\em International Journal of Approximate Reasoning}, 46(1):188--225, 2007.

\bibitem{molchanov1988uniform}
Ilya Molchanov.
\newblock Uniform laws of large numbers for empirical associated functionals of random closed sets.
\newblock {\em Theory of Probability and Its Applications}, 32(3):556--559, 1988.

\bibitem{molchanov1997statistical}
Ilya Molchanov.
\newblock Statistical problems for random sets.
\newblock In {\em Random Sets}, pages 27--45. Springer, 1997.

\bibitem{Molchanov05}
Ilya Molchanov.
\newblock {\em Theory of Random Sets}.
\newblock Springer-Verlag, 2005.

\bibitem{Mundici2016}
Daniele Mundici.
\newblock A geometric approach to {MV}-algebras.
\newblock In Susanne Saminger-Platz and Radko Mesiar, editors, {\em On Logical, Algebraic, and Probabilistic Aspects of Fuzzy Set Theory}, pages 57--70. Springer International Publishing, 2016.

\bibitem{nguyen21belief}
Hung~T. Nguyen.
\newblock On belief functions and random sets.
\newblock In Thierry Den{\oe}ux and Marie-H{\'e}l{\`e}ne Masson, editors, {\em Advances in Intelligent and Soft Computing}, volume 164, pages 1--19. Springer, Berlin Heidelberg, 2012.

\bibitem{nigam1999using}
Kamal Nigam, John Lafferty, and Andrew McCallum.
\newblock Using maximum entropy for text classification.
\newblock In {\em IJCAI-99 workshop on machine learning for information filtering}, volume~1, pages 61--67. Stockholom, Sweden, 1999.

\bibitem{Oxley92}
James~G. Oxley.
\newblock {\em Matroid theory}.
\newblock Oxford University Press, 1992.

\bibitem{pan2018new}
Lipeng Pan and Yong Deng.
\newblock A new belief entropy to measure uncertainty of basic probability assignments based on belief function and plausibility function.
\newblock {\em Entropy}, 20(11):842, 2018.

\bibitem{Philippe99decision}
Fabrice Philippe, Gabriel Debs, and Jean-Yves Jaffray.
\newblock Decision making with monotone lower probabilities of infinite order.
\newblock {\em Mathematics of Operations Research}, 24(3):767--784, 1999.

\bibitem{RePEc:hal:wpaper:hal-00441923}
Yann R\'ebill\'e.
\newblock {A Radon-Nikodym derivative for almost subadditive set functions}.
\newblock Working Paper hal-00441923, HAL, 2009.

\bibitem{REBILLE2009872}
Yann R\'ebill\'e.
\newblock Law of large numbers for non-additive measures.
\newblock {\em Journal of Mathematical Analysis and Applications}, 352(2):872--879, 2009.

\bibitem{doi:10.1142/S0218488509006212}
Yann R\'ebill\'e.
\newblock Laws of large numbers for continuous belief measures on compact spaces.
\newblock {\em International Journal of Uncertainty, Fuzziness and Knowledge-Based Systems}, 17(5):685--704, 2009.

\bibitem{Reineking2014}
Thomas Reineking.
\newblock {\em Belief Functions: Theory and Algorithms}.
\newblock {PhD} dissertation, Universit{\"{a}}t Bremen, 2014.

\bibitem{ristic04ipmu}
Branko Ristic and Philippe Smets.
\newblock Belief function theory on the continuous space with an application to model based classification.
\newblock In {\em Modern Information Processing}, pages 11--24, 2006.

\bibitem{Rogova1994777}
Galina Rogova.
\newblock Combining the results of several neural network classifiers.
\newblock {\em Neural Networks}, 7(5):777--781, 1994.

\bibitem{Rougier2007}
Jonathan Rougier.
\newblock Probabilistic inference for future climate using an ensemble of climate model evaluations.
\newblock {\em Climatic Change}, 81(3):247--264, April 2007.

\bibitem{santalo1950integral}
LA~Santalo.
\newblock Integral geometry in projective and affine spaces.
\newblock {\em Annals of Mathematics}, pages 739--755, 1950.

\bibitem{sensoy2018evidential}
Murat Sensoy, Lance Kaplan, and Melih Kandemir.
\newblock Evidential deep learning to quantify classification uncertainty.
\newblock {\em arXiv:1806.01768}, 2018.

\bibitem{Shafer76}
Glenn Shafer.
\newblock {\em A Mathematical Theory of Evidence}.
\newblock Princeton University Press, 1976.

\bibitem{shafer81jeffrey}
Glenn Shafer.
\newblock Jeffrey's rule of conditioning.
\newblock {\em Philosophy of Sciences}, 48:337--362, 1981.

\bibitem{Shafer07game}
Glenn Shafer.
\newblock Game-theoretic probability: Theory and applications.
\newblock In {\em Proceedings of the Fifth International Symposium on Imprecise Probabilities and Their Applications (ISIPTA'07)}, 2007.

\bibitem{Shafer81}
Glenn Shafer.
\newblock Constructive probability.
\newblock {\em Synthese}, 48(1):1--60, July 1981.

\bibitem{Shafer2008}
Glenn Shafer and Amos Tversky.
\newblock Languages and designs for probability judgment.
\newblock In Roland~R. Yager and Liping Liu, editors, {\em Classic Works of the Dempster-Shafer Theory of Belief Functions}, pages 345--374. Springer, Berlin, Heidelberg, 2008.

\bibitem{shafer01book}
Glenn Shafer and Vladimir Vovk.
\newblock {\em Probability and Finance: It's Only a Game!}
\newblock Wiley, New York, 2001.

\bibitem{shariatmadar2022introduction}
Keivan Shariatmadar, Kaizheng Wang, Calvin~R Hubbard, Hans Hallez, and David Moens.
\newblock An introduction to optimization under uncertainty--a short survey.
\newblock {\em arXiv preprint arXiv:2212.00862}, 2022.

\bibitem{DBLP:journals/corr/ShenTSO13}
Yun Shen, Michael~J. Tobia, Tobias Sommer, and Klaus Obermayer.
\newblock Risk-sensitive reinforcement learning.
\newblock {\em Machine learning}, 49(2-3):267--290, 2002.

\bibitem{4306979}
Prakash~P. Shenoy and Glenn Shafer.
\newblock Propagating belief functions with local computations.
\newblock {\em IEEE Expert}, 1(3):43--52, 1986.

\bibitem{Shi2015clt}
Xiaomin Shi.
\newblock Central limit theorems for bounded random variables under belief measures.
\newblock ArXiv preprint arxiv:1501.00771v1, 2015.

\bibitem{smets81degree}
Philippe Smets.
\newblock The degree of belief in a fuzzy event.
\newblock {\em Information Sciences}, 25:1--19, 1981.

\bibitem{Smets:1990:CPP:647232.719592}
Philippe Smets.
\newblock Constructing the pignistic probability function in a context of uncertainty.
\newblock In {\em Proceedings of the Fifth Annual Conference on Uncertainty in Artificial Intelligence (UAI '89)}, pages 29--40. North-Holland, 1990.

\bibitem{ubf}
Philippe Smets.
\newblock The nature of the unnormalized beliefs encountered in the transferable belief model.
\newblock In {\em Proceedings of the 8th Annual Conference on Uncertainty in Artificial Intelligence (UAI-92)}, pages 292--29, San Mateo, CA, 1992. Morgan Kaufmann.

\bibitem{smets93belief}
Philippe Smets.
\newblock Belief functions : the disjunctive rule of combination and the generalized {B}ayesian theorem.
\newblock {\em International Journal of Approximate Reasoning}, 9(1):1--35, 1993.

\bibitem{smets93jeffrey}
Philippe Smets.
\newblock Jeffrey's rule of conditioning generalized to belief functions.
\newblock In {\em Proceedings of the Ninth International Conference on Uncertainty in Artificial Intelligence (UAI'93)}, pages 500--505. Morgan Kaufmann, 1993.

\bibitem{Smets:1999:PUB:2073796.2073865}
Philippe Smets.
\newblock Practical uses of belief functions.
\newblock In {\em Proceedings of the Fifteenth Conference on Uncertainty in Artificial Intelligence (UAI'99)}, pages 612--621. Morgan Kaufmann, 1999.

\bibitem{smets93deductibility}
Philippe Smets.
\newblock Probability of deductibility and belief functions.
\newblock In M.~Clark, R.~Kruse, and Seraf\'in Moral, editors, {\em Proceedings of the European Conference on Symbolic and Quantitative Approaches to Reasoning and Uncertainty (ECSQARU'93)}, pages 332--340, November 1993.

\bibitem{smith2009bayesian}
Richard~L Smith, Claudia Tebaldi, Doug Nychka, and Linda~O Mearns.
\newblock Bayesian modeling of uncertainty in ensembles of climate models.
\newblock {\em Journal of the American Statistical Association}, 104(485):97--116, 2009.

\bibitem{spies94conditional}
Marcus Spies.
\newblock Conditional events, conditioning, and random sets.
\newblock {\em IEEE Transactions on Systems, Man, and Cybernetics}, 24(12):1755--1763, 1994.

\bibitem{Strat84}
Thomas~M. Strat.
\newblock Continuous belief functions for evidential reasoning.
\newblock In {\em Proceedings of the Fourth National Conference on Artificial Intelligence (AAAI-84)}, pages 308--313, August 1984.

\bibitem{suppes1977}
Patrick Suppes and Mario Zanotti.
\newblock On using random relations to generate upper and lower probabilities.
\newblock {\em Synthese}, 36(4):427--440, 1977.

\bibitem{Terán2015185}
Pedro Ter\'an.
\newblock {Counterexamples to a Central Limit Theorem and a Weak Law of Large Numbers for capacities}.
\newblock {\em Statistics and Probability Letters}, 96:185--189, 2015.

\bibitem{Terán2014lln}
Pedro Ter\'an.
\newblock Laws of large numbers without additivity.
\newblock {\em Transactions of the American Mathematical Society}, 366(10):5431--5451, October 2014.

\bibitem{tong2021evidential}
Zheng Tong, Philippe Xu, and Thierry Denoeux.
\newblock An evidential classifier based on dempster-shafer theory and deep learning.
\newblock {\em Neurocomputing}, 450:275--293, 2021.

\bibitem{Vannobel2012}
Jean-Marc Vannobel.
\newblock Continuous belief functions: {F}ocal intervals properties.
\newblock In Thierry Den{\oe}ux and Marie-H{\'e}l{\`e}ne Masson, editors, {\em Belief Functions: Theory and Applications: Proceedings of the 2nd International Conference on Belief Functions, Compi{\`e}gne, France 9-11 May 2012}, pages 93--100. Springer, Berlin Heidelberg, 2012.

\bibitem{vapnik1998statistical}
Vladimir~N. Vapnik.
\newblock {\em Statistical learning theory}, volume~1.
\newblock Wiley New York, 1998.

\bibitem{vapnik1999overview}
Vladimir~N. Vapnik.
\newblock An overview of statistical learning theory.
\newblock {\em IEEE Transactions on Neural Networks}, 10(5):988--999, 1999.

\bibitem{vapnik2013nature}
Vladimir~N. Vapnik.
\newblock {\em The nature of statistical learning theory}.
\newblock Springer science \& business media, 2013.

\bibitem{Vershynin2009}
Roman Vershynin.
\newblock Lectures in geometric functional analysis, 2011.

\bibitem{voorbraak89efficient}
F.~Voorbraak.
\newblock A computationally efficient approximation of {D}empster--{S}hafer theory.
\newblock {\em International Journal on Man-Machine Studies}, 30(5):525--536, 1989.

\bibitem{walley91book}
Peter Walley.
\newblock {\em Statistical Reasoning with Imprecise Probabilities}.
\newblock Chapman and Hall, New York, 1991.

\bibitem{walley00towards}
Peter Walley.
\newblock Towards a unified theory of imprecise probability.
\newblock {\em International Journal of Approximate Reasoning}, 24(2-3):125--148, 2000.

\bibitem{wang91geometrical}
Chua-Chin Wang and Hon-Son Don.
\newblock A geometrical approach to evidential reasoning.
\newblock In {\em Proceedings of the IEEE International Conference on Systems, Man, and Cybernetics (SMC'91)}, volume~3, pages 1847--1852, 1991.

\bibitem{wang92continuous}
Chua-Chin Wang and Hon-Son Don.
\newblock A continuous belief function model for evidential reasoning.
\newblock In J.~Glasgow and R.~F. Hadley, editors, {\em Proceedings of the Ninth Biennial Conference of the Canadian Society for Computational Studies of Intelligence}, pages 113--120, May 1992.

\bibitem{wang08-reliable}
Ping Wang.
\newblock The reliable combination rule of evidence in {D}empster--{S}hafer theory.
\newblock In {\em Proceedings of the 2008 Congress on Image and Signal Processing (CISP '08)}, volume~2, pages 166--170, May 2008.

\bibitem{wilson1991monte}
Nic Wilson.
\newblock A monte-carlo algorithm for dempster-shafer belief.
\newblock In {\em Uncertainty Proceedings 1991}, pages 414--417. Elsevier, 1991.

\bibitem{wright1995logistic}
Raymond~E Wright.
\newblock Logistic regression.
\newblock 1995.

\bibitem{xu96reasoning}
Hong Xu and Philippe Smets.
\newblock Reasoning in evidential networks with conditional belief functions.
\newblock {\em International Journal of Approximate Reasoning}, 14(2-3):155--185, 1996.

\bibitem{155943}
Lei Xu, Adam Krzyzak, and Ching~Y. Suen.
\newblock Methods of combining multiple classifiers and their applications to handwriting recognition.
\newblock {\em IEEE Transactions on Systems, Man, and Cybernetics}, 22(3):418--435, 1992.

\bibitem{Xu:2008:SPA:1622673.1622687}
Lin Xu, Frank Hutter, Holger~H. Hoos, and Kevin Leyton-Brown.
\newblock {SATzilla: Portfolio-based Algorithm Selection for SAT}.
\newblock {\em Journal of Artificial Intelligence Research}, 32(1):565--606, June 2008.

\bibitem{Yager:2010:CWD:1951793}
Ronald~R. Yager and Liping Liu.
\newblock {\em Classic Works of the {D}empster--{S}hafer Theory of Belief Functions}.
\newblock Springer Publishing Company, 1st edition, 2010.

\bibitem{YAMADA20081689}
Koichi Yamada.
\newblock A new combination of evidence based on compromise.
\newblock {\em Fuzzy Sets and Systems}, 159(13):1689--1708, 2008.

\bibitem{yokonuma1992tensor}
Takeo Yokonuma.
\newblock {\em Tensor spaces and exterior algebra}.
\newblock Number 108. American Mathematical Soc., 1992.

\bibitem{yu94conditional}
Chunhai Yu and Fahard Arasta.
\newblock On conditional belief functions.
\newblock {\em International Journal of Approxiomate Reasoning}, 10(2):155--172, 1994.

\bibitem{Zhang11weak}
Jianchun Zhang and Chuanhai Liu.
\newblock Dempster--{S}hafer inference with weak beliefs.
\newblock {\em Statistica Sinica}, 21:475--494, 2011.

\bibitem{Zhou2012}
Chunlai Zhou.
\newblock Belief functions on distributive lattices.
\newblock In {\em Proceedings of the National Conference on Artificial Intelligence (AAAI 2012)}, pages 1968--1974, 2012.

\bibitem{zhou2017total}
Chunlai Zhou and Fabio Cuzzolin.
\newblock The total belief theorem.
\newblock In {\em UAI}, 2017.

\bibitem{ZhouWQ14}
Chunlai Zhou, Mingyue Wang, and Biao Qin.
\newblock Belief-kinematics {J}effrey-s rules in the theory of evidence.
\newblock In {\em Proceedings of the Thirtieth Conference on Uncertainty in Artificial Intelligence (UAI 2014)}, pages 917--926, July 2014.

\bibitem{Zhou:2012:LWC:2999325.2999380}
Dengyong Zhou, John~C. Platt, Sumit Basu, and Yi~Mao.
\newblock Learning from the wisdom of crowds by minimax entropy.
\newblock In {\em Proceedings of the 25th International Conference on Neural Information Processing Systems (NIPS'12)}, pages 2195--2203, 2012.

\end{thebibliography}

\end{document}